\documentclass[notitlepage,11pt,reqno]{amsart}
\usepackage{cite,bm,nicefrac,upgreek}
\usepackage[mathscr]{eucal}
\allowdisplaybreaks
\usepackage{color}

\newcommand{\ftn}[1]{\relax}
\renewcommand{\ftn}[1]{\footnote{#1}}
\newcommand{\ftnanup}[1]{\relax}
\renewcommand{\ftnanup}[1]{\footnote{#1}}
\newcommand{\ftnb}[1]{\relax}
\renewcommand{\ftnb}[1]{\footnote{\color{dblue}#1}}
\newcommand{\ftnr}[1]{\relax}
\renewcommand{\ftnr}[1]{\footnote{\bfseries\color{dmagenta}#1}}
\definecolor{dmagenta}{rgb}{.4,.1,.5}
\definecolor{dred}{rgb}{.6,.0,.0}
\definecolor{dgreen}{rgb}{.0,.5,.0}
\definecolor{dblue}{rgb}{.0,.0,.7}
\definecolor{violet}{rgb}{.3,.0,.9}
\definecolor{orange}{cmyk}{0,.5,.1,.0}
\definecolor{dcyan}{cmyk}{.5,.0,.0,.0}

\usepackage[margin=1in]{geometry}

\theoremstyle{plain}
\newtheorem{lemma}{Lemma}[section]
\newtheorem{theorem}{Theorem}[section]
\newtheorem{proposition}{Proposition}[section]
\newtheorem{corollary}{Corollary}[section]

\theoremstyle{definition}
\newtheorem{definition}{Definition}[section]

\newtheorem{condition}{Condition}[section]

\theoremstyle{remark}
\newtheorem{remark}{Remark}[section]
\newtheorem{example}{Example}[section]

\numberwithin{equation}{section}

\newcommand{\tpi}{{\widetilde{\uppi}}}

\DeclareMathOperator{\Exp}{\mathbb{E}}
\DeclareMathOperator{\Prob}{\mathbb{P}}
\newcommand{\D}{\mathrm{d}}
\newcommand{\E}{\mathrm{e}}
\newcommand{\RR}{\mathbb{R}}
\newcommand{\NN}{\mathbb{N}}

\newcommand{\fL}{\mathfrak{L}_{\alpha}}
\newcommand{\fLsym}{\mathfrak{L}^{\mathsf{sym}}_{\alpha}}

\newcommand{\R}{\mathbb{R}}
\newcommand{\I}{\mathcal{I}}
\newcommand{\dd}{\mathfrak{d}}
\newcommand{\cZ}{\mathcal{Z}}
\newcommand{\scrP}{\mathscr{P}}
\newcommand{\calA}{{\mathcal{A}}}
\newcommand{\calB}{{\mathcal{B}}}
\newcommand{\calD}{{\mathcal{D}}}

\newcommand{\calF}{{\mathcal{F}}}

\newcommand{\calH}{{\mathcal{H}}}
\newcommand{\cK}{{\mathcal{K}}}

\newcommand{\calP}{{\mathcal{P}}}
\newcommand{\calS}{{\mathcal{S}}}
\newcommand{\calV}{{\mathcal{V}}}

\newcommand{\abs}[1]{\lvert#1\rvert}
\newcommand{\norm}[1]{\lVert#1\rVert}

\newcommand{\dabs}[1]
{\bm\lvert\kern-0.50ex\bm\lvert\kern0.1ex #1
\kern0.1ex\bm\rvert\kern-0.50ex\bm\rvert}
\newcommand{\bdabs}[1]
{\bigl\lvert\kern-0.72ex\bigl\rvert\kern-0.55ex
\bigl\lvert\kern-0.72ex\bigl\rvert\kern0.1ex #1
\kern0.1ex\bigl\lvert\kern-0.72ex\bigl\rvert\kern-0.55ex
\bigl\lvert\kern-0.72ex\bigl\rvert}

\newcommand{\dbrac}[1]
{[\kern-0.45ex[\kern0.1ex #1\kern0.1ex]\kern-0.45ex]}
\newcommand{\bdbrac}[1]
{\bigl[\kern-0.75ex\bigl[\kern0.1ex #1\kern0.1ex\bigr]\kern-0.75ex\bigr]}

\newcommand{\grad}{\nabla}
\newcommand{\Rd}{\mathbb{R}^{d}}
\newcommand{\Rds}{\mathbb{R}^{d}}
\newcommand{\In}{\bm1_{\{\abs{z}\le 1\}}}

\newcommand{\babs}[1]{\bigl\lvert#1\bigr\rvert}
\newcommand{\Babs}[1]{\Bigl\lvert#1\Bigr\rvert}
\newcommand{\babss}[1]{\biggl\lvert#1\biggr\rvert}
\newcommand{\bnorm}[1]{\bigl\lVert#1\bigr\rVert}

\newcommand{\bnormm}[1]{\biggl\lVert#1\biggr\rVert}
\newcommand{\df}{:=}
\DeclareMathOperator{\dist}{dist}
\DeclareMathOperator{\diam}{diam}
\let\oldtocsection=\tocsection
\let\oldtocsubsection=\tocsubsection
\let\oldtocsubsubsection=\tocsubsubsection
\renewcommand{\tocsection}[2]{\hspace{0em}\oldtocsection{#1}{#2}}
\renewcommand{\tocsubsection}[2]{\hspace{1em}\oldtocsubsection{#1}{#2}}
\renewcommand{\tocsubsubsection}[2]{\hspace{2em}\oldtocsubsubsection{#1}{#2}}
\begin{document}
\title[On a class of stochastic differential equations
with jumps and its properties]{On a class of stochastic differential equations\\
with jumps and its properties}

\author{Ari Arapostathis}
\address{Department of Electrical and Computer Engineering,
The University of Texas at Austin, 
1 University Station, Austin, TX 78712}
\curraddr{1616 Guadalupe St.\\
UTA 7.508\\
Austin, TX~~7801}
\email{ari@ece.utexas.edu}
\thanks{Corresponding Author:  Ari Arapostathis, 1616 Guadalupe St.,
UTA~7.502, Austin, TX~~78701, ari@ece.utexas.edu, (512) 471-3265.}
\author{Anup Biswas}
\address{Department of Electrical and Computer Engineering,
The University of Texas at Austin, 
1 University Station, Austin, TX 78712}
\email{anupbiswas@utexas.edu}
\author{Luis Caffarelli}
\address{Department of Mathematics,
The University of Texas at Austin, 
1 University Station, Austin, TX 78712}
\email{caffarel@math.utexas.edu}

\date{\today}

\begin{abstract}
We study stochastic differential equations with jumps with no diffusion part.
We provide some basic stochastic characterizations of solutions of the
corresponding non-local partial differential equations and prove the Harnack
inequality for a class of these operators.
We also establish key connections between the recurrence properties of
these jump processes and the non-local partial differential operator.
One of the key results is the regularity of solutions of the Dirichlet problem
for a class of operators with weakly H\"older continuous kernels.
\end{abstract}



\maketitle

\hrule
\tableofcontents
\vspace{-0.7cm}
\hrule

\bigskip
\section{Introduction}

Stochastic differential equations (SDEs) with jumps have received
wide attention in stochastic analysis as well as
in the theory of differential equations.
Unlike continuous diffusion processes, SDEs with jumps have long range
interactions and therefore the generators of such processes are non-local
in nature.
These processes arise in various applications, for instance, in mathematical
finance and control \cite{cont-tankov, soner}
and image processing \cite{gilboa-osher}.
There have been various studies on such processes from a stochastic analysis
viewpoint concentrating on existence, uniqueness, and stability
properties of the solution of the stochastic differential equation
\cite{abels-kassmann,bass04,chen-kim-song,chen-wang,kurtz, komatsu},
as well as from a differential equation viewpoint focusing on the existence and
regularity of viscosity solutions
\cite{barles-chas-imbert,barles-chas-imbert-11,caffarelli-silvestre-regu}.
One of our objectives in this paper is to establish stochastic representations
of solutions of SDEs with jumps via the associated
integro-differential operator. 

Let us consider a Markov process $X$ in $\Rd$ with generator $\calA$.
Let $D$ be a smooth bounded domain in $\Rd$.
We denote the first exit time of the process $X$ from $D$
by $\tau(D)=\inf\{t\ge 0: X_{t}\notin D\}$.
One can formally say that
\begin{equation}\label{ee1.1}
u(x)\;\df\;\Exp_{x}\biggl[\int_{0}^{\tau(D)}f(X_{s})\,\D{s}\biggr]
\end{equation}
satisfies the following equation
\begin{equation}\label{ee1.2}
\calA u\;=\;-f \quad \text{in}~D, \quad u\;=\;0\quad \text{in}~D^{c}\,,
\end{equation}
where $\Exp_{x}$ denotes the expectation operator on the canonical
space of the process
starting at $x$ when $t=0$.
An important question is when can we actually identify the solution of
\eqref{ee1.2} as the right hand side of \eqref{ee1.1}.
When $\calA=\Delta + b$, i.e., $X$ is a drifted Brownian notion, one can use
the regularity of the solution
and It\^{o}'s formula to establish \eqref{ee1.1}.
Clearly then, one standard method to obtain a representation of the
mean first exit time from $D$ is to find
a classical solution of \eqref{ee1.2} for non-local operators.
This is related to the work in \cite{bass-2009}
where estimates on classical solutions are obtained when $D=\Rd$.
The author in \cite{bass-2009} also raises questions
concerning the existence and regularity of solutions to the Dirichlet problem
for non-local operators.
We provide a partial answer to these questions in Theorem~\ref{T6.1}.

One of the main results of this paper is the existence of
a classical solution of \eqref{ee1.2} for a fairly general class of
non-local operators.
We focus on operators of the form
\begin{equation}\label{ee1.3}
\I f(x)\;=\; b(x)\cdot \grad f(x)
+\int_{\Rd}\dd f(x;z)\,\uppi(x,z)\,\D{z}\,,
\end{equation}
where
\begin{equation}\label{ee1.4}
\dd f(x;z)\;\df\; f(x+z)-f(x)-\In\grad{f}(x)\cdot z\,,
\end{equation}
with $\bm1_{A}$ denoting the indicator function of a set $A$.
The kernel $\uppi$ satisfies the usual integrability conditions.
When $\uppi(x,z) = \frac{k(x,z)}{\abs{z}^{d+\alpha}}$, with $\alpha\in (1,2)$,
and $b$, $k$ and $f$ are locally H\"older in $x$ with exponent $\beta$, and
$k(x,\cdot)-k(x,0)$ satisfies the integrability condition in \eqref{ee3.10},
we show in Theorem~\ref{T3.1}
that $u$ defined by \eqref{ee1.1} is the unique solution of \eqref{ee1.2} in
$C^{2s+\beta}_{\mathrm{loc}}(D)\cap C(\Rd)$.
This result can be extended to include non-zero boundary
conditions provided that the boundary data is regular enough.
The proof is based on various regularity results concerning the Dirichlet problem,
that are of independent interest and can be found in Section~\ref{S6}. 
For the case $k\equiv1$, with continuous $f$ and $g$,
we characterize the solution of
\begin{equation}\label{ee1.5}
\I u\;=\;-f \quad \text{in}~D, \quad u\;=\;g\quad \text{in}~D^{c}
\end{equation}
in the viscosity framework.
Theorem~\ref{T3.2}, which appears later in Section~\ref{sec-pde}, asserts that
\begin{equation*}
u(x)\;=\;\Exp_{x}\biggl[\int_{0}^{\tau(D)}f(X_{s})\,\D{s}
+g(X_{\tau(D)})\biggr]\,,
\quad x\in\Rd\,,
\end{equation*}
is the unique viscosity solution to \eqref{ee1.5}.
One of the hurdles in establishing this lies in
showing that $\Exp_{x}[\tau(\Bar{D})]=0$ whenever $x\in\partial D$.
When $X$ is a drifted Brownian motion, this can be easily deduced from the
fact that Brownian motion has infinitely many zeros in every finite interval.
But similar crossing properties are not known for $\alpha$-stable processes.
We also have to restrict ourselves to the regime $\alpha\in (1,2)$, so that the
jump process `dominates' the drift, and this allows us to
establish that $\Exp_{x}[\tau(\Bar{D})]=0$ whenever $x\in\partial D$.
The proof technique uses an estimate of the first exit time 
of an $\alpha$-stable process from a cone \cite{hernandez}.
These auxiliary results can be found in Section~\ref{sec-pde}.

Recall that a function $h$ is said to be \emph{harmonic} with respect to $X$ in $D$
if $h(X_{t\wedge\tau(D)})$ is a martingale.
One of the important properties of nonnegative harmonic functions for
nondegenerate continuous diffusions is the
Harnack inequality, which plays a crucial role in various regularity
and stability estimates.
The work in \cite{bass-levin} proves the Harnack inequality for a class of pure
jump processes, and this is further generalized in \cite{bass-kassmann} 
for non-symmetric kernels that may have variable order.
A parabolic Harnack inequality is obtained in \cite{Barlow-Bass-Chen-Kassmann}
for symmetric jump processes associated with the Dirichlet form with a
symmetric kernel.
In \cite{song-vondracek-04} sufficient 
conditions on Markov processes to satisfy the Harnack inequality are identified.
Let us also mention the work in \cite{ari-ghosh-04,foondun,song-vondracek-05}
where a Harnack
inequality is established for jump processes with a non-degenerate diffusion part.
Recently, \cite{kassmann-mimica} proves a Harnack type estimate for harmonic
functions that are not necessarily nonnegative in all of $\Rd$.

In this paper we prove a Harnack inequality for harmonic functions
relative to the operator $\I$ in \eqref{ee1.3} when $k$ and $b$ are locally
bounded and measurable, and either $k(x,z) = k(x,-z)$,
or $\abs{\uppi(x,z)-\uppi(x,0)}$ is a lower order kernel (Theorem~\ref{T4.1}).
The method of the proof is based on verifying the sufficient conditions in
\cite{song-vondracek-04}.
Later we use this Harnack estimate to obtain certain stability results
for the process.
Let us also mention that the estimates obtained in Section~\ref{sec-pde} and
Section~\ref{sec-harnack} may also be used to establish H\"{o}lder continuity for
harmonic functions by following a similar method as in
\cite{bass-kassmann-holder}.
However we don't pursue this here.

In Section~\ref{sec-stability} we discuss the ergodic properties
of the process such as positive recurrence, invariant probability measures, etc.
We provide a sufficient condition for positive recurrence and the existence
of an invariant probability measure.
This is done via imposing a \emph{Lyapunov stability} condition on the
generator.
Following Has$'$minski\u{\i}'s method, we establish the existence of a unique
invariant probability measure for a fairly large class of processes.
We also show that one may obtain a positive recurrent process by
using a non-symmetric kernel and no drift (see Theorem~\ref{T5.3}).
In this case, the non-symmetric part of the kernel plays the role of the drift.
Let us mention here that in \cite{wee} the author provides sufficient conditions
for stability for a class of jump diffusions and this is accomplished
by constructing suitable Lyapunov type functions.
However, the class of kernels considered in \cite{wee} satisfies a
different set of hypotheses than those assumed in this paper, and in a certain way
lies in the complement of the class of L\'evy kernels that we consider.
Stability of $1$-dimensional processes is discussed in
\cite{Wang-08} under the assumption of Lebesgue-irreducibility.
Lastly, we want to point out one of the interesting results of this paper,
which is the characterization of the mean hitting time of a bounded domain
as a viscosity solution of the exterior Dirichlet problem (Theorem~\ref{T5.4}).
This is established for the class of operators with weakly
H\"older continuous kernels in Definition~\ref{D-hkernel}.

The organization of the paper is as follows.
In Section~\ref{S1.1} we introduce the notation used in the paper.
In Section~\ref{S2} we introduce the model and assumptions.
Section~\ref{sec-pde} establishes stochastic representations
of viscosity solutions.
In Section~\ref{sec-harnack} we show the Harnack inequality.
Section~\ref{sec-stability} establishes the connections
between the recurrence properties
of the process and solutions of the non-local equations.
Finally, Section~\ref{S6} is devoted to the proof of the
regularity of solutions to the Dirichlet problem for weakly H\"older continuous
kernels.
These results are used in Section~\ref{sec-stability}.

\subsection{Notation}\label{S1.1}
The standard norm in the $d$-dimensional Euclidean space
$\Rd$ is denoted by $\abs{\,\cdot\,}$,
and we let $\Rd_{*}\df \Rd\setminus\{0\}$.
The set of non-negative real numbers is denoted by $\RR_{+}$,
$\NN$ stands for the set of natural numbers, and $\bm1_{A}$ denotes
the indicator function of a set $A$.
For vectors $a, b\in\Rd$, we
denote the scalar product by $a\cdot b$.
We denote the maximum (minimum) of two real numbers
$a$ and $b$ by $a\vee b$ ($a\wedge b$).
We let $a^{+}\df a\vee 0$ and $a^-\df (-a)\vee 0$.
By $\lfloor a\rfloor$ ($\lceil a\rceil$) we denote the largest (least) integer less
than (greater than) or equal to the real number $a$.
For $x\in\Rd$ and $r\ge 0$, we denote by $B_{r}(x)$ the open ball of radius
$r$ around $x$ in $\Rd$, while $B_{r}$ without an argument
denotes the ball of radius $r$ around the origin.
Also in the interest of simplifying the
notation we use $B\equiv B_{1}$, i.e., the unit ball centered at $0$.

Given a metric space $\calS$, we denote by $\calB(\calS)$ and
$B_{b}(\calS)$ the Borel $\sigma$-algebra of $\calS$
and the set of bounded Borel measurable functions on $\calS$, respectively.
The set of Borel probability measures on $\calS$ is denoted
by $\calP(\calS)$, $\parallel \cdot\parallel_{\mathrm{TV}}$ denotes the total
variation norm on $\calP(\calS)$, and
$\delta_{x}$ the Dirac mass at $x$.
For any function $f:\calS\to \Rd$ we define
$\norm{f}_{\infty}\df\sup_{x\in\calS}\,\abs{f(x)}$.

The closure and the boundary of a set $A\subset\Rd$ are denoted
by $\Bar{A}$ and $\partial{A}$, respectively, and
$\abs{A}$ denotes the Lebesgue measure of $A$.
We also define
\begin{equation*}
\tau(A)\;\df\;\inf\;\{s\ge 0 : X_{s}\notin A\}\,.
\end{equation*}
Therefore $\tau(A)$ denotes the first exit time of the process $X$ from $A$.
For $R>0$, we often use the abbreviated notation $\tau_{R}\df\tau(B_{R})$.

We introduce the following notation for spaces of real-valued functions on
a set $A\subset\Rd$.
The space $L^{p}(A)$, $p\in[1,\infty)$, stands for the Banach space
of (equivalence classes) of measurable functions $f$ satisfying
$\int_{A} \abs{f(x)}^{p}\,\D{x}<\infty$, and $L^{\infty}(A)$ is the
Banach space of functions that are essentially bounded in $A$.
For an integer $k\ge 0$, the space $C^{k}(A)$ ($C^{\infty}(A)$)
refers to the class of all functions whose partial
derivatives up to order $k$ (of any order) exist and are continuous,
$C_{c}^{k}(A)$ is the space of functions in $C^{k}(A)$ with
compact support, and $C_{b}^{k}(A)$ is the subspace of $C^{k}(A)$
consisting of those functions whose derivatives up to order $k$ are bounded.
Also, the space $C^{k,r}(A)$, $r\in(0,1]$, is the class of all functions
whose partial derivatives up to order $k$ are H\"older continuous of order $r$.
For simplicity we write  $C^{0,r}(A)=C^{r}(A)$.
For any $\gamma> 0$, $C^\gamma(A)$ denotes the space
$C^{\lfloor\gamma\rfloor,\gamma-\lfloor\gamma\rfloor}(A)$,
under the convention $C^{k,0}(A) = C^{k}(A)$.

In general if $\mathcal{X}$ is a space of real-valued functions on a
domain $D$,
$\mathcal{X}_{\mathrm{loc}}$ consists of all functions $f$ such that
$f\varphi\in\mathcal{X}$ for every $\varphi\in C_{c}^{\infty}(D)$.

For a nonnegative multiindex $\beta=(\beta_{1},\dotsc,\beta_{d})$,
we let $\abs{\beta}\df \beta_{1}+\dotsb+\beta_{d}$ and
$D^{\beta}\df \partial_{1}^{\beta_{1}}\dotsb
\partial_{d}^{\beta_{d}}$, where
$\partial_{i} \df\frac{\partial}{\partial x_{i}}$.

Let $D$ be a bounded domain with a $C^{2}$ boundary.
Define $d_{x} \;\df\; \text{dist}(x, \partial D)$
and $d_{xy}\df\min(d_{x}, d_{y})$.
For $u\in C(D)$ and $r\in\RR$,
we introduce the weighted norm
\begin{align*}
\dbrac{u}^{(r)}_{0;D}&\;\df\;
\sup_{x\in D}\; d_{x}^{r}\,\abs{u(x)}\,,\\
\intertext{and, for $k\in\NN$ and $\delta\in(0,1]$, the seminorms}
\dbrac{u}^{(r)}_{k;D}&\;\df\; \sup_{\abs{\beta}=k}\;
\sup_{x\in D}\; d_{x}^{k+r}\babs{D^{\beta}u(x)}\\
\dbrac{u}^{(r)}_{k,\delta;D}&\;\df\; \sup_{\abs{\beta}=k}\;
\sup_{x,y\in D}\; \biggl(d_{xy}^{k+\delta+r}\,
\frac{\babs{D^{\beta}u(x)-D^{\beta}u(y)}}{\abs{x-y}^{\delta}}\biggr)\,.
\end{align*}
For $r\in\RR$ and $\gamma\ge0$, with $\gamma+r\ge0$, we define the space
\begin{equation*}
\mathscr{C}^{(r)}_{\gamma}(D)\;\df\;
\bigl\{u\in C^{\gamma}(D)\cap C(\R^{d})\,\colon
u(x)=0~\text{for}\,x\in D^{c},
~ \dabs{u}^{(r)}_{\gamma;D}<\infty\bigr\}\,,
\end{equation*}
where
\begin{equation*}
\dabs{u}^{(r)}_{\gamma;D}\;\df\;
\sum_{k=0}^{\lceil\gamma\rceil-1}
\dbrac{u}^{(r)}_{k,D}
+ \dbrac{u}^{(r)}_{\lceil\gamma\rceil-1,\,\gamma+1-\lceil\gamma\rceil;\,D}\,,
\end{equation*}
under the convention $\dabs{u}^{(r)}_{0;D}=\dbrac{u}^{(r)}_{0;D}$.
We also use the notation $\dabs{u}^{(r)}_{k,\delta;D}
=\dabs{u}^{(r)}_{k+\delta;D}$ for $\delta\in(0,1]$.
It is straightforward to verify that $\dabs{u}^{(r)}_{\gamma;D}$
is a norm, under which $\mathscr{C}^{(r)}_{\gamma}(D)$ is a Banach space.

If the distance functions $d_x$ or $d_{xy}$ are not included in the above
definitions, we denote the corresponding seminorms by
$[\,\cdot\,]_{k;D}$ or $[\,\cdot\,]_{k,\delta;D}$ and define
\begin{equation*}
\norm{u}_{C^{k,\delta}(D)} \;\df\;
\sum_{\ell=0}^{k} [u]_{\ell;D} + [ u]_{k,\delta;D}\,.
\end{equation*}
Thus, $\norm{u}_{C^{\gamma}(D)}$ is well defined for any $\gamma>0$,
by the identification
$C^{\gamma}(D)=C^{\lfloor\gamma\rfloor,\gamma-\lfloor\gamma\rfloor}(A)$.

We recall the well known interpolation inequalities
\cite[Lemma~6.32, p.~30]{GilTru}.
Let $u\in C^{2,\beta}(D)$.
Then for any $\varepsilon$ there exists a constant $C=C(\varepsilon,j,k,r)$
such that 
\begin{equation*}
\begin{gathered}
\dbrac{u}^{(0)}_{j,\gamma;D}\;\le\; C\,\dabs{u}^{(0)}_{0;D}
+ \varepsilon\,[\![u]\!]^{(0)}_{k,\beta;D}
\\[5pt]
\dabs{u}^{(0)}_{j,\gamma;D}\;\le\; C\,\dabs{u}^{(0)}_{0;D}
+ \varepsilon\,\dbrac{u}^{(0)}_{k,\beta;D}
\end{gathered}
\qquad
j=0,1,2,\qquad 0\le\beta,\,\gamma\le 1\,,\quad
j+\gamma < k + \beta\,.
\end{equation*}

\section{Preliminaries}\label{S2}
Let $\alpha\in(1,2)$.
Let $b:\Rd\to\Rd$ and $\uppi:\Rd\times\Rd\to \R$ be two given measurable
functions where $\uppi$ is nonnegative.
We define the non-local operator $\I$ as follows:
\begin{equation}\label{ee2.1}
\I f(x)\;\df\;
b(x)\cdot\grad f(x) +\int_{\Rds}\dd f(x;z)\,
\uppi(x,z)\,\D{z}\,,
\end{equation}
with $\dd f$ as in \eqref{ee1.4}.
We always assume that
\begin{equation*}
\int_{\Rds}(\abs{z}^{2}\wedge 1)\,\uppi(x,z)\,\D{z}\;<\;\infty
\qquad \forall\,x\in\Rd\,.
\end{equation*}
Note that \eqref{ee2.1} is well-defined for any $f\in C^{2}_{b}(\Rd)$.
Let $\Omega=\calD([0, \infty), \Rd)$ denote the space of all right continuous
functions mapping $[0,\infty)$ to $\Rd$, having finite left limits
(c\'{a}dl\'{a}g).
Define $X_{t}=\omega(t)$ for $\omega\in\Omega$ and let $\{\calF_{t}\}$ be the
right-continuous filtration generated by the process $\{X_{t}\}$.
In this paper we always assume that given any initial distribution $\nu_{0}$
there exists a strong Markov process $(X, \Prob_{\nu_{0}})$ that satisfies
the martingale problem corresponding to $\I$,
i.e., $\Prob_{\nu_{0}}(X_{0}\in A)=\nu_{0}(A)$ for all $A\in\calB(\Rd)$ and 
for any $f\in C^{2}_{b}(\Rd)$, 
\begin{equation*}
f(X_{t})-f(X_{0})-\int_{0}^{t}\I f(X_{s})\,\D{s}
\end{equation*}
is a martingale with respect to the filtration $\{\calF_{t}\}$.
We denote the law of the process by $\Prob_{x}$ when $\nu_{0}=\delta_{x}$.
Sufficient conditions on $b$ and $\uppi$ to ensure
the existence of such processes are available in the literature.
Unfortunately, the available sufficient conditions do not cover a wide class
of processes.
We refer the reader to \cite{bass04} for the available results in this direction,
as well as to \cite{applebaum, chen-kim-song, chen-wang, komatsu, kurtz}.
When $b\equiv 0$, well-posedness of the martingale problem is obtained
under some regularity assumptions on $\uppi$ in \cite{abels-kassmann}.

Let us mention once more that our goal here is not to study the existence
of a solution to the martingale problem.
Therefore, we do not assume any regularity conditions on the coefficients,
unless otherwise stated.
Before we proceed to state our assumptions
and results, we recall the L\'{e}vy-system formula, the proof of which is a
straightforward adaptation of the proof for a purely non-local operator
and can be found in \cite[Proposition 2.3 and Remark 2.4]{bass-levin}
(see also \cite{chen-kim-song, foondun}).

\begin{proposition}\label{levy-system}
If $A$ and $B$ are disjoint Borel sets in $\calB(\Rd)$, then for any $x\in\Rd$,
\begin{equation*}
\sum_{s\le t}\bm1_{\{X_{s-}\in A,\, X_{s}\in B\}}
-\int_{0}^{t}\int_{B} \bm1_{\{X_{s}\in A\}}
\uppi(X_{s}, z-X_{s})\,\D{z}\,\D{s}
\end{equation*}
is a $\Prob_{x}$-martingale.
\end{proposition}

\section{Probabilistic representations of solutions of
non-local PDE}\label{sec-pde}
The aim in this section is to give a rigorous mathematical justification
of the connections between stochastic differential equations with jumps and
viscosity solutions to associated non-local differential equations.

Recall the generator in \eqref{ee2.1} 
where $f$ is in $C^{2}_{b}(\R^{d})$. 
We also recall the definition of a viscosity solution
\cite{barles-chas-imbert, caffarelli-silvestre-regu}.

\begin{definition}\label{D-visc}
Let $D$ be a domain with $C^{2}$ boundary.
A function $u:\Rd\to\R$ which is upper (lower)
semi-continuous on $\Bar{D}$ is said to
be a sub-solution (super-solution) to
\begin{align*}
\I u& \;=\; -f \quad \text{in}~D\,, \\[3pt]
u& \;=\; g \quad \text{in}~D^{c}\,,
\end{align*}
where $\I$ is given by \eqref{ee2.1},
if for any $x\in\Bar{D}$ and a function $\varphi\in C^{2}(\Rd)$ such that
$\varphi(x)=u(x)$ and
$\varphi(z)> u(z)$ $\bigl(\varphi(z)< u(z)\bigr)$ on $\Rd\setminus\{x\}$,
it holds that
\begin{equation*}
\I\varphi(x)\ge -f(x) \quad\bigl(\I\varphi(x)\le -f(x) \bigr)\,,
\quad\text{if}~x\in D\,,
\end{equation*}
while, if $x\in\partial D$, then
\begin{equation*}
\max\;(\I\varphi(x)+f(x), g(x)-u(x))\ge 0
\quad\bigl(\min\;(\I\varphi(x)+f(x), g(x)-u(x))\le 0\bigr)\,.
\end{equation*}
A function
$u$ is said to be a viscosity solution if it is both a sub- and a super-solution.
\end{definition}

In Definition~\ref{D-visc}
we may assume that $\varphi$ is bounded, provided $u$ is bounded.
Otherwise, we may modify the function $\varphi$ by
replacing it with $u$ outside a small ball around $x$.
It is evident that every classical solution is also a viscosity solution.

Let $f$ and $g$ be two continuous functions on $\Rd$, with $g$ bounded.
Given a bounded domain $D$, we let
\begin{equation}\label{E3.1}
u(x)\;=\;\Exp_{x}\biggl[\int_{0}^{\tau(D)}f(X_{s})\,\D{s}+g(X_{\tau(D)})\biggr]
\quad \text{for}~x\in\Rd\,,
\end{equation}
where $\Exp_{x}$ denotes the expectation operator relative to $\Prob_{x}$.
In this section we characterize $u$ as
a solution of a non-local differential equation.
As usual, we say that $b$ is locally bounded, if for any compact set $K$,
$\sup_{x\in K}\;\abs{b(x)}<\infty$.

\subsection{Three lemmas concerning operators with measurable kernels}

\begin{lemma}\label{L3.1}
Let $D$ be a bounded domain.
Suppose $X$ is a strong Markov process associated
with $\I$ in \eqref{ee2.1}, with $b$ locally bounded,
and that the integrability conditions
\begin{equation}\label{ee3.2}
\sup_{x\in K}\;\int_{\{\abs{z}>1\}}\abs{z}\,\uppi(x,z)\,\D{z}\;<\;\infty\,,
\quad\text{and}\quad \inf_{x\in K}\;\int_{\Rds}\abs{z}^{2}\uppi(x,z)\,\D{z}
\;=\;\infty
\end{equation}
hold for any compact set $K$.
Then
$\sup_{x\in D}\Exp_{x}[(\tau(D))^{m}]<\infty$, for any positive integer $m$.
\end{lemma}

\begin{proof}
Without loss of generality we assume that $0\in D$.
Otherwise we inflate the domain to include $0$.
Let $\Bar d=\diam(D)$ and $M_{D}=\sup_{x\in D}\,\abs{b(x)}$.
Recall that $B_{R}$ denotes the ball of radius $R$ around the origin.
We choose $R>1\vee 2(\Bar d\vee M_{D})$, and large enough so as to satisfy
the inequality
\begin{equation*}
\inf_{x\in D}\;\int_{B_{R}}\abs{z}^{2}\,\uppi(x,z)\,\D{z}
\;>\; 1+ 2\Bar d M_{D}+2\Bar d\sup_{x\in D}\;
\int_{\{1< \abs{z}\le R\}} \abs{z}\, \uppi(x,z)\,\D{z}\,.
\end{equation*} 
We let $f\in C^{2}_{b}(\Rd)$ be a radially increasing
function such that $f(x)=\abs{x}^{2}$ for $\abs{x}\le 2R$ and
$f(x)=8R^{2}$ for $\abs{x}\ge 2R+1$.
Then, for any $x\in D$, we have
\begin{align*}
\I f (x) &\;=\;b(x)\cdot\grad f(x)
+\int_{\Rds}\dd f(x;z)\,\uppi(x,z)\,\D{z}\\
&\;\ge\; -2\Bar d M_{D}
+ \int_{B_{R}}\bigl(f(x+z)-f(x)-\grad f(x)\cdot z\bigr)\,\uppi(x,z)\,\D{z}\\
&\mspace{50mu} +\int_{\{1< \abs{z}\le R\}}\grad f(x)\cdot z\;\uppi(x,z)\,\D{z}
+\int_{B_{R}^{c}}\bigl(f(x+z)-f(x)\bigr)\,\uppi(x,z)\,\D{z}\,.
\end{align*}
Also, for any $\abs{z}\ge R$, it holds that $\abs{x+z}\ge \Bar d\ge \abs{x}$.
Therefore $f(x+z)\ge f(x)$.
Hence 
\begin{align*}
\I f(x) &\;\ge\; -2\Bar d M_{D} +\int_{\{1< \abs{z}\le R\}}\grad f(x)\cdot z\;
\uppi(x,z)\,\D{z}\\
&\mspace{200mu}+ \int_{B_{R}}\bigl(f(x+z)-f(x)-\grad f(x)\cdot z\bigr)\,
\uppi(x,z)\,\D{z}\\
&\;\ge\; -2\Bar d M_{D} -2\Bar d\int_{\{1<\abs{z}\le R\}} \abs{z}\,
\uppi(x,z)\,\D{z}
+\int_{B_{R}}\abs{z}^{2}\,\uppi(x,z)\,\D{z}\\
&\;\ge\; 1\,.
\end{align*}
Thus
\begin{align*}
\Exp_{x}[f(X_{\tau(D)\wedge t})]-f(x)
&\;=\; \Exp_{x}\biggl[\int_{0}^{\tau(D)\wedge t}\I f(X_{s})\,\D{s}\biggr]\\[5pt]
&\;\ge\; \Exp_{x}[\tau(D)\wedge t]\qquad\forall\,x\in D\,.
\end{align*}
Letting $t\to\infty$ we obtain $\Exp_{x}[\tau(D)]\le 8R^{2}$.
Since $x\in D$ is arbitrary, this shows that
$$\sup_{x\in D}\;\Exp_{x}[\tau(D)]\le 8R^{2}\,.$$

We continue by using the method of induction.
We have proved the result for $m=1$.
Assume that it is true for $m$, i.e., 
$M_{m}\df\sup_{x\in D}\Exp_{x}[(\tau(D))^{m}]<\infty$. 
Let $h(x)= M_{m}f(x)$ where $f$ is defined above.
Then from the calculations above we obtain
\begin{equation}\label{ee3.3}
\Exp_{x}[h(X_{\tau(D)\wedge t})]-h(x)\;\ge\; \Exp_{x}[M_{m}(\tau(D)\wedge t)]
\qquad\forall\,x\in D\,.
\end{equation}
Denoting $\tau(D)$ by $\tau$ we have
\begin{align*}
\Exp_{x}[\tau^{m+1}]
&\;=\;\Exp_{x}\biggl[\int_{0}^{\infty} (m+1)(\tau-t)^{m}\,
\bm1_{\{t<\tau\}}\,\D{t}\biggr]
\\[5pt]
&\;=\; \Exp_{x}\biggl[\int_{0}^{\infty} (m+1)
\Exp_{x}\bigl[(\tau-t)^{m}\,
\bm1_{\{t<\tau\}}\bigm|\calF_{t\wedge\tau}\bigr]\,\D{t}\biggr]
\\[5pt]
&\;=\; \Exp_{x}\biggl[\int_{0}^{\infty} (m+1) \bm1_{\{t\wedge \tau<\tau\}}
\Exp_{X_{t\wedge \tau}}[\tau^{m}]\,\D{t}\biggr]
\\[5pt]
&\;\le\; \sup_{x\in D}\;\Exp_{x}[\tau^{m}]\,
\Exp_{x}\biggl[\int_{0}^{\infty} (m+1) \bm1_{\{t\wedge \tau<\tau\}}\,\D{t}\biggr]
\\[5pt]
&\;\le\; M_{m}(m+1)\,\Exp_{x}[\tau]\,,
\end{align*}
and in view of \eqref{ee3.3}, the proof is complete.
\end{proof}

Boundedness of solutions to the Dirichlet problem 
on bounded domains and with zero boundary data is asserted in the following lemma.

\begin{lemma}\label{L3.2}
Let $b$ and $f$ be locally bounded functions and $D$ a bounded domain.
Suppose $\uppi$ satisfies \eqref{ee3.2}.
Then there exists a constant $C$, depending on
$\diam(D)$, $\sup_{x\in D}\,\abs{b(x)}$ and $\uppi$,
such that any viscosity solution $u$ to the equation
\begin{align*}
\I u &\;=\;f \quad \text{in}~D\,,\\
u&\;=\;0 \quad \text{in}~D^{c}\,,
\end{align*} 
satisfies $\norm{u}_{\infty}\le C\, \sup_{x\in D}\,\abs{f(x)}$.
\end{lemma}

\begin{proof}
As shown in the proof of Lemma~\ref{L3.1},
there exists a nonnegative, radially nondecreasing
function $\xi\in C^{2}_{b}(\Rd)$ satisfying
$\I\xi(x)>\sup_{x\in D}\,\abs{f(x)}$ for all $x\in \bar{D}$.
Let $M>0$ be the smallest number such that $M-\xi$
touches $u$ from above at least at one point.
We claim that $M\le \norm{\xi}_{\infty}$.
If not, then $M-\xi(x)>0$ for all $x\in D^{c}$.
Therefore $M-\xi$ touches $u$ in $D$ from above.
Hence by the definition of a viscosity solution we have
$\I (M-\xi(x))\ge f(x)$, or equivalently,
$\I \xi(x)\le - f(x)$, where $x\in D$ is a point of contact from above.
But this contradicts the definition of $\xi$.
Thus $M\le \norm{\xi}_{\infty}$.
Also by the definition of $M$ we have
\begin{equation*}
\sup_{x\in D}\;u(x)\;\le\; \sup_{x\in D}\;(M-\xi(x))\;\le\; M
\;\le\; \norm{\xi}_{\infty}\,.
\end{equation*}
The result then follows by applying the same argument to $-u$.
\end{proof}

\begin{definition}\label{D3.2}
Let $\fL$ denote the class of operators $\I$ of the form
\begin{equation}\label{gen-op}
\I f(x)\;\df\;b(x)\cdot\grad f(x)+\int_{\Rds}
\dd f(x;z)\,
\frac{k(x,z)}{\abs{z}^{d+\alpha}}\,\D{z}\,,\quad f\in C^{2}_{b}(\Rd)\,,
\end{equation}
with $b:\Rd\to\Rd$ and $k:\Rd\times\Rd\to(0,\infty)$ Borel measurable and
locally bounded, and $\alpha\in(1,2)$.
We also assume that $x\mapsto\sup_{z\in\Rd}\,k^{-1}(x,z)$ is locally bounded.
The subclass of $\fL$ consisting of those $\I$ satisfying
$k(x,z)=k(x,-z)$ is denoted by $\fLsym$.
\end{definition}

Consider the following \emph{growth condition}:
There exists a constant $K_{0}$ such that
\begin{equation}\label{E3.5}
x\cdot b(x)\,\vee\,\abs{x}\,k(x,z)\;\le\; K_{0}\,(1+\abs{x}^{2})\qquad
\forall\, x, z\in\Rd\,.
\end{equation}
It turns out that under \eqref{E3.5}, the Markov process
associated with $\I$
does not have finite explosion time, as the following lemma shows.

\begin{lemma}\label{L3.3}
Let $\I\in\fL$ and
suppose that for some constant $K_{0}>0$, the data satisfies
the growth condition in \eqref{E3.5}.
Let $X$ be a Markov process associated with $\I$.
Then 
\begin{equation*}
\Prob_{x}\,\biggl(\sup_{s\in[0, T]}\,\abs{X_{s}}<\infty\biggr)\;=\;1
\qquad\forall\, T\,>\,0\,.
\end{equation*}
\end{lemma}

\begin{proof}
Let $\delta\in(0,\alpha-1)$, and $f\in C^{2}(\Rd)$ be a non-decreasing,
radial function satisfying
\begin{equation*}
f(x)\;=\;\bigl(1+\abs{x}^{\delta}\bigr)\quad\text{for}~\abs{x}\ge 1\,,
\quad\text{and}\quad
f(x)\;\ge\; 1\quad\text{for}~\abs{x}< 1\,.
\end{equation*}
We claim that
\begin{equation}\label{ee3.6}
\babss{\int_{\Rds}\dd f(x;z)\,\frac{k(x,z)}{\abs{z}^{d+\alpha}}\,\D{z}}
\;\le\; \kappa_{0}\,(1+\abs{x}^{\delta})\qquad\forall x\in\Rd\,,
\end{equation}
for some constant $\kappa_{0}$.
To prove \eqref{ee3.6} first note that since the second partial derivatives
of $f$ are bounded over $\Rd$, it follows that
$\Babs{\int_{\abs{z}\le 1}\dd f(x;z)\,\frac{k(x,z)}{\abs{z}^{d+\alpha}}\,\D{z}}$
is bounded by some constant.
It is easy to verify that, provided $z\ne 0$, then
\begin{equation}\label{ee3.7}
\begin{aligned}
\babs{\abs{x+z}^{\delta} - \abs{x}^{\delta}}
&\;\le\; 2\delta \abs{z}\,\abs{x}^{\delta-1}\,,
&\qquad\text{if}~ \abs{x}\ge2\abs{z}\,,\\[5pt]
\babs{\abs{x+z}^{\delta} - \abs{x}^{\delta}}
&\;\le\; 8 \abs{z}^{\delta}\,, &\qquad\text{if}~ \abs{x}<2\abs{z}\,,
\end{aligned}
\end{equation}
for some constant $\kappa$.
By the hypothesis in \eqref{E3.5}, for some constant $c$, we have
\begin{equation}\label{ee3.8}
k(x,z)\;\le\; c\,(1+\abs{x})\qquad\forall x\in\Rd\,.
\end{equation}
Combining \eqref{ee3.7}--\eqref{ee3.8} we obtain,
for $\abs{x}>1$,
\begin{align*}
\babss{\int_{\abs{z}>1}\dd f(x;z)\,\frac{k(x,z)}{\abs{z}^{d+\alpha}}\,\D{z}}
&\;\le\;
\int_{1<\abs{z}\le \frac{\abs{x}}{2}}
2\delta\,c\,(1+\abs{x})\,\abs{x}^{\delta-1}\,\abs{z}\,
\frac{1}{\abs{z}^{d+\alpha}}\,\D{z}\nonumber\\[5pt]
&\mspace{100mu}
+\int_{\abs{z}>\frac{\abs{x}}{2}}
8c\,(1+\abs{x})\,\abs{z}^{\delta}\,
\frac{1}{\abs{z}^{d+\alpha}}\,\D{z}\nonumber\\[5pt]
&\;\le\;
\kappa(d)\biggl(\frac{2\,\delta\,c}{\alpha-1}\,
(1+\abs{x})\,\abs{x}^{\delta-1}
+\frac{2^{3+\alpha-\delta}c}{\alpha-\delta}\,(1+\abs{x})\,\abs{x}^{\delta-\alpha}
\biggr)
\end{align*}
for some constant $\kappa(d)$,
thus establishing \eqref{ee3.6}.

By \eqref{ee3.6} and the assumption on the growth of $b$ in \eqref{E3.5},
we obtain
\begin{equation*}
\abs{\I f(x)}\;\le\; K_{1}\; f(x)\qquad\forall x\in\Rd\,,
\end{equation*}
for some constant $K_{1}$.
Then, by Dynkin's formula, we have,
\begin{align*}
\Exp_{x}\bigl[f(X_{t\wedge\tau_{n}})\bigr]&\;=\;
f(x) + \Exp_{x}\,\biggl[
\int_{0}^{t\wedge\tau_{n}}\I f(X_{s})\,\D{s}\biggr]\nonumber\\[5pt]
&\;\le\; f(x) + K_{1}\,
\Exp_{x}\,\biggl[\int_{0}^{t\wedge\tau_{n}} f(X_{s})\,\D{s}\biggr]\nonumber\\[5pt]
&\;\le\; f(x) + K_{1}\,
\int_{0}^{t} \Exp_{x}\,\bigl[f(X_{s\wedge\tau_{n}})\bigr]\,\D{s}\,,
\end{align*}
where in the second inequality we use the property that $f$ is radial
and non-decreasing.
Hence, by the Gronwall inequality, we have
\begin{equation}\label{ee3.9}
\Exp_{x}\bigl[f(X_{t\wedge\tau_{n}})\bigr]\;\le\; f(x)\,\E^{K_{1} t}\qquad
\forall\,t>0\,,\quad \forall\,n\in\NN\,.
\end{equation}
Since
$\Exp_{x}\bigl[f(X_{t\wedge\tau_{n}})\bigr]\,\ge\,f(n)\,\Prob_{x}(\tau_{n}\le t)$,
we obtain by \eqref{ee3.9} that
\begin{align*}
\Prob_{x}\,\biggl(\sup_{s\in[0, T]}\,\abs{X_{s}}\ge n\biggr)&\;=\;
\Prob_{x}(\tau_{n}\le T)\\[5pt]
&\;\le\; \frac{f(x)}{1+n^{\delta}}\,\E^{K_{1} T}\qquad
\forall\,T>0\,,\quad \forall\,n\in\NN\,,
\end{align*}
from which the conclusion of the lemma follows.
\end{proof}

\subsection{A class of operators with weakly H\"older continuous kernels}
\label{S3.2}

We introduce a class of kernels whose numerators
$k(x,z)$ are locally H\"older continuous in $x$, and $z\mapsto k(x,z)$
is bounded, locally in $x$.
We call such kernels $\uppi$ weakly H\"older continuous since they have the
property that for any $f$ satisfying
$\int_{\Rd} \frac{f(z)}{\abs{z}^{d+\alpha}}\,\D{z}<\infty$
the map $x\mapsto \int_{\Rd} f(z)\frac{k(x,z)}{\abs{z}^{d+\alpha}}\,\D{z}$
is locally H\"older continuous.

\begin{definition}\label{D-hkernel}
Let $\lambda: [0,\infty) \to (0,\infty)$ be a nondecreasing function
that plays the role of a parameter.
For a bounded domain $D$ define $\lambda_{D}\df \sup\,\{\lambda(R)\,\colon
D\subset B_{R+1}\}$.
Let $\mathfrak{I}_{\alpha}(\beta,\theta,\lambda)$, where
$\beta\in(0,1]$, $\theta\in(0,1)$,
denote
the class of operators $\I$ as in \eqref{gen-op}
that satisfy, on each bounded domain $D$, the following properties:
\begin{itemize}
\item[(a)]
$\alpha\in (1,2)$.
\item[(b)]
$b$ is locally H\"older continuous with exponent $\beta$,
and satisfies
\begin{equation*}
\abs{b(x)}\;\le\;\lambda_{D}\,,\quad\text{and}\quad
\abs{b(x)-b(y)}\;\le\; \lambda_{D}\,\abs{x-y}^{\beta}
\qquad \forall\, x,\,y\in D\,.
\end{equation*}
\item[(c)]
The map $k(x,z)$ is continuous in $x$ and measurable in $z$ and
satisfies
\begin{equation*}
\begin{split}
\abs{k(x,z)-k(y,z)}\;\le\; \lambda_{D}\,\abs{x-y}^{\beta}
\qquad \forall\, x,\,y\in D\,,\quad \forall\,z\in\Rd\\[5pt]
\lambda^{-1}_{D}\;\le\; k(x,z)\;\le\;\lambda_{D}
\qquad \forall\, x\in D\,,\quad \forall\,z\in\Rd\,.
\end{split}
\end{equation*}
\item[(d)]
For any $x\in D$, we have
\begin{equation}\label{ee3.10}
\babss{\int_{\Rds} \bigl(\abs{z}^{\alpha-\theta}\wedge1\bigr)\,\,
\frac{\abs{k(x,z)-k(x,0)}}{\abs{z}^{d+\alpha}}\,\D{z}}
\;\le\;\lambda_{D}\,.
\end{equation}
\end{itemize}
\end{definition}

\begin{remark}
It is evident that if
$\abs{k(x,z)-k(x,0)}\le \Tilde{\lambda}_{D}\abs{z}^{\theta'}$ for some
$\theta'>\theta$,
then property (d) of Definition~\ref{D-hkernel} is satisfied .
\end{remark}

We may view $\I$ as the sum of the translation invariant operator
$\I_{0}$ defined by
\begin{equation*}
\I_{0}\,u(x)\;\df\;b(x)\cdot\grad u(x)
+ \int_{\Rds} \dd u(x;z)\, \frac{k(x,0)}{\abs{z}^{d+2s}}\,\D{z}\,,
\end{equation*}
which is uniformly elliptic on every bounded domain,
and a perturbation that takes the form
\begin{equation*}
\widetilde\I\,u(x)\;\df\;
 \int_{\Rds} \dd u(x;z)\, \frac{k(x,z)-k(x,0)}{\abs{z}^{d+2s}}\,\D{z}\,.
\end{equation*}
We are not assuming that the numerator $k$ is symmetric, as
in the approximation techniques in
\cite{caffarelli-silvestre-approx, Kriventsov,Bjorland-12}.
Moreover, these operators are
not addressed in \cite{ChangLara-Davila} due to the presence of
the drift term.

For operators in the class $\mathfrak{I}_{\alpha}(\beta,\theta,\lambda)$,
we have the following
regularity result concerning solutions to the Dirichlet problem.

\begin{theorem}\label{T3.1}
Let $\I\in\mathfrak{I}_{\alpha}(\beta,\theta,\lambda)$,
$D$ be a bounded domain with $C^{2}$ boundary, and $f\in C^{\beta}(\Bar{D})$.
We assume that neither $\beta$, nor $2s+\beta$ are integers, and
that either $\beta<s$, or that
$\beta\ge s$ and
\begin{equation*}
\abs{k(x,z)-k(x,0)}\;\le\; \Tilde{\lambda}_{D}\, \abs{z}^{\theta}\qquad
\forall x\in D\,, \;\forall z\in \Rd\,,
\end{equation*}
for some positive constant $\Tilde{\lambda}_{D}$.
Let $\Exp_{x}$ denote the expectation operator
corresponding to the Markov process $X$ with generator given by $\I$.
Then $u(x)\df\Exp_{x}\bigl[\int_{0}^{\tau(D)}f(X_{s})\,\D{s}\bigr]$
is the unique solution in $C^{\alpha+\beta}(D)\cap C(\Bar{D})$
to the equation
\begin{align*}
\I u &\;=\;-f \quad \text{in}~D\,,\\
u&\;=\;0 \quad \text{in}~D^{c}\,.
\end{align*}
\end{theorem}

\begin{proof}
For $\varepsilon>0$, 
we denote by $D_{\varepsilon}$ the $\varepsilon$-neighborhood of $D$, i.e.,
\begin{equation}\label{ee3.11}
D_{\varepsilon}\;\df\;\{z\in \Rd:\,\dist(z, D)< \varepsilon\}\,. 
\end{equation}
Note that for $\varepsilon$ small enough, $D_{\varepsilon}$
has a $C^{2}$ boundary.
Then by Theorem~\ref{T6.1} there
exists $u_\varepsilon\in C^{\alpha+\beta}(D)\cap C(\Bar{D})$ satisfying
\begin{align*}
\I u_\varepsilon &\;=\;-f \quad \text{in}~D_{\varepsilon},\\
u_\varepsilon&\;=\;0 \quad \text{in}~D_{\varepsilon}^{c}\,.
\end{align*}
In the preceding equation $f$ stands for the Lipschitz extension of $f$.
We also have the estimate (recall the definition of
$\dabs{\,\cdot\,}^{(r)}_{\beta;D}$ in Section~\ref{S1.1})
\begin{equation*}
\dabs{u_\varepsilon}^{(-r)}_{\alpha+\beta;D_{\varepsilon}}
\;\le\; C_{0} \, \norm{f}_{C^{\beta}(\Bar{D}_\varepsilon)}\,,
\end{equation*}
with $r$ some fixed constant in $\bigl(0,\frac{\alpha}{2}\bigr)$.
As can be seen from the Lemma~\ref{L3.2} and the proof of
Theorem~\ref{T6.1}, we may select a constant $C_{0}$,
that does not depend on $\varepsilon$, for $\varepsilon$ small enough.
Since $u_\varepsilon=0$ in $D^{c}_\varepsilon$, it follows that
\begin{equation*}
\norm{u_\varepsilon}_{C^{r}(\R^{d})}\;\le\;
c_{1}\,\dabs{u_\varepsilon}^{(-r)}_{\alpha+\beta;D_{\varepsilon}}
\end{equation*}
for some constant $c_{1}$, independent of $\varepsilon$,
for all small enough $\varepsilon$.
Hence $u_\varepsilon\to u$ as $\varepsilon\to 0$, along
some subsequence, and $u\in C^{\alpha+\beta}(D)\cap C(\Bar{D})$
by Theorem~\ref{T6.1}.
By It\^{o}'s formula, we obtain
\begin{equation*}
u_\varepsilon(x)\;=\; \Exp_{x}\bigl[u_\varepsilon(X_{\tau(D)})\bigr]
+\Exp_{x}\biggl[\int_{0}^{\tau(D)}f(X_{s})\,\D{s}\biggr]\,.
\end{equation*}
Letting $\varepsilon\searrow 0$, we obtain the result.
Uniqueness follows from Theorem~\ref{T6.1}.
\end{proof}

Theorem~\ref{T3.1} can be extended to account for
non-zero boundary conditions, provided the boundary data is
regular enough, say in $C^3(\R^{d})\cap C_{b}(\Rd)$.

\subsection{Some results concerning the fractional Laplacian with drift}

In the rest of this section we consider a smaller class of operators,
but the data of the Dirichlet problem is only continuous.
We focus on stochastic differential equations
driven by a symmetric $\alpha$-stable process.
More precisely, we consider a process $X$ satisfying
\begin{equation}\label{ee3.12}
dX_{t}\;=\;b(X_{t})\,\D{t}+\D{L}_{t}\,,
\end{equation}
where $L_{t}$ is a symmetric $\alpha$-stable process with generator given by
\begin{equation*}
-(-\Delta)^{\nicefrac{\alpha}{2}}f(x)
\;=\;c(d,\alpha) \int_{\Rds}\dd f(x;z)\,\frac{1}{\abs{z}^{d+\alpha}}\,\D{z}\,,
\quad f\in C^{2}_{b}(\Rd)\,,
\end{equation*}
with $\alpha\in (1,2)$, and $c(d,\alpha)$ a normalizing constant.
Then the solution of \eqref{ee3.12} is also a solution to the
martingale problem for $\overline\I$ given by
\begin{equation*}
\overline\I f(x) \;\df\; -(-\Delta)^{\nicefrac{\alpha}{2}}f(x)
+ b(x)\cdot \grad{f}(x) \,,\quad\alpha\in (1,2)\,.
\end{equation*}
The following condition is in effect for the
rest of this section unless mentioned otherwise.

\begin{condition}\label{cond1}
There exists a positive constant $M$ such that
\begin{align*}
\abs{b(x)-b(y)} \;\le\; & M \abs{x-y}\qquad\forall\,x, y\in\R^{d}\,, \\[3pt]
\norm{b}_{\infty} \;\le\; & M\,.
\end{align*}
\end{condition}

Under Condition~\ref{cond1}, equation \eqref{ee3.12}
has a unique adapted strong c\'{a}dl\'{a}g solution
for any initial condition $X_{0}=x\in\R^{d}$, which is a Feller process
\cite{applebaum}.
We need the following assertion whose proof is standard, and therefore
omitted.

\begin{lemma}\label{L3.4}
Assume Condition~\ref{cond1} holds and $T>0$.
Let $x_{n}\to x$ as $n\to \infty$ and $X^{n}$, $X$ denote the solutions to
\eqref{ee3.12} with initial data $X^{n}_{0}=x_{n}$, $X_{0}=x$, respectively.
Then
\begin{equation*}
\lim_{n\to\infty}\;
\Exp\Biggl[\sup_{s\in[0, T]}\,\abs{X^{n}_{s}-X_{s}}^{2}\Biggr]\;=\;0\,.
\end{equation*}
\end{lemma}

The rest of the section is devoted to the proof of the following result.

\begin{theorem}\label{T3.2}
Let $D\subset\Rd$ be a bounded domain
with $C^{1}$ boundary, $f\in C_b(D)$, and $g\in C_{b}(D^{c})$.
The function $u(\cdot)$ defined in \eqref{E3.1} is continuous and bounded,
and is the unique viscosity solution to the equation
\begin{equation}\label{ee3.13}
\begin{split}
\overline\I u& \;=\;-f \quad \text{in}~D\,,\\
u& \;=\; g \quad \text{in}~D^{c}\,.
\end{split}
\end{equation}
\end{theorem}

The proof of Theorem~\ref{T3.2} relies on several lemmas which follow.
The following lemma is a careful modification of
\cite[Lemma~2.1]{song-vondracek-05}.

\begin{lemma}\label{L3.5}
Let $D$ be a given bounded domain.
There exits a constant $\kappa_{1}>0$ such that for any $x\in D$ and 
$r\in (0, 1)$
\begin{equation*}
\Prob_{x}\biggl(\sup_{0\le s\le t}\;\abs{X_{s}-x}>r\biggr)\;\le\;
\kappa_{1} t\, r^{-\alpha}\qquad \forall\,x\in D\,,
\end{equation*}
where $X$ satisfies \eqref{ee3.12}, and $X_{0}=x$.
\end{lemma}

\begin{proof}
Let $f\in C^{2}_{b}(\Rd)$ be such that
$f(x)=\abs{x}^{2}$ for $\abs{x}\le \frac{1}{2}$,
and $f(x)=1$ for $\abs{x}\ge 1$.
Let $c_{1}$ be a constant such that
\begin{align*}
\norm{\grad f}_{\infty}&\;\le\; c_{1}\,,\\[5pt]
\babs{f(x+z)-f(x)-\grad f(x)\cdot z}&\;\le\; c_{1} \abs{z}^{2}\qquad
\forall\, x,z\in\Rd\,.
\end{align*}
Define $f_{r}(y)=f(\frac{y-x}{r})$ where $x$ is a point in $D$.
For $y\in \Bar{B}_{r}(x)$, we obtain
\begin{align*}
\biggl|\int_{\Rds} \dd f_{r}(y;z)\,
\frac{1}{\abs{z}^{d+\alpha}}\,\D{z}\biggr|
&\;\le\; \biggl|\int_{\abs{z}\le r}
\bigl(f_{r}(y+z)-f_{r}(y)-\grad f_{r}(y)\cdot z\bigr)
\frac{1}{\abs{z}^{d+\alpha}}\,\D{z}\biggr|\\[5pt]
&\mspace{150mu}+\biggl| \int_{\abs{z}>r}\bigl(f_{r}(y+z)-f_{r}(y)\bigr)
\frac{1}{\abs{z}^{d+\alpha}}\,\D{z}\biggr| \\[5pt]
& \;\le\; c_{1}\, \frac{1}{r^{2}}\int_{\abs{z}\le r}\abs{z}^{2-d-\alpha}\,\D{z}
+ 2\int_{\abs{z}>r}\abs{z}^{-d-\alpha}\,\D{z} \\[5pt]
&\;\le\; \frac{c_{2}}{r^{\alpha}}
\end{align*}
for some constant $c_{2}$.
Since $\alpha>1$, we have
\begin{equation*}
\babs{\overline\I f_{r}(y)}\;\le\; \frac{c_{3}}{r^{\alpha}}
\qquad \forall y\in \Bar{B}_{r}(x)\,,
\end{equation*}
where $c_{3}$ is a positive constant depending on $c_{2}$ and $M$.
Therefore using It\^o's formula we obtain
\begin{equation*}
\frac{c_{3}}{r^{\alpha}}\, \Exp_{x}\bigl[\tau(\Bar{B}_{r}(x))\wedge t\bigr]
\;\ge\; \Exp_{x}\bigl[f_{r}(X_{\tau(\Bar{B}_{r}(x))\wedge t})\bigr]\,.
\end{equation*}
Since $f_{r}=1$ on $B^{c}_{r}(x)$, we have
$\Prob_{x}(\tau(\Bar{B}_{r}(x))\le t)\le c_{3}r^{-\alpha}t$.
This completes the proof.
\end{proof}

\begin{remark}\label{R3.2}
It is clear from the proof of Lemma~\ref{L3.5} that the result also holds
for operators $\I\in\fL$.
However, in this case, the constant $\kappa_{1}$ depends also
on the local bounds of $k$ and $b$.
\end{remark}

We define the following process
\begin{equation}\label{ee3.14}
Y_{t}\;\df\;x+L_{t}\,.
\end{equation}
In other words, $Y$ is a symmetric $\alpha$-stable L\'{e}vy
process starting at $x$.
It is straightforward to verify using the martingale property that
for any measurable function $f:\Rd\to \R$, we have
\begin{equation}\label{ee3.15}
\Exp_{x}[f(Y_t)]\;=\;\Exp_{\frac{x}{a}}[f(aY_{a^{-\alpha}t})]\,.
\end{equation}

We recall the following theorem from \cite[Theorem~1]{hernandez}.

\begin{theorem}\label{hernan}
Let $\theta\in (0, \pi)$.
Let $G$ be a closed cone in $\Rd$, $d\ge 2$, of angle $\theta$ with vertex at $0$.
For $d=1$ we let the cone to be the closed half line.
Define
\begin{equation*}
\eta(G)\;=\;\inf\;\{t\ge 0\,\colon Y_{t}\notin G\}\,.
\end{equation*}
Then there exists a constant $p_\alpha(\theta)>0$ such that
\begin{align*}
\Exp_{x}[(\eta(G))^{p}]
&\;<\;\infty \qquad \text{for}~ p<p_\alpha(\theta)\,,\\[5pt]
\Exp_{x}[(\eta(G))^{p}]
&\;=\;\infty \qquad \text{for}~ p>p_\alpha(\theta)\,,
\end{align*}
for all $x\in G\setminus\{0\}$.
\end{theorem}
The result in \cite{hernandez} is proven for open cones.
The statement in Theorem~\ref{hernan} follows from the fact that every
closed cone is contained in an open cone except for the vertex of the cone
and with probability $1$ the exit location from an open cone is not the vertex.
The following result is also obtained in \cite{chen-kim-song} for $d\ge 2$,
using estimates of the transition density.
Our proof technique is different, 
so we present it here.

\begin{lemma}\label{L3.6}
Under the process $X$ defined in \eqref{ee3.12}, for any bounded domain $D$,
satisfying the exterior cone condition,
and $x\in\partial D$, it holds that $\Prob_{x}(\tau(\Bar{D})>0)=0$.
\end{lemma}

\begin{proof}
Let $x_{0}\in\partial D$ be a fixed point.
We consider an open cone $G$ in the complement of $\Bar{D}$ at a distance $r$
from the boundary $\partial D$, such that the distance between the cone and
the boundary equals the length of the linear segment
connecting $x_{0}$ with the vertex of the cone at $x_{r}$.
In fact, we may choose an angle $\theta$ and an axis for the cone that can be kept
fixed for all $r$ small enough and the above mentioned property holds.
It is quite clear that this can be done for some truncated cone.
So first we assume that the full cone $G$ with angle $\theta$ and vertex
$x_{r}$ lies in $D^{c}$.
Let $\eta(G^{c})$ denote the first hitting time of $G$.
Since a translation of coordinates does not affect the first hitting time,
we may assume that $x_{r}=0$.
Then from Theorem~\ref{hernan}, there exists $p\in(0,p_\alpha(\theta))$, satisfying
\begin{align*}
\Exp_{x_{0}}\bigl[\bigl(\eta(G^{c})\bigr)^{p}\bigr]\;=\;\abs{x_{0}}^{\alpha p}\,
\Exp_{\frac{x_{0}}{\abs{x_{0}}}}\bigl[\bigl(\eta(G^{c})\bigr)^{p}\bigr]
\;<\; \infty\,,
\end{align*}
where we used the property \eqref{ee3.15}.
Since $G$ is open, by upper semi-continuity we have
\begin{equation*}
\sup\;\bigl\{\Exp_{z}\bigl[\bigl(\eta(G^{c})\bigr)^{p}\bigr]\,
\colon z\in G^{c},\, \abs{z}=1\bigr\}\;<\;\infty\,.
\end{equation*}
Therefore we can find a constant $\kappa>0$ not depending on $r$
(for $r$ small) such that
\begin{equation}\label{ee3.16}
\Exp_{x_{0}}\bigl[\bigl(\eta(G^{c})\bigr)^{p}\bigr]
\;\le\;\kappa |r|^{\alpha p}\,.
\end{equation}
Let $\alpha'\in (1,\alpha)$.
Then by \eqref{ee3.16}, for any $\varepsilon>0$,
we may choose $r$ small enough so that
\begin{equation}\label{ee3.17}
\Prob_{x_{0}}\bigl(\eta(G^{c})>r^{\alpha'}\bigr)\;\le\;
\kappa r^{(\alpha-\alpha')p}\;<\;\varepsilon\,.
\end{equation}
Using Condition~\ref{cond1} and \eqref{ee3.12}, \eqref{ee3.14}, we have
\begin{equation}\label{ee3.18}
\sup_{s\in[0, r^{\alpha'}]}\;\abs{X_{s}-Y_{s}}\;\le\; M r^{\alpha'}\,,
\end{equation}
with probability $1$.
Hence on $\{\eta(G^{c})\le r^{\alpha'}\}$ we have 
$|Y_{\eta(G^{c})}-X_{\eta(G^{c})}|\le M r^{\alpha'}$ by \eqref{ee3.18}.
But $Y_{\eta(G^{c})}\in G$ and $\dist(x_{0}, G)=r$.
Since $Mr^{\alpha'}< r$ for $r$ small enough we have
$X_{\eta(G^{c})}\in (\Bar{D})^{c}$ on 
$\{\eta(G^{c})\le r^{\alpha'}\}$.
Therefore from \eqref{ee3.17} we obtain
\begin{equation*}
\Prob_{x_{0}}\bigl(\tau(\Bar{D})>r^{\alpha'}\bigr)\;<\;\varepsilon
\end{equation*}
for all $r$ small enough.
This concludes the proof for the case when we can fit a whole cone in $D^{c}$
near $x_{0}$.
For any other scenario we can modify the domain locally around $x_{0}$ and
deduce that the first exit time from the new domain is $0$.
We use Lemma~\ref{L3.5} to assert that with high probability the paths spend
$r^{\alpha}$ amount of time in a ball of radius of order $r$.
Combining these two facts concludes the proof.
\end{proof}

\begin{remark}
The result of Lemma~\ref{L3.6} still holds if $X$ satisfies
\eqref{ee3.12} in a weak-sense 
for some locally bounded measurable drift $b$ (see also \cite{chen-wang}).
\end{remark}

The following corollary follows from Lemma~\ref{L3.6}.

\begin{corollary}\label{C3.1}
Under the process $X$ defined in \eqref{ee3.12}, 
for any bounded domain $D$ with $C^{1}$ boundary,
$\Prob_{x}\bigl(\tau(D)=\tau(\Bar{D})\bigr)=1$ for all $x\in\Bar{D}$.
\end{corollary}

\begin{lemma}\label{L3.7}
Under the process $X$ defined in \eqref{ee3.12}, for any bounded domain $D$
with $C^{1}$ boundary, and
$x\in D$, we have
\begin{align*}
\Prob_{x}\bigl(X_{\tau(D)-}\in\partial D,\, X_{\tau(D)}\in \Bar{D}^{c}\bigr)
&\;=\;0\,,\\[5pt]
\Prob_{x}\bigl(X_{s-}\in\partial D,\, X_{s}\in D,\,
X_{t}\in\Bar{D}~\text{for all}~t\in[0,s]\bigr)
&\;=\;0\,.
\end{align*}
\end{lemma}

\begin{proof}
We only prove the first equality,
as the proof for the second one follows along the same lines.
Condition~\ref{cond1} implies that $X_{t}$ has a density
for every $t>0$ \cite{bogdan-jakubowski}.
We let
\begin{equation*}
\Hat{D}_{R}\;\df\;\{z\in D^{c}:\,\dist(z, D)\ge R\}\,. 
\end{equation*}
It is enough to prove that
$\Prob_{x}\bigl(X_{\tau(D)-}\in\partial D,\, X_{\tau(D)}\in \Hat{D}_{R}\bigr)=0$
for every $R>0$.
For any $t>0$, we obtain by Proposition~\ref{levy-system} that
\begin{align*}
\Prob_{x}\bigl(X_{t\wedge\tau(D)-}\in\partial D,\, X_{t\wedge\tau(D)}
\in \Hat{D}_{R}\bigr)
&\;\le\; 
\Exp_{x}\Biggl[\sum_{s\le t}
\bm1_{\{X_{s-}\in\partial D,\, X_{s}\in \Hat{D}_{R}\}}\biggr]
\\[5pt]
&\;=\; c(d,\alpha)\;\Exp_{x}\biggl[\int_{0}^{t} \bm1_{\{X_{s}\in\partial D\}}
\int_{\Hat{D}_{R}}\frac{1}{\abs{X_{s}-z}^{d+\alpha}}\,\D{z}\,\D{s}\biggr]
\\[5pt]
&\;\le\; \frac{\kappa}{R^{\alpha}}\,\Exp_{x}
\biggl[\int_{0}^{t} \bm1_{\{X_{s}\in\partial D\}}\,\D{s}\biggr]
\end{align*}
for some constant $\kappa$.
But the term on the right hand side
of the above inequality is $0$ by the fact the $X_{s}$ has density.
Hence
$\Prob_{x}\bigl(X_{t\wedge\tau(D)-}\in\partial D,\,
X_{t\wedge\tau(D)}\in \Hat{D}_{R}\bigr)=0$ for any $t>0$.
Letting $t\to \infty$ completes the proof.
\end{proof}

\medskip

\begin{proof}[Proof of Theorem~\ref{T3.2}]
Uniqueness follows by the comparison principle in \cite[Corollary 2.9]{hector}.
Since $f$ and $g$ are bounded, it follows from Lemma~\ref{L3.1} that $u$
is bounded.
Also in view of Lemma~\ref{L3.6}, $u(x)=g(x)$ for $x\in \overline{D^{c}}$.
First we show that $u$ is continuous in $D$.
Let $x_{n}\to x$ in $D$ as $n\to\infty$.
To simplify the notation, we let $\tau^{n}$ denote the first exit time
from $D$ for the process $X^{n}$ that starts at $x_{n}$.
Similarly, $\tau$ corresponds to the process $X$ that starts at $x$.
From Lemma~\ref{L3.4} we obtain
\begin{equation*}
\Exp\biggl[\sup_{s\in [0, T+1]}\;\abs{X^{n}_{s}-X_{s}}^{2}\biggr]
\;\xrightarrow[n\to\infty]{}\; 0\,.
\end{equation*}
Passing to a subsequence we may assume that 
\begin{equation}\label{ee3.19}
\sup_{s\in [0, T+1]}\;\abs{X^{n}_{s}-X_{s}}\;\to\; 0
\qquad\text{as $n\to\infty$, a.s.}
\end{equation}
Recall the definition of $D_{\varepsilon}$ in \eqref{ee3.11}.
It is evident that, for any $\varepsilon>0$, \eqref{ee3.19} implies that
\begin{equation*}
\liminf_{n\to\infty}\; \tau^{n}\wedge T\;\le\; \tau(D_{\varepsilon})\wedge T\,.
\end{equation*}
Since $\tau(D_{\varepsilon})
\;\xrightarrow[\varepsilon\searrow0]{}\;\tau(\Bar{D})$ a.s.,
we obtain
\begin{equation}\label{ee3.20}
\liminf_{n\to\infty}\; \tau^{n}\wedge T\;\le\; \tau(\Bar{D})\wedge T\,.
\end{equation}
On the other hand,
$\tau(\Bar{D})\,=\,\tau(D)$ a.s.\ by Corollary~\ref{C3.1},
and thus we obtain from \eqref{ee3.20} that
\begin{equation*}
\liminf_{n\to\infty}\; \tau^{n}\wedge T\;\le\; \tau\wedge T
\quad\text{a.s.}
\end{equation*}
The reverse inequality, i.e.,
\begin{equation*}
\liminf_{n\to\infty}\; \tau^{n}\wedge T\;\ge\; \tau\wedge T
\quad\text{a.s.}\,,
\end{equation*}
is evident from \eqref{ee3.19}.
Hence we have
\begin{equation}\label{ee3.21}
\lim_{n\to\infty}\; \tau^{n}\wedge T\;=\;\tau\wedge T \,,
\end{equation}
with probability $1$.
It then follows by \eqref{ee3.19} and \eqref{ee3.21} that
\begin{equation*}
\Exp \biggl[\int_{0}^{\tau^{n}\wedge T} f(X^{n}_{s})\,\D{s} -
\int_{0}^{\tau\wedge T} f(X_{s})\,\D{s}\biggr]\;\xrightarrow[n\to\infty]{}\;0
\qquad \forall\,T>0\,.
\end{equation*}
By Lemma~\ref{L3.1} we can take limits as $T\to\infty$ to obtain,
\begin{equation}\label{ee3.22}
\Exp\, \biggl[\int_{0}^{\tau^{n}} f(X^{n}_{s})\,\D{s} -
\int_{0}^{\tau} f(X_{s})\,\D{s}\biggr]\;\xrightarrow[n\to\infty]{}\;0\,.
\end{equation}
Since $g$ is bounded, by Lemma~\ref{L3.1} we obtain
\begin{equation*}
\Exp\,\bigl[\bm1_{\{\tau\ge T\}}\,\babs{g(X^{n}_{\tau^{n}})-g(X_{\tau})}
\bigr]\;\le\;
2\,\norm{g}_{\infty}\,\Prob(\tau\ge T)\;\xrightarrow[T\to\infty]{}\;0\,.
\end{equation*}
From now on we consider a continuous extension of $g$ on $\Rd$, also
denoted by $g$.
We use the triangle inequality
\begin{align}\label{ee3.23}
\Exp\,\bigl[\bm1_{\{\tau< T\}}\,\babs{g(X^{n}_{\tau^{n}})-g(X_{\tau})}\bigr]
\;&\le\;
\Exp\,\bigl[\bm1_{\{\tau< T\}}\,\babs{g(X^{n}_{\tau^{n}})-g(X_{\tau^{n}})}\bigr]
+\Exp\,\bigl[\bm1_{\{\tau< T\}}\,\babs{g(X_{\tau^{n}})-g(X_{\tau})}\,\bigr]
\nonumber\\[5pt]
&\le\;
\Exp\,\bigl[\bm1_{\{\tau^{n}< T+\varepsilon\}}\bm1_{\{\tau< T\}}\,
\babs{g(X^{n}_{\tau^{n}})-g(X_{\tau^{n}})}\bigr]\nonumber\\[5pt]
&\mspace{10mu}+
2\norm{g}_{\infty}\,\Prob(\tau^{n}\ge\tau+\varepsilon)
+
\Exp\,\bigl[\bm1_{\{\tau< T\}}\,
\babs{g(X_{\tau^{n}})-g(X_{\tau})}\bigr]\,.
\end{align}
The second term on the right hand side of
\eqref{ee3.23} tends to $0$ as $n\to\infty$, by \eqref{ee3.21}.
The first term is dominated by
\begin{equation*}
\Exp\,\biggl[\biggl(\sup_{0\le t\le T+\varepsilon}\;
\babs{g(X^{n}_{t})-g(X_{t})}\biggr)\,
\bm1_{\{\tau^{n}< T+\varepsilon\}}\bm1_{\{\tau< T\}}\biggr]\,,
\end{equation*}
so it also tends to $0$ as $n\to\infty$,
by \eqref{ee3.19}, and the
continuity and boundedness of $g$.
For the third term, we write
\begin{multline}\label{ee3.24}
\Exp\,\bigl[\bm1_{\{\tau< T\}}\,
\babs{g(X_{\tau^{n}})-g(X_{\tau})}\bigr]\;\le\;
\Exp\,\bigl[\bm1_{\{\tau< T\}}\,
\babs{g(X_{\tau^{n}\wedge T})-g(X_{\tau\wedge T})}\bigr]\\[5pt]
+\Exp\,\bigl[\bm1_{\{\tau< T\}}\,
\babs{g(X_{\tau^{n}})-g(X_{\tau^{n}\wedge T})}\bigr]\,.
\end{multline}
The first term on the right hand side of \eqref{ee3.24}
tends to $0$, as $n\to\infty$, by \eqref{ee3.21}, Lemma~\ref{L3.7},
and the continuity and boundedness of $g$.
The second term also tends to $0$ as $T\to\infty$, uniformly in $n$,
by Lemma~\ref{L3.1}.
Combining the above, we obtain
\begin{equation}\label{ee3.25}
\Exp\,\bigl[\babs{g(X^{n}_{\tau^{n}})-g(X_{\tau})}\bigr]
\;\xrightarrow[T\to\infty]{}\;0\,.
\end{equation}
By \eqref{E3.1}, \eqref{ee3.22} and \eqref{ee3.25}, it follows
that $u(x_{n})\to u(x)$, as $n\to \infty$, which shows that
$u$ is continuous.

Next we show that $u$ is a viscosity solution to \eqref{ee3.13}.
By the strong Markov property of $X$, for any
$t\ge 0$, we have
\begin{equation}\label{ee3.26}
u(x)\;=\;
\Exp_{x}\biggl[\int_{0}^{\tau(D)\wedge t}f(X_{s})\,\D{s}
+u(X_{\tau(D)\wedge t})\biggr]\,.
\end{equation}
Let $\varphi\in C^{2}_{b}(\Rd)$ be such that $\varphi(x)=u(x)$ and
$\varphi(z)>u(z)$ for all $z\in\Rd\setminus\{x\}$.
Then by \eqref{ee3.26} and It\^o's formula we have
\begin{align*}
\Exp_{x}\biggl[\int_{0}^{\tau(D)\wedge t}\I\varphi(X_{s})\,\D{s}\biggr]
& \;=\;\Exp_{x}\bigl[\varphi(X_{\tau(D)\wedge t})\bigr]-\varphi(x)
\\[5pt]
&\;\ge\; \Exp_{x}\bigl[u(X_{\tau(D)\wedge t})\bigr]-u(x)\\[5pt]
&\;=\;-\,\Exp_{x}\biggl[\int_{0}^{\tau(D)\wedge t}f(X_{s})\,\D{s}\biggr]\,.
\end{align*}
Dividing both sides by $t$ and letting $t\to 0$ we obtain
$\I\varphi(x)\ge -f(x)$ and thus $u$ is a sub-solution.
Similarly we can show that $u$ is super-solution and so is a viscosity solution. 
\end{proof}

\smallskip
The following theorem proves the stability of the
viscosity solutions over a convergent sequence of domains.

\begin{theorem}
Let $D_{n}$, $D$ be a collection of $C^{1}$ domains such that $D_{n}\to D$ in the
Hausdorff topology, as $n\to\infty$. Let $f, \, g\in C_b(\Rd)$.
Then $u_{n}\to u$, as $n\to\infty$,
where $u_{n}$ and $u$ are the viscosity solutions of
\eqref{ee3.13} in $D_{n}$ and $D$, respectively.
\end{theorem}

\begin{proof}
We only need to establish that for any $T>0$,
$\tau(D_{n})\wedge T\to\tau(D)\wedge T$
with probability $1$, and that $X_{\tau(D_{n})}\to X_{\tau(D)}$
on $\{\tau(D)<T\}$, as $n\to\infty$.
This can be shown following the same argument as in the proof of
Theorem~\ref{T3.2}.
\end{proof}

\section{The Harnack property for operators containing a drift term}
\label{sec-harnack}
In this section, we prove a Harnack inequality for harmonic functions.
The classes of operators considered are summarized in the following definition.

\begin{definition}\label{D4.1}
With $\lambda$ as in Definition~\ref{D-hkernel},
we let $\fL(\lambda)$,
denote the class of operators $\I\in\fL$ satisfying
\begin{equation*}
\abs{b(x)}\;\le\;\lambda_{D}\,,\quad\text{and}\quad
\lambda^{-1}_{D}\;\le\; k(x,z)\;\le\; \lambda_{D} \qquad\forall\,  x\in D\,,~
z\in\Rd\,,
\end{equation*}
for a bounded domain $D$.
As in Definition~\ref{D3.2},
the subclass of $\fL(\lambda)$ consisting of those $\I$ satisfying
$k(x,z)=k(x, -z)$ is denoted by $\fLsym(\lambda)$.
Also by $\mathfrak{I}_{\alpha}(\theta,\lambda)$ we denote the subset of
$\fL(\lambda)$ satisfying
\begin{equation*}
\babss{\int_{\Rds} \bigl(\abs{z}^{\alpha-\theta}\wedge1\bigr)\,\,
\frac{\abs{k(x,z)-k(x,0)}}{\abs{z}^{d+\alpha}}\,\D{z}}
\;\le\;\lambda_{D}\qquad \forall x\in D\,,
\end{equation*}
for any bounded domain $D$.
\end{definition}

A measurable function $h:\Rd\to\R$ is said to be \emph{harmonic} with
respect to $\I$ in a domain $D$
if for any bounded subdomain $G\subset D$, it satisfies
\begin{equation*}
h(x)\;=\;\Exp_{x}[h(X_{\tau(G)})] \qquad \forall\, x\in G\,,
\end{equation*}
where $(X, \Prob_{x})$ is a strong Markov process associated with
$\I$.
 
\begin{theorem}\label{T4.1}
Let $D$ be a bounded domain of $\Rd$ and $K\subset D$ be compact.
Then there exists a constant $C_{H}$ depending on $K$, $D$ and $\lambda$,
such that any bounded, nonnegative function which
is harmonic in $D$ with respect to an operator
$\I\in\fLsym(\lambda)\cup\mathfrak{I}_{\alpha}(\theta,\lambda)$,
$\theta \in (0, 1)$,
satisfies
\begin{equation*}
h(x)\;\le\; C_{H}\,  h(y)\qquad \text{for all}~ x, y\in K\,. 
\end{equation*}
\end{theorem}

We prove Theorem~\ref{T4.1} by verifying
the conditions in \cite{song-vondracek-04} where
a Harnack inequality is established for a general class of Markov processes.
We accomplish this through Lemmas~\ref{L4.1}--\ref{L4.3} which follow.
Let us also mention that some of the proof techniques are
standard but we still add them for clarity.
In fact, the Harnack property with non-symmetric kernel is also discussed
in \cite{song-vondracek-04} under some regularity condition on $k(\cdot,\cdot)$
and under the assumption of the existence of a harmonic measure.
Our proof of Lemma~\ref{L4.1}(b) which follows holds under
very general conditions, and does not rely on the existence of a harmonic measure.
In Lemmas~\ref{L4.1}--\ref{L4.3}, $(X, \Prob_{x})$ is a
strong Markov process associated with
$\I\in\fLsym(\lambda)\cup\mathfrak{I}_{\alpha}(\theta,\lambda)$,
and $D$ is a bounded domain.

\begin{lemma}\label{L4.1}
Let $D$ be a bounded domain.
There exist positive constants $\kappa_{2}$ and $r_{0}$ such that for any
$x\in D$ and $r\in (0,r_{0})$,
\begin{itemize}
\item[(a)]
$\inf_{z\in B_{\frac{r}{2}}(x)}\Exp_{z}[\tau(B_{r}(x))]
\ge \kappa^{-1}_{2}r^{\alpha}$,
\item[(b)]
$\sup_{z\in B_{r}(x)}\Exp_{z}[\tau(B_{r}(x))]\le \kappa_{2}\, r^{\alpha}$.
\end{itemize}
\end{lemma}

\begin{proof}
By  Lemma~\ref{L3.5} and Remark~\ref{R3.2} there exists
a constant $\kappa_{1}$ such that
\begin{equation}\label{ee4.1}
\Prob_{x}(\tau(B_{r}(x))\le t)\;\le\; \kappa_{1} t r^{-\alpha}\,,
\end{equation}
for all $t\ge 0$, and all $x\in D_{2}\df\{y: \,\dist(y, D)< 2\}$.
We choose
$t=\frac{r^{\alpha}}{2\kappa_{1}}$.
Then for $z\in B_{\frac{r}{2}}(x)$, we obtain by \eqref{ee4.1} that
\begin{align*}
\Exp_{z}[\tau(B_{r}(x))]&\;\ge\; \Exp_{z}[\tau(B_\frac{r}{2}(z))]\\[5pt]
&\;\ge\; \frac{r^{\alpha}}{2\kappa_{1}}\,
\Prob_{z}\Bigl(\tau(B_\frac{r}{2}(z))>
\frac{r^{\alpha}}{2\kappa_{1}}\Bigr)
\\[5pt]
&\;\ge\; \frac{r^{\alpha}}{4\kappa_{1}}\,.
\end{align*}
This proves the part~(a).

To prove part~(b) we consider 
a radially non-decreasing function $f\in C^{2}_{b}(\Rd)$, which is convex
in $B_{4}$, and satisfies
\begin{equation*}
f(x+z)-f(x)-z\cdot\grad f(x)\;\ge\; c_{1}\abs{z}^{2}
\qquad \text{for}~\abs{x}\le 1\,,\;\abs{z}\le 3\,,
\end{equation*}
for some positive constant $c_{1}$.
For an arbitrary point $x_{0}\in D$, define
$g_{r}(x)\df f(\frac{x-x_{0}}{r})$.
Then for $x\in B_{r}(x_{0})$ and $\I\in\fLsym(\lambda)$ we have
\begin{align*}
\int_{\Rds}\dd g_{r}(x;z)\,\frac{k(x,z)}{\abs{z}^{\alpha+d}}\,\D{z}
&\;=\;\int_{\abs{z}\le 3r}\bigl(g_{r}(x+z)-g_{r}(x)-
z\cdot\grad g_{r}(x)\bigr)
\frac{k(x,z)}{\abs{z}^{\alpha+d}}\,\D{z}\\
&\mspace{200mu}+
\int_{\abs{z}> 3r}\bigl(g_{r}(x+z)-g_{r}(x)\bigr)
\frac{k(x,z)}{\abs{z}^{\alpha+d}}\,\D{z}
\\[5pt]
&\;\ge\;  \frac{c_{1}}{r^{2}}\,\lambda^{-1}_{D}
 \int_{\abs{z}\le 3r} \abs{z}^{2-d-\alpha}\,\D{z}\\[5pt]
&\;=\; c_{2}\,\frac{3^{2-\alpha}}{2-\alpha}\, \lambda^{-1}_{D}\,r^{-\alpha}
\end{align*}
for some constant $c_{2}>0$,
where in the first equality we use the fact that $k(x,z)=k(x, -z)$, and
for the second inequality we use the property that $g(x+z)\ge g(x)$ 
for $\abs{z}\ge 3r$.
It follows that we may choose $r_{0}$ small enough such that 
\begin{equation*}
\I g_{r}(x)\;\ge\; c_{3} r^{-\alpha}
\qquad\text{for all}~r\in (0, r_{0})\,,~x\in B_{r}(x_{0})\,,~\text{and}~
x_{0}\in D\,,
\end{equation*}
with $c_{3}\df \frac{c_{2}}{2}\,\frac{3^{2-\alpha}}{2-\alpha}\,\lambda^{-1}_{D}$.

To obtain a similar estimate for $\I\in\mathfrak{I}_{\alpha}(\theta,\lambda)$
we fix some $\theta_1\in(0, \theta\wedge(\alpha-1))$.
Let $\Hat{k}(x,z)\df k(x,z)-k(x,0)$.
We have
\begin{align*}
\int_{\Rds}\dd g_{r}(x;z)\,\frac{k(x,z)}{\abs{z}^{\alpha+d}}\,\D{z}
&\;=\;\int_{\abs{z}\le 3r}\dd g_{r}(x;z)
\frac{k(x,z)}{\abs{z}^{\alpha+d}}\,\D{z}
 -\int_{3r<\abs{z}<1}z\cdot\grad g_{r}(x)\frac{k(x,z)-k(x,0)}{\abs{z}^{d+\alpha}}\,
 \D{z}\\[5pt]
&\mspace{150mu} +\int_{\abs{z}> 3r}\bigl(g_{r}(x+z)-g_{r}(x)\bigr)
\frac{k(x,z)}{\abs{z}^{\alpha+d}}\,\D{z}\\[5pt]
&\ge\; \frac{c_{1}}{\lambda_{D}\,r^{2}}
\int_{\abs{z}\le 3r} \abs{z}^{2-d-\alpha}\,\D{z}
- \frac{\norm{\grad f}_{\infty}}{r}\int_{3r<\abs{z}<1}\abs{z}
\frac{\abs{\Hat{k}(x,z)}}{\abs{z}^{d+\alpha}}\, \D{z}
\\[5pt]
&\ge\; c_{2}\,\frac{3^{2-\alpha}}{(2-\alpha)\,\lambda_{D}\,r^{\alpha}}
 -\frac{\norm{\grad f}_{\infty}}{r}\int_{3r<\abs{z}<1}
\abs{z}^{\alpha-\theta_{1}}(3r)^{-\alpha+\theta_{1}+1}
\frac{\abs{\Hat{k}(x,z)}}{\abs{z}^{d+\alpha}}\, \D{z}\\[5pt]
&\ge\; c_{2}\,\frac{3^{2-\alpha}}{(2-\alpha)\,\lambda_{D}\,r^{\alpha}}
-\frac{\norm{\grad f}_{\infty}}{r}\int_{3r<\abs{z}<1}
\abs{z}^{\alpha-\theta}(3r)^{-\alpha+\theta_{1}+1}
\frac{\abs{\Hat{k}(x,z)}}{\abs{z}^{d+\alpha}}\, \D{z}\\[5pt]
&\ge\; c_{2}\,\frac{3^{2-\alpha}}{(2-\alpha)\,\lambda_{D}\,r^{\alpha}}
- \kappa(d) 3^{\alpha-\theta_{1}+1}r^{-\alpha+\theta_{1}}\,\lambda_{D}\,
\norm{\grad f}_{\infty}  
\\[5pt]
&\ge\; c_{4}\,r^{-\alpha} \qquad \forall\, x\in B_r(x_{0})\,,
\end{align*}
for some constant $c_{4}>0$ and $r$ small,
where in the third inequality we used the fact that 
$\theta_{1}<\alpha-1$.
Thus by It\^{o}'s formula we obtain
\begin{equation*}
\Exp_{x}\bigl[\tau(B_{r}(x_{0}))\bigr]\;\le\; c_{4}^{-1}
r^{\alpha}\norm{f}_{\infty} \qquad \forall\, x\in B_r(x_{0})\,.
\end{equation*}
This completes the proof.
\end{proof}

\begin{lemma}\label{L4.2}
There exists a constant $\kappa_{3}>0$ such that for any $r\in(0, 1)$, $x\in D$ and
$A\subset B_{r}(x)$ we have
\begin{equation*}
\Prob_{z}\bigl(\tau(A^{c})
\,<\,\tau(B_{3r}(x))\bigr)\;\ge\; \kappa_{3}\,\frac{|A|}{|B_{r}(x)|}
\qquad \forall\, z\in B_{2r}(x)\,.
\end{equation*}
\end{lemma}

\begin{proof}
Let $\Hat\tau\df\tau(B_{3r}(x))$.
Suppose $\Prob_{z}(\tau(A^{c})<\Hat\tau)<\nicefrac{1}{4}$
for some $z\in B_{2r}(x)$.
Otherwise there is nothing to prove as $\frac{|A|}{|B_{r}(x)|}\le 1$.
By Lemma~\ref{L3.5} and Remark~\ref{R3.2} there exists $t>0$ such that
$\Prob_{y}(\Hat\tau\le t r^{\alpha})\le \nicefrac{1}{4}$
for all $y\in B_{2r}(x)$.
Hence using the L\'{e}vy-system formula we obtain
\begin{align}\label{ee4.2}
\Prob_{y}(\tau(A^{c})<\Hat\tau)&\;\ge\;
\Exp_{y}\Biggl[\sum_{s\le \tau(A^{c})\wedge\Hat\tau\wedge t r^{\alpha}}
\bm1_{\{X_{s-}\neq X_{s}, X_{s}\in A\}}\Biggr]\nonumber\\[5pt]
&\;=\;\Exp_{y}\biggl[\int_{0}^{\tau(A^{c})\wedge\Hat\tau\wedge t r^{\alpha}}
\int_{A}
\frac{k(X_{s}, z-X_{s})}{|z-X_{s}|^{d+\alpha}}\,\D{z}\,\D{s}\biggr]
\nonumber\\[5pt]
& \;\ge\; \Exp_{y}\biggl[\int_{0}^{\tau(A^{c})\wedge\Hat\tau\wedge t r^{\alpha}}
\int_{A}\frac{\lambda^{-1}_{D}}{(4r)^{d+\alpha}}\,\D{z}\,\D{s}\biggr]
\nonumber\\[5pt]
&\;\ge\; \kappa'_{3}\, r^{-\alpha}\frac{|A|}{|B_{r}(x)|}
\Exp_{y}[\tau(A^{c})\wedge\Hat\tau\wedge t r^{\alpha}]
\end{align}
for some constant $\kappa'_{3}>0$,
where in the third inequality we use the fact that $|X_{s}-z|\le 4r$ for
$s< \Hat\tau$, $z\in A$.
On the other hand, we have
\begin{align}\label{ee4.3}
\Exp_{y}[\tau(A^{c})\wedge\Hat\tau\wedge t r^{\alpha}]&\;\ge\; t\, r^{\alpha}\,
\Prob_{y}(\tau(A^{c})\ge\Hat\tau\ge t r^{\alpha})
\nonumber\\[5pt]
&\;=\; t\, r^{\alpha}\,\bigl[1-\Prob_{y}(\tau(A^{c})<\Hat\tau)
-\Prob_{y}(\Hat\tau<t r^{\alpha})\bigr]\nonumber\\[5pt]
&\;\ge\; \frac{t}{2}\,r^{\alpha}\,.
\end{align}
Therefore combining \eqref{ee4.2}--\eqref{ee4.3}, we obtain
$\Prob_{z}(\tau(A^{c})<\Hat\tau)\ge
\frac{t\kappa'_{3}}{2}\frac{|A|}{|B_{r}(x)|}$.
\end{proof}

\begin{lemma}\label{L4.3}
There exists positive constants $\kappa_{i}$, $i=4, 5$, such that if $x\in D$,
$r\in (0, 1)$, $z\in B_{r}(x)$, and $H$ is a bounded nonnegative function with
support in $B_{2r}^{c}(x)$, then
\begin{equation*}
\Exp_{z}\bigl[H(X_{\tau(B_{r}(x)})\bigr]\;\le\; \kappa_{4}\,
\Exp_{z}\bigl[\tau(B_{r}(x)\bigr]
\int_{\Rds} H(y)\frac{k(x, y-x)}{|y-x|^{d+\alpha}}\,\D{y}\,,
\end{equation*}
and
\begin{equation*}
\Exp_{z}\bigl[H(X_{\tau(B_{r}(x)})\bigr]\;\ge\; \kappa_{5}\,
\Exp_{z}\bigl[\tau(B_{r}(x)\bigr]
\int_{\Rds} H(y)\frac{k(x, y-x)}{|y-x|^{d+\alpha}}\,\D{y}\,.
\end{equation*} 
\end{lemma}

The proof follows using the same argument as in
\cite[Lemma~3.5]{song-vondracek-04}.

\medskip

\begin{proof}[Proof of Theorem~\ref{T4.1}]
By Lemmas~\ref{L4.1}, \ref{L4.2}
and \ref{L4.3}, the hypotheses (A1)--(A3) in
\cite{song-vondracek-04} are satisfied.
Hence the proof follows from \cite[Theorem~2.4]{song-vondracek-04}.
\end{proof}

\section{Positive recurrence and invariant probability measures}
\label{sec-stability}
In this section we study the recurrence properties for a Markov
process with generator $\I\in\fL$ (see Definition~\ref{D3.2} and ~\ref{D4.1}).
Many of the results of this section
are based on the assumption of the existence of a
\emph{Lyapunov function}.

\begin{definition}
We say that the operator $\I$ of the form
\eqref{gen-op} satisfies the Lyapunov stability
condition if there exists a $\calV\in C^{2}(\Rd)$ such that
$\inf_{x\in\Rd}\calV(x)>-\infty$, and for some compact
set $\cK\subset\Rd$ and $\varepsilon>0$, we have
\begin{equation}\label{E5.01}
\I\,\calV(x)\;\le\; -\varepsilon \qquad\forall\, x\in \cK^{c}.
\end{equation}
\end{definition}

It is straightforward to verify that if $\calV$ satisfies \eqref{E5.01} for
$\I\in\fL$, then
\begin{equation}\label{E5.02}
\int_{\abs{z}\ge 1}|\calV(z)|\frac{1}{\abs{z}^{d+\alpha}}\,\D{z}<\infty\,.
\end{equation}

\begin{proposition}
If there exists a constant $\gamma\in (1,\alpha)$ such that
\begin{equation*}
\frac{b(x)\cdot x}{\abs{x}^{2-\gamma}\,\sup_{z\in\Rd} k(x,z)\vee 1}\;
\xrightarrow[\abs{x}\to\infty]{}\;-\infty\,,
\end{equation*}
then the operator $\I$ satisfies
the Lyapunov stability condition.
\end{proposition}

\begin{proof}
Consider a nonnegative function $f\in C^{2}(\Rd)$ such that
$f(x)=\abs{x}^\gamma$ for $\abs{x}\ge 1$,
and let $\Bar{k}(x) \df \sup_{z\in\Rd} k(x,z)$.
Since the second derivatives of $f$ are bounded in $\Rd$, and
$k$ is also bounded, it follows that
\begin{equation*}
\babss{\int_{\abs{z}\le1}\dd f(x;z)\,\frac{k(x,z)}{\abs{z}^{d+\alpha}}\,\D{z}}
\;\le\; \kappa_{1}\,\Bar{k}(x)
\end{equation*}
for some constant $\kappa_{1}$ which depends on the bound of the trace
of the Hessian of $f$.
Following the same steps as in the proof of \eqref{ee3.6},
and using the fact that
$k$ is bounded in $\Rd\times\Rd$, we obtain
\begin{equation}\label{E5.03}
\babss{\int_{\abs{z}>1} (\abs{x+z}^\gamma -\abs{x}^\gamma)\,
\frac{k(x,z)}{\abs{z}^{d+\alpha}}\,\D{z}}
\;\le\; \kappa_{2}\,\Bar{k}(x)\,
(1+\abs{x}^{\gamma-\alpha})\qquad \text{if~} \abs{x}>1\,,
\end{equation}
for some constant $\kappa_{2}>0$.
Since also,
\begin{equation}\label{E5.04}
\babss{\int_{\Rds}\bm1_{B_{1}}(x+z)\,
\frac{k(x,z)}{\abs{z}^{d+\alpha}}\,\D{z}}
\;\le\; \kappa_{3}\,\Bar{k}(x)\,(\abs{x}-1)^{-\alpha}\qquad\text{for~} \abs{x}>2\,,
\end{equation}
for some constant $\kappa_{3}$,
it follows by the above that
\begin{equation}\label{E5.05}
\babss{\int_{\Rds}\dd f(x;z)\,\frac{k(x,z)}{\abs{z}^{d+\alpha}}\,\D{z}}
\;\le\; \kappa_{4}\,\Bar{k}(x)\,(1+\abs{x}^{\gamma-\alpha})\qquad\forall x\in\Rd\,,
\end{equation}
for some constant $\kappa_{4}$.
Therefore by the hypothesis and \eqref{E5.05}, it follows that
$\I f(x)\to-\infty$ as $\abs{x}\to\infty$.
\end{proof}

\begin{lemma}\label{L5.1}
Let $X$ be the Markov process associated with a generator
$\I\in\fL(\lambda)$, and suppose that $\I$ satisfies the Lyapunov stability
hypothesis \eqref{E5.01} and the growth condition in \eqref{E3.5}.
Then for any $x\in \cK^{c}$ we have
\begin{equation*}
\Exp_{x}[\tau(\cK^{c})]
\;\le\;\frac{2}{\varepsilon}\,\bigl(\calV(x)+(\inf\calV)^-\bigr)\,.
\end{equation*}
\end{lemma}

\begin{proof}
Let $R_{0}>0$ be such that $\cK\subset B_{R_{0}}$.
We choose a cut-off function $\chi$ which equals $1$ on $B_{R_{1}}$, with
$R_{1}>2R_{0}$, vanishes outside of $B_{R_{1}+1}$, and $\norm{\chi}_{\infty}=1$.
Then $f\df\chi\calV$ is in $C^{2}_{b}(\Rd)$.
Clearly if $\abs{x}\;\le\; R_{0}$ and $\abs{x+z}\;\ge\; R_{1}$, then
$\abs{z}\;>\; R_{0}$, and thus $\abs{x+z}\;\le\; 2\abs{z}$.
Therefore, for large enough $R_{1}$, we obtain
\begin{align*}
\babss{\int_{\Rds}\bigl(f(x+z)-\calV(x+z)\bigr)
\frac{k(x,z)}{\abs{z}^{d+\alpha}}\,\D{z}}
&\;\le\; 2\int_{\{\abs{x+z}\ge R_{1}\}}|\calV(x+z)|\,
\frac{k(x,z)}{\abs{z}^{d+\alpha}}\,\D{z}
\\[5pt]
&\;\le\; 2^{d+\alpha+1}\lambda_{B_{R_{0}}}
\int_{\{\abs{x+z}\ge R_{1}\}}|\calV(x+z)|\,\frac{1}{|x+z|^{d+\alpha}}\,\D{z}
\\[5pt]
&\;\le\; \frac{\varepsilon}{2}\qquad \forall\, x\in B_{R_{0}}\,.
\end{align*}
Hence, for all $R_{1}$ large enough, we have
\begin{equation*}
\I f(x)\;\le\;-\frac{\varepsilon}{2}
\qquad \forall x\in B_{R_{0}}\setminus \cK\,.
\end{equation*}
Let $\widetilde\tau_{R}=\tau(\cK^{c})\wedge\tau(B_{R})$.
Then applying It\^{o}'s formula we obtain
\begin{equation*}
\Exp_{x}\bigl[\calV(X_{\widetilde\tau_{R_{0}}})\bigr]-\calV(x)\;\le\;
-\frac{\varepsilon}{2}\,\Exp_{x}[\widetilde\tau_{R_{0}}]\qquad\forall\,
x\in B_{R_{0}}\setminus \cK\,,
\end{equation*}
implying that
\begin{equation}\label{E5.06}
\Exp_{x}[\widetilde\tau_{R_{0}}]\;\le\;
\frac{2}{\varepsilon}\,\bigl(\calV(x)+(\inf\calV)^-\bigr)\,.
\end{equation}
By the growth condition and Lemma~\ref{L3.3},
$\tau(B_{R})\to\infty$ as $R\to\infty$ with probability $1$.
Hence the result follows by applying Fatou's lemma to \eqref{E5.06}.
\end{proof}

\subsection{Existence of invariant probability measures}

Recall that a Markov process is said be to positive recurrent
if for any compact set $G$ with positive Lebesgue measure it holds that
$\Exp_{x}[\tau(G^{c})]<\infty$ for any $x\in\Rd$.
We have the following theorem. 

\begin{theorem}\label{T5.1}
If $\I\in\fL(\lambda)$ satisfies the Lyapunov stability hypothesis,
and the growth condition in \eqref{E3.5}, then the associated Markov
process is positive recurrent.
\end{theorem}

\begin{proof}
First we note that if the Lyapunov condition is satisfied for some compact set
$\cK$,
then it is also satisfied for any compact set containing $\cK$.
Hence we may assume that $\cK$
is a closed ball centered at origin.
Let $D$ be an open ball with center at origin and containing $\cK$.
We define
\begin{equation*}
\Hat\tau_{1}\;\df\;\inf\;\{t\ge 0\,:\, X_{t}\notin D\}\,,\qquad
\Hat\tau_{2}\;\df\;\inf\;\{t>\tau\,:\, X_{t}\in \cK\}\,.
\end{equation*}
Therefore for $X_{0}=x\in \cK$, $\Hat\tau_{2}$ denotes the first return time to
$\cK$ after hitting $D^{c}$.
First we prove that
\begin{equation}\label{E5.07}
\sup_{x\in \cK}\;\Exp_{x}[\Hat\tau_{2}]\;<\;\infty\,.
\end{equation}
By Lemma~\ref{L5.1} we have
$\Exp_{x}[\tau(\cK^{c})]\le\frac{2}{\varepsilon}[\calV(x)+(\inf\calV)^-]$
for $x\in \cK^{c}$.
By Lemma~\ref{L3.1} we have
$\sup_{x\in \cK}\Exp_{x}[\Hat\tau_{1}]<\infty$. 
Let $\scrP_{\Hat\tau_{1}}(x,\cdot\,)$ denote
the exit distribution of the process $X$ starting from $x\in \cK$.
In order to prove \eqref{E5.07} it suffices to show that
\begin{equation*}
\sup_{x\in \cK}\;\int_{D^{c}}
\bigl(\calV(y)+(\inf\calV)^-\bigr)\,\scrP_{\Hat\tau_{1}}(x,\D{y})\;<\;\infty\,,
\end{equation*}
and since $\calV$ is locally bounded it is enough that
\begin{equation}\label{E5.08}
\sup_{x\in \cK}\;\int_{B_{R}^{c}}\bigl(\calV(y)+(\inf\calV)^-\bigr)\,
\scrP_{\Hat\tau_{1}}(x,\D{y})\;<\;\infty
\end{equation}
for some ball $B_{R}$.
To accomplish this we choose $R$ large enough so that
\begin{equation*}
\frac{|x-z|}{\abs{z}}\;>\;\frac{1}{2}\qquad\text{for}~ \abs{z}\ge R\,,\; x\in D\,.
\end{equation*}
Then, for any Borel set $A\subset B_{R}^{c}$, by
Proposition~\ref{levy-system} we have that
\begin{align*}
\Prob_{x}(X_{\Hat\tau_{1}\wedge t}\in A)
&\;=\;\Exp_{x}\Biggl[\sum_{s\le \Hat\tau_{1}\wedge t}\bm1_{\{X_{s-}\in
D,\, X_{s}\in A\}}\Biggr]
\\[5pt]
&\;=\;\Exp_{x}\biggl[\int_{0}^{\Hat\tau_{1}\wedge t}\bm1_{\{X_{s}\in D\}}
\int_{A}\frac{k(X_{s}, z-X_{s})}{|X_{s}-z|^{d+\alpha}}\,\D{z}\,\D{s}\biggr]
\\[5pt]
&\;\le\; 2^{d+\alpha}\lambda_{D}\,
\Exp_{x}\biggl[\int_{0}^{\Hat\tau_{1}\wedge t}\int_{A}
\frac{1}{\abs{z}^{d+\alpha}}\,\D{z}\,\D{s}\biggr]
\\[5pt]
&\;=\; 2^{d+\alpha}\lambda_{D}\,\Exp_{x}[\Hat\tau_{1}\wedge t]\,\mu(A)\,,
\end{align*} 
where $\mu$ is the $\sigma$-finite measure on
$\Rd_*$ with density $\frac{1}{\abs{z}^{d+\alpha}}$.
Thus letting $t\to\infty$ we obtain
\begin{equation*}
\scrP_{\Hat\tau_{1}}(x, A)\;\le\;
2^{d+\alpha}\lambda_{D}\,\biggl(\sup_{x\in \cK}\;\Exp_{x}[\Hat\tau_{1}]\biggr)
\,\mu(A)\,.
\end{equation*}
Therefore, using a standard approximation argument, we deduce that for any
nonnegative function $g$ it holds that
\begin{equation*}
\int_{B_{R}^{c}}g(y)\,\scrP_{\Hat\tau_{1}}(x, \D{y})\;\le\;
\Tilde\kappa \int_{B_{R}^{c}}g(y)\mu(\D{y})
\end{equation*}
for some constant $\Tilde\kappa$.
This proves \eqref{E5.08} since $\calV$
is integrable on $B_{R}^{c}$ with respect to $\mu$ and $\mu(B_{R}^{c})<\infty$. 

Next we prove that the Markov process is positive recurrent.
We need to show that for any compact set $G$ with positive Lebesgue measure,
$\Exp_{x}[\tau(G^{c})]<\infty$ for any $x\in\Rd$.
Given a compact $G$ and $x\in G^{c}$ we choose a closed ball
$\cK$, which satisfies the Lyapunov condition relative to $\calV$,
and such that $G\cup\{x\}\subset \cK$.
Let $D$ be an open ball containing $\cK$.
We define a sequence of stopping times $\{\Hat\tau_{k}\,,\; k=0,1,\dotsc\}$
as follows:
\begin{align*}
\Hat\tau_{0} &\;=\;0\\[3pt]
\Hat\tau_{2n+1}&\;=\;\inf\{t>\Hat\tau_{2n}:\,X_{t}\notin D\}\,,\\[3pt]
\Hat\tau_{2n+2}&\;=\;\inf\{t>\Hat\tau_{2n+1}:\,X_{t}\in \cK\}\,,
\quad n=0,1,\dotsc.
\end{align*}
Using the strong Markov property and \eqref{E5.08},
we obtain $\Exp_{x}[\Hat\tau_{n}]<\infty$ for all $n\in\NN$.
From Lemma~\ref{L3.5} there exist positive constants $t$ and $r$ such that 
\begin{equation*}
\sup_{x\in\cK}\;\Prob_{x}(\tau(D)<t)\;\le\;
\sup_{x\in\cK}\;\Prob_{x} (\tau(B_r(x))< t)\;\le\; \frac{1}{4}\,.
\end{equation*}
Therefore, using a similar argument as in Lemma~\ref{L4.2}, we can find a
constant $\delta>0$
such that
\begin{equation*}
\inf_{x\in \cK}\;\Prob_{x}(\tau(G^{c})<\tau(D))\;>\;\delta\,.
\end{equation*}
Hence
\begin{equation*}
p\;\df\;\sup_{x\in \cK}\;\Prob_{x}(\tau(D)<\tau(G^{c}))
\;\le\; 1-\delta\;<\;1\,.
\end{equation*}
Thus by the strong Markov property we obtain
\begin{equation*}
\Prob_{x}(\tau(G^{c})>\Hat\tau_{2n})\;\le\;
p\,\Prob_{x}(\tau(G^{c})>\Hat\tau_{2n-2})\;\le\;\dotsb\;\le\; p^{n}
\qquad \forall\,x\in \cK\,.
\end{equation*}
This implies $\Prob_{x}(\tau(G^{c})<\infty)=1$.
Hence, for $x\in \cK$, we obtain
\begin{align*}
\Exp_{x}[\tau(G^{c})]&\;\le\; \sum_{n=1}^{\infty}
\Exp_{x}\bigl[\Hat\tau_{2n}
\bm1_{\{\Hat\tau_{2n-2}<\tau(G^{c})\le\Hat\tau_{2n}\}}\bigr]\\[5pt]
&\;=\; \sum_{n=1}^{\infty}\sum_{l=1}^{n}
\Exp_{x}\bigl[(\Hat\tau_{2l}- \Hat\tau_{2l-2})
\bm1_{\{\Hat\tau_{2n-2}<\tau(G^{c})\le\Hat\tau_{2n}\}}\bigr]\\[5pt]
&\;=\; \sum_{l=1}^{\infty}\sum_{n=l}^{\infty}
\Exp_{x}\bigl[(\Hat\tau_{2l}- \Hat\tau_{2l-2})
\bm1_{\{\Hat\tau_{2n-2}<\tau(G^{c})\le\Hat\tau_{2n}\}}\bigr]
\\[5pt]
&\;=\; \sum_{l=1}^{\infty}\Exp_{x}\bigl[(\Hat\tau_{2l}- \Hat\tau_{2l-2})
\bm1_{\{\Hat\tau_{2l-2}<\tau(G^{c})\}}\bigr]\\[5pt]
&\;\le\; \sum_{l=1}^{\infty} p^{l-1}\sup_{x\in \cK}\;\Exp_{x}[\Hat\tau_{2}]\\[5pt]
&\;=\;\frac{1}{1-p}\,\sup_{x\in \cK}\;\Exp_{x}[\Hat\tau_{2}]\;<\;\infty\,.
\end{align*}
Since also $\Exp_{x}[\tau(\cK^{c})]<\infty$ for all $x\in\Rd$,
this completes the proof.
\end{proof}

\begin{theorem}\label{T5.2}
Let $X$ be a Markov process associated with a generator
$\I\in\fLsym(\lambda)\cup\mathfrak{I}_{\alpha}(\theta,\lambda)$,
and suppose that
the Lyapunov stability hypothesis \eqref{E5.01}
 and the growth condition in \eqref{E3.5} hold.
Then $X$ has an invariant probability measure.
\end{theorem}

\begin{proof}
The proof is based on Has$'$minski\u{\i}'s construction.
Let $\cK$, $D$, $\Hat\tau_{1}$, and $\Hat\tau_{2}$ be as in the proof
of Theorem~\ref{T5.1}.
Let $\Hat X$ be a Markov process on $\cK$ with transition kernel given by
\begin{equation*}
\Hat\Prob_{x}(\D{y})\;=\;\Prob_{x}(X_{\Hat\tau_{2}}\in \D{y})\,.
\end{equation*}
Let $f$ be any bounded, nonnegative measurable function on $D$.
Define $Q_{f}(x)=\Exp_{x}[f(X_{\Hat\tau_{2}})]$. 
We claim that $Q_{f}$ is harmonic in $D$.
Indeed if we define $\Tilde f(x)=\Exp_{x}[f(X_{\tau(\cK^{c})})]$ for $x\in D^{c}$,
then by the strong Markov property we obtain
$Q_{f}(x)=\Exp_{x}[\Tilde f(X_{\Hat\tau_{1}})]$, and the claim follows.
By Theorem~\ref{T4.1} there exists a positive constant $C_{H}$,
independent of $f$, satisfying
\begin{equation}\label{E5.09}
Q_{f}(x)\;\le\; C_{H} Q_{f}(y)\qquad \forall\,x, y\in \cK\,.
\end{equation}
We note that $Q_{\bm1_{\cK}}\equiv 1$.
Let $Q(x, A)\df Q_{\bm1_{A}}(x)$, for $A\subset \cK$.
For any pair of probability measures $\mu$ and $\mu'$ on $\cK$, we 
claim that
\begin{equation}\label{E5.10}
\bnormm{\int_{\cK}\bigl(\mu(\D{x})-\mu'(\D{x})\bigr)
Q(x,\cdot\,)}_{\mathrm{TV}}\;\le\;
\frac{C_{H}-1}{C_{H}}\,\norm{\mu-\mu'}_{\mathrm{TV}}\,.
\end{equation}
This implies that the map $\mu\to\int_{\cK}Q(x,\cdot\,)\mu(\D{x})$ is a 
contraction and hence it has a unique
fixed point $\Hat\mu$ satisfying $\Hat\mu(A)=\int_{\cK}Q(x, A)\Hat\mu(\D{x})$
for any Borel set $A\subset \cK$.
In fact, $\Hat\mu$ is the invariant probability measure of the Markov chain
$\Hat X$.
Next we prove \eqref{E5.10}.
Given any two probability measure $\mu$, $\mu'$ on $\cK$, we can find subsets 
$F$ and $G$ of $\cK$ such that
\begin{align*}
\bnormm{\int_{\cK}\bigl(\mu(\D{x})-\mu'(\D{x})\bigr)
Q(x,\cdot\,)}_{\mathrm{TV}}&\;=\;
2\int_{\cK}\bigl(\mu(\D{x})-\mu'(\D{x})\bigr)\,Q(x,F)\,,
\\[5pt]
\norm{\mu-\mu'}_{\mathrm{TV}} & \;=\; 2(\mu-\mu')(G)\,.
\end{align*}
In fact, the restriction of $(\mu-\mu')$ to $G$ is a nonnegative measure and
its restriction to
$G^{c}$ it is non-positive measure.
By \eqref{E5.09}, we have
\begin{equation}\label{E5.11}
\inf_{x\in G^{c}}\;Q(x, F)\;\ge\;\sup_{x\in G}\;Q(x, F)
\end{equation}
Hence, using \eqref{E5.11}, we obtain
\begin{align*}
\bnormm{\int_{\cK}\bigl(\mu(\D{x})-\mu'(\D{x})\bigr)Q(x,\cdot\,)}_{\mathrm{TV}}
&\;=\;
2\int_{G}\bigl(\mu(\D{x})-\mu'(\D{x})\bigr)Q(x, F)
+2\int_{G^{c}}\bigl(\mu(\D{x})-\mu'(\D{x})\bigr)Q(x,F)
\\[5pt]
&\;\le\; 2(\mu-\mu')(G)\,\sup_{x\in G}\;Q(x, F)
+ 2(\mu-\mu')(G^{c})\,\inf_{x\in G^{c}}\;Q(x, F)\\[5pt]
&\;\le\; 2(\mu-\mu')(G)\,\sup_{x\in G}\;Q(x, F)
- \frac{2}{C_{H}}(\mu-\mu')(G)\,\sup_{x\in G}\;Q(x, F)\\[5pt]
&\;\le\; \bigl(1-C_{H}^{-1}\bigr)\norm{\mu-\mu'}_{\mathrm{TV}} \,.
\end{align*}
This proves \eqref{E5.10}.

We define a probability measure $\nu$ on $\Rd$ as follows.
\begin{equation*}
\int_{\Rd}f(x)\,\nu(\D{x})
\;=\;\frac{\int_{\cK}\Exp_{x}\bigl[\int_{0}^{\Hat\tau_{2}}f(X_{s})\,\D{s}\bigr]
\Hat\mu(\D{x})}
{\int_{\cK}\Exp_{x}[\Hat\tau_{2}]\Hat\mu(\D{x})}\,, \qquad f\in C_{b}(\Rd)\,.
\end{equation*}
It is straight forward to verify that $\nu$ is an invariant probability
measure of $X$
(see for example, \cite[Theorem~2.6.9]{ari-bor-ghosh}).
\end{proof}

\begin{remark}
If $k(\cdot,\cdot)=1$ and the drift $b$ belongs to certain Kato class,
in particular bounded, (see \cite{bogdan-jakubowski}) then the transition
probability has a continuous density, and therefore any
invariant probability measure has a continuous density.
Since any two distinct ergodic measures are mutually singular,
this implies the uniqueness of the invariant probability measure.
As shown later in Proposition~\ref{P5.3} open sets have strictly positive
mass under any invariant measure.
\end{remark}

The following result is fairly standard. 

\begin{proposition}
Let $\I\in\fL$, and $\calV\in C^{2}(\Rd)$ be a nonnegative function satisfying
satisfying $\calV(x)\to\infty$ as $\abs{x}\to\infty$, and
$\I\,\calV\le 0$ outside some compact set $\cK$.
Let $\nu$ be an invariant probability measure of the Markov process
associated with the generator $\I$.
Then 
\begin{equation*}
\int_{\Rd}|\I\,\calV(x)|\,\nu(\D{x})\;\le\; 2\int_{\cK}|\I\,\calV(x)|\,\nu(\D{x})\,.
\end{equation*}
\end{proposition}

\begin{proof}
Let $\varphi_{n}:\R_{+}\to\R_{+}$ be a smooth non-decreasing,
concave, function such that
\begin{equation*}
\varphi_{n}(x)\;=\;\begin{cases}
x & \text{for}~x\le n\,,\\
n+\nicefrac{1}{2} & \text{for}~x\ge n+1\,.
\end{cases}
\end{equation*}
Due to concavity we have $\varphi_{n}(x)\le \abs{x}$ for all $x\in\R_{+}$.
Then $\calV_{n}(x)\df\varphi_{n}(\calV(x))$ is in $C^{2}_{b}(\Rd)$ and
it also follows that $\I\,\calV_{n}(x)\to\I\,\calV(x)$ as $n\to\infty$.
Since $\nu$ is an invariant probability measure, it holds that
\begin{equation}\label{E5.12}
\int_{\Rd}\I\,\calV_{n}(x)\,\nu(\D{x})\;=\;0\,.
\end{equation}
By concavity,
$\varphi_{n}(y)\le \varphi_{n}(x)+(y-x)\cdot\varphi_{n}'(x)$ for all
$x,y\in\R_{+}$.
Hence
\begin{align*}
\I\,\calV_{n}(x)
&\;=\; \int_{\Rds}\dd \calV_{n}(x;z)\,
\frac{k(x,z)}{\abs{z}^{d+\alpha}}\,\D{z}
+ \varphi_{n}'(\calV(x))\,b(x)\cdot\grad\calV(x)
\\[5pt]
&\;\le\; \int_{\Rds}\varphi_{n}'(\calV(x))\,\dd \calV(x;z)\,
\frac{k(x,z)}{\abs{z}^{d+\alpha}}\,\D{z}
+ \varphi_{n}'(\calV(x))\,b(x)\cdot\grad\calV(x)
\\[5pt]
&\;=\;\varphi_{n}'(\calV(x))\,\I\,\calV(x)\,,
\end{align*}
which is negative for $x\in \cK^{c}$.
Therefore using \eqref{E5.12} we obtain
\begin{align}\label{E5.13}
\int_{\Rd}|\I\,\calV_{n}(x)|\,\nu(\D{x})
&\;=\;\int_{\cK}|\I\,\calV_{n}(x)|\,\nu(\D{x})
-\int_{\cK^{c}}\I\,\calV_{n}(x)\,\nu(\D{x})
\nonumber\\[5pt]
&\;=\;\int_{\cK}|\I\,\calV_{n}(x)|\,\nu(\D{x})
+\int_{\cK}\I\,\calV_{n}(x)\,\nu(\D{x})
\nonumber\\[5pt]
&\;\le\; 2\int_{\cK}|\I\,\calV_{n}(x)|\,\nu(\D{x})\,.
\end{align}
On the other hand, with $A_{n}\df\{y\in\Rd\,\colon \calV(y)\ge n\}$,
and provided $\calV(x)<n$, we have
\begin{align*}
\abs{\I\,\calV_{n}(x)} &\;\le\; \abs{\I\,\calV(x)}
+\int_{x+z\in A_{n}} \abs{\calV(x+z)-\calV_{n}(x+z)}\,
\frac{k(x,z)}{\abs{z}^{d+\alpha}}\,\D{z}\\[5pt]
&\;\le\; \abs{\I\,\calV(x)}
+\int_{x+z\in A_{n}} \abs{\calV(x+z)}\,
\frac{k(x,z)}{\abs{z}^{d+\alpha}}\,\D{z}\,.
\end{align*}
This together with \eqref{E5.02} imply that there exists a constant
$\kappa$ such that
\begin{equation*}
\abs{\I\,\calV_{n}(x)} \;\le\; \kappa+\abs{\I\,\calV(x)}
\qquad \forall\,x\in\cK\,,
\end{equation*}
and all large enough $n$.
Therefore, letting $n\to\infty$ and using Fatou's lemma for the term on the
left hand side of \eqref{E5.13}, and the
dominated convergence theorem for the term on the right hand side,
we obtain the result.
\end{proof}

\subsection{A class of operators with variable order kernels}
It is quite evident from Theorem~\ref{T5.2} that the Harnack inequality plays
a crucial role in the analysis.
Therefore one might wish to establish positive recurrence
for an operator with a variable order kernel, and deploy
the Harnack inequality from \cite{bass-kassmann} to prove a similar result
as in Theorem~\ref{T5.2}.

\begin{theorem}\label{T5.3}
Let $\uppi:\Rd\times\Rd\to\Rd$ be a nonnegative measurable function satisfying the
following properties, for $1<\alpha'<\alpha<2$:
\begin{itemize}
\item[(a)]
There exists a constant $c_{1}>0$ such that
$\bm1_{\{\abs{z}>1\}}\uppi(x,z)\le \frac{c_{1}}{\abs{z}^{d+\alpha'}}$
for all $x\in\Rd$;
\item[(b)]
There exists a constant $c_{2}>0$ such that
\begin{equation*}
\uppi(x,z-x)\;\le\; c_{2}\, \uppi(y,y-z)\,,\quad \text{whenever}
\quad \abs{z-x}\wedge\abs{z-y}\ge 1\,,\;\abs{x-y}\le 1\,;
\end{equation*}
\item[(c)]
For each $R>0$ there exists $q_{R}>0$ such that
\begin{equation*}
\frac{q_{R}^{-1}}{\abs{z}^{d+\alpha'}}\;\le\; \uppi(x,z)
\;\le\; \frac{q_{R}}{\abs{z}^{d+\alpha}}\qquad
\forall x\in\Rd\,,\;\forall z\in B_{R}\,;
\end{equation*}
\item[(d)]
For each $R>0$ there exists $R_{1}>0$, $\sigma\in(1,2)$,
and $\kappa_{\sigma}=\kappa_{\sigma}(R,R_{1})>0$ such that
\begin{equation*}
\frac{\kappa_{\sigma}^{-1}}{\abs{z}^{d+\sigma}}\;\le\;
\uppi(x,z)\;\le\; \frac{\kappa_{\sigma}}{\abs{z}^{d+\sigma}}\qquad
\forall x\in B_{R}\,,\;\forall z\in B_{R_{1}}^{c};
\end{equation*}
\item[(e)]
There exists $\calV\in C^{2}(\Rd)$ that is bounded from below in $\Rd$,
a compact set $\cK\subset\Rd$ and a constant $\varepsilon>0$, such that 
\begin{equation*}
\int_{\Rds}\dd \calV(x;z)\,
\uppi(x,z)\,\D{z}\;<\;-\varepsilon\quad \forall x\in \cK^{c}\,.
\end{equation*}
\end{itemize}
Then the Markov process associated with the above kernel has an
invariant probability measure.
\end{theorem} 

The first three assumptions guarantee the Harnack property for
associated harmonic functions \cite{bass-kassmann}.
Then the conclusion of Theorem~\ref{T5.3} follows by using an argument
similar to the one used in the proof of
Theorem~\ref{T5.2}.

Next we present an example of a kernel $\uppi$ that
satisfies the conditions in Theorem~\ref{T5.3}.
We accomplish this by adding a non-symmetric bump
function to a symmetric kernel.

\begin{example}
Let $\varphi:\Rd\to[0,1]$ be a smooth function such that
\begin{equation*}
\varphi(x)\;=\;\begin{cases}
1 & \text{for}~\abs{x}\le \frac{1}{2}\,,\\[5pt]
0 & \text{for}~\abs{x}\ge 1\,.
\end{cases}
\end{equation*}
Define for $1< \alpha'<\beta'< \alpha<2$,
\begin{equation*}
\gamma(x,z)\;\df\;\varphi\biggl(2\frac{x+z}{1+\abs{x}}\biggr)(1-\varphi(4x))
(\alpha'-\beta')\,,
\end{equation*}
and let 
\begin{align*}
\widetilde\uppi(x,z)&\;\df\;\frac{1}{\abs{z}^{d+\beta'+\gamma(x,z)}}\,,\\
\uppi(x,z)&\;\df\;\frac{1}{\abs{z}^{d+\alpha}}+\widetilde\uppi(x,z)\,.
\end{align*}
We prove that $\uppi$ satisfies the conditions of Theorem~\ref{T5.3}.
Let us also mention that there exists a unique solution to the
martingale problem corresponding to the kernel $\uppi$ \cite{Komatsu-73, komatsu}.
We only show that conditions (b) and (e) hold.
It is straightforward to verify (a), (c) and (d).

Note that $\alpha'-\beta'\le \gamma(x,z)\le 0$ for all $x,z$.
Let $x,\, y,\, z\in\Rd$ such that $|x-z|\wedge|y-z|\ge 1$ and $\abs{x-y}\le 1$.
Then $\abs{z-y}\le 1+\abs{z-x}$.
By a simple calculation we obtain
\begin{align*}
\widetilde\uppi(x,z-x) & \;\le\;
\biggl(1+\frac{1}{\abs{z-x}}\biggr)^{d+\beta'+\gamma(x,z-x)}
\frac{1}{\abs{z-y}^{d+\beta'+\gamma(x,z-x)}}\\[5pt]
&\;\le\; 2^{d+\beta'}\frac{1}{\abs{z-y}^{d+\beta'+\gamma(y, z-y)}}
\abs{z-y}^{-\beta'(x,z-x)
+\gamma(y, z-y)}\,.
\end{align*}
Hence it is enough to show that
\begin{equation}\label{E5.14}
\abs{z-y}^{-\gamma(x,z-x)+\gamma(y, z-y)}\;<\;\varrho
\end{equation}
for some constant $\varrho$ which does not depend on $x$, $y$ and $z$.
Note that if $\abs{x}\le 2$, which implies that
$\abs{y}\le 3$, then for $\abs{z}\ge 4$ we have 
$\gamma(x,z-x)=0=\gamma(y, z-y)$.
Therefore for $\abs{x}\le 2$, it holds that
\begin{equation}\label{E5.15}
\abs{z-y}^{-\gamma(x,z-x)+\gamma(y, z-y)}\;\le\; 7^{\beta'-\alpha'}\,.
\end{equation}
Suppose that $\abs{x}\ge 2$. Then $\abs{y}\ge 1$.
Since we only need to consider the case
where $\gamma(x,z-x)\neq \gamma(y, z-y)$
we restrict our attention to $z\in\Rd$ such that $\abs{z}\le 2(1+\abs{x})$.
We obtain
\begin{align}\label{E5.16}
\log(\abs{z-y})(-\gamma(x,z-x)+\gamma(y, z-y))&\;\le\;
\log\bigl(3(1+\abs{x})\bigr)\,\norm{\varphi'}_{\infty}\,
\frac{2\abs{z}(\beta'-\alpha')}{(1+\abs{x})(1+\abs{y})}\nonumber\\[5pt]
&\;\le\; \log\bigl(3(1+\abs{x})\bigr)\,\norm{\varphi'}_{\infty}\,
\frac{4(1+\abs{x})(\beta'-\alpha')}{(1+\abs{x})\abs{x}}\,.
\end{align}
Since the term on the right hand side of \eqref{E5.16}
is bounded in $\Rd$, the bound in \eqref{E5.14} follows by
\eqref{E5.15}--\eqref{E5.16}.

Next we prove the Lyapunov property.
We fix a constant $\eta\in(\alpha',\beta')$, and choose
some function $\calV\in C^{2}(\Rd)$ such that
$\calV(x) = \abs{x}^{\eta}$ for $\abs{x}>1$.
Since $\widetilde\uppi(x,z)\,\le\,\frac{1}{\abs{z}^{d+\alpha'}}$ for
all $x\in\Rd$ and $z\in\Rd_{*}$, it follows that
\begin{equation*}
x\;\mapsto\;
\babss{\int_{\abs{z}\le 1} \dd \calV(x;z)\,\widetilde\uppi(x,z)\,\D{z}}
\end{equation*}
is bounded by some constant on $\Rd$.
By \eqref{E5.05},
\begin{equation*}
\babss{\int_{\Rds}\dd \calV(x;z)\,\uppi(x,z)\,\D{z}}
\;\le\; c_{0}\,(1+\abs{x}^{\eta-\alpha})\qquad\forall x\in\Rd\,,
\end{equation*}
for some constant $c_{0}$.
Therefore, in view of \eqref{E5.04}, it is enough to show that for
$\abs{x}\ge 4$, there exist positive constants $c_{1}$ and $c_{2}$ such that
\begin{equation}\label{E5.17}
\int_{\abs{z}>1}
\bigl(\abs{x+z}^\eta-\abs{x}^\eta\bigr)\,\widetilde\uppi(x,z)\,\D{z}\;\le\;
c_{1}-c_{2}\abs{x}^{\eta-\alpha'}\,.
\end{equation}
By the definition of $\gamma$ it holds that
\begin{equation}\label{E5.18}
\widetilde\uppi(x,z)\;=\;\frac{1}{\abs{z}^{d+\beta'}}\,,\qquad
\text{if~}\abs{x+z}\,\ge\,\frac{3}{4}\,\abs{x}\,,~\text{and~}\abs{x}\,\ge\,2\,,
\end{equation}
while
\begin{equation}\label{E5.19}
\widetilde\uppi(x,z)\;=\;\frac{1}{\abs{z}^{d+\alpha'}}\,,\qquad
\text{if~}\abs{x+z}\,\le\,\frac{\abs{x}}{4}\,.
\end{equation}
Suppose that $\abs{x}>2$.
Since $\abs{x+z}\,\le\,\frac{\abs{x}}{4}$ implies
that $\frac{3}{4}\abs{x}\,\le\,\abs{z}\le\frac{5}{4}\abs{x}$,
we obtain by \eqref{E5.19} that
\begin{align}\label{E5.20}
\int_{\abs{x+z}\le\frac{\abs{x}}{4},\,\abs{z}>1}
\bigl(\abs{x+z}^\eta-\abs{x}^\eta\bigr)\,\widetilde\uppi(x,z)\,\D{z}
&\;\le\;
-\int_{\abs{x+z}\le\frac{\abs{x}}{4}}
\left(1-\tfrac{1}{4^\eta}\right)\abs{x}^\eta\,\left(\tfrac{4}{5}\right)^{d+\alpha'}
\frac{1}{\abs{x}^{d+\alpha'}}\,\D{z}\nonumber\\[5pt]
&\;\le\;
-\left(1-\tfrac{1}{4^\eta}\right)\,\left(\tfrac{4}{5}\right)^{d+\alpha'}\,
\abs{x}^{\eta-\alpha'}\int_{\abs{x+z}\le\frac{\abs{x}}{4}}
\frac{\D{z}}{\abs{x}^{d}}\nonumber\\[5pt]
&\;\le\;
-m_{1}\,\abs{x}^{\eta-\alpha'}\,,\qquad\text{if~} \abs{x}>2\,,
\end{align}
for some constant $m_{1}>0$, where we use the fact
that the integral in the
second inequality is independent of $x$ due to rotational invariance.
Also, $\abs{x+z}\le \frac{3}{4}\abs{x}$ implies
$\frac{1}{4}\abs{x}\le \abs{z}\le \frac{7}{4}\abs{x}$, and
in a similar manner, using \eqref{E5.18}, we obtain
\begin{align}\label{E5.21}
\int_{\abs{x+z}\le\frac{3\abs{x}}{4},\,\abs{z}>1}
\bigl(\abs{x+z}^\eta-\abs{x}^\eta\bigr)\,\frac{1}{\abs{x}^{d+\beta'}}\,\D{z}
&\;\ge\;
-\int_{\frac{1}{4}\abs{x}\le \abs{z}\le \frac{7}{4}\abs{x}}
\abs{x}^\eta\,4^{d+\beta'}
\frac{1}{\abs{x}^{d+\beta'}}\,\D{z}\nonumber\\[5pt]
&\;\ge\;
-m_{2}\,\abs{x}^{\eta-\beta'}\,,\qquad\text{if~} \abs{x}>2\,,
\end{align}
for some constant $m_{2}>0$. 
Let $A_1\df \bigl\{z\, : \frac{1}{4}\abs{x}\le \abs{x+z}\le
\frac{3}{4}\abs{x}\bigr\}$.
Since $\eta$ is positive, we have
\begin{equation*}
\int_{\{\abs{z}\geq 1\}\cap A_1}
\bigl(\abs{x+z}^\eta-\abs{x}^\eta\bigr)\widetilde\uppi(x,z)\, \D{z}
\;\le\;0\,.
\end{equation*}
Thus, combining this observation with \eqref{E5.03} and \eqref{E5.21}, we obtain
\begin{align}\label{E5.22}
\int_{\abs{x+z}>\frac{\abs{x}}{4},\,\abs{z}>1}
\bigl(\abs{x+z}^\eta-\abs{x}^\eta\bigr)\,\widetilde\uppi(x,z)\,\D{z}
&\;\le\;
\int_{\abs{x+z}>\frac{3}{4}\abs{x},\,\abs{z}>1}
\bigl(\abs{x+z}^\eta-\abs{x}^\eta\bigr)\,\frac{1}{\abs{z}^{d+\beta'}}\,\D{z}
\nonumber\\[5pt]
&\;=\;
\int_{\abs{z}>1}
\bigl(\abs{x+z}^\eta-\abs{x}^\eta\bigr)\,\frac{1}{\abs{z}^{d+\beta'}}\,\D{z}
\nonumber\\[5pt]
&\mspace{50mu}-
\int_{\abs{x+z}\le\frac{3\abs{x}}{4},\,\abs{z}>1}
\bigl(\abs{x+z}^\eta-\abs{x}^\eta\bigr)\,\frac{1}{\abs{x}^{d+\beta'}}\,\D{z}
\nonumber\\[5pt]
&\;\le\;
m_{3}\,(1+\abs{x}^{\eta-\beta'})
\end{align}
for some constant $m_{3}>0$.
Combining \eqref{E5.20} and \eqref{E5.22}, we obtain
\begin{equation}\label{E5.23}
\int_{\abs{z}>1}
\bigl(\abs{x+z}^\eta-\abs{x}^\eta\bigr)\,\widetilde\uppi(x,z)\,\D{z}
\;\le\;
m_{3}\,(1+\abs{x}^{\eta-\beta'})
-m_{1}\,\abs{x}^{\eta-\alpha'}
\,,\qquad\text{if~} \abs{x}>2\,.
\end{equation}
Therefore, \eqref{E5.17} follows by \eqref{E5.23},
and the Lyapunov property holds.
\end{example}

\begin{proposition}\label{P5.3}
Let $D$ be any bounded open set in $\Rd$ and $X$ be a Markov process
associated with either $\I\in\fL$, or a generator
with kernel $\uppi$ as in Theorem~\ref{T5.3}.
Suppose that for any compact set $K$ and any open set $G$, it holds that
$sup_{x\in K}\Prob_{x}(\tau(G^{c})>T)\to 0$ as $T\to \infty$.
Then for any invariant probability measure $\nu$ of $X$ we have $\nu(D)>0$.
\end{proposition}

\begin{proof}
We argue by contradiction.
Suppose $\nu(D)=0$.
Let $x_{0}\in D$ and $r\in(0,1)$ be such that $B_{2r}(x_{0})\subset D$.
By Lemma~\ref{L3.5} and Remark~\ref{R3.2}
(see also \cite[Proposition 3.1]{bass-kassmann}), we have
\begin{equation*}
\sup_{x\in B_{r}(x_{0})}\;\Prob_{x}\bigl(\tau(B_{r}(x))\le t\bigr)\;\le\;
\kappa\, t\,,\quad t>0\,,
\end{equation*}
for some constant $\kappa$ which depends on $r$.
Therefore there exists $t_{0}>0$ such that
\begin{equation*}
\inf_{x\in B_{r}(x_{0})}\;\Prob_{x}\bigl(\tau(B_{r}(x))\;\ge\;
t_{0}\bigr)\;\ge\; \frac{1}{2}\,.
\end{equation*}
Let $K$ be a compact set satisfying $\nu(K)>\frac{1}{2}$.
By the hypothesis there exists $T_{0}>0$ such that 
$\sup_{x\in K}\Prob_{x}(\tau(B^{c}_{r}(x_{0})>T)\le\nicefrac{1}{2}$
for all $T\ge T_{0}$.
Hence
\begin{align*}
0\;=\;\nu(D)&\;\ge\; \frac{1}{T_{0}+t_{0}}
\int_{0}^{T_{0}+t_{0}}\int_{\Rd}\nu(\D{x})P(t,x; B_{2r}(x_{0}))\,\D{t}
\\[5pt]
&\;=\;\frac{1}{T_{0}+t_{0}}\int_{\Rd}\nu(\D{x})\,
\Exp_{x}\Biggl[\int_{0}^{T_{0}+t_{0}}\bm1_{\{B_{2r}(x_{0})\}}(X_{s})\,\D{t}\Biggr]
\\[5pt]
&\;\ge\;\frac{1}{T_{0}+t_{0}}\int_{K}\nu(\D{x})\,
\Exp_{x}\Biggl[\bm1_{\{\tau(B^{c}_{r}(x_{0}))\le T_{0}\}}
\Exp_{X_{\tau(B^{c}_{r}(x_{0}))}}
\biggl[\bm1_{\{\tau(B_{2r}(x_{0}))\ge t_{0}\}}\\[5pt]
&\mspace{400mu}
\int_{\tau(B_{r}(x_{0}))}^{T_{0}+t_{0}}
\bm1_{\{B_{2r}(x_{0})\}}(X_{s})\,\D{t}\biggr]\Biggr] 
\\[5pt]
&\;\ge\;\frac{1}{T_{0}+t_{0}}\nu(K)
\inf_{x\in K}\;\Prob_{x}\bigl(\tau(B^{c}_{r}(x_{0}))\le T_{0}\bigr)
\inf_{x\in B_{r}(x_{0})}\;\Prob_{x}\bigl(\tau(B_{2r}(x_{0}))\ge t_{0}\bigr)\,t_{0}
\\[5pt]
&\;\ge\;\frac{1}{T_{0}+t_{0}}\frac{\nu(K)}{2}\,
\inf_{x\in B_{r}(x_{0})}\;\Prob_{x}\bigl(\tau(B_{r}(x))\ge t_{0}\bigr)\,t_{0}
\\[5pt]
&\;\ge\; \frac{t_{0}}{T_{0}+t_{0}}\frac{\nu(K)}{4}\;>\;0\,.
\end{align*}
But this is a contradiction. Hence $\nu(D)>0$.
\end{proof}

\subsection{Mean recurrence times for weakly H\"older continuous kernels}

This section is devoted to the characterization of the mean
hitting time of bounded open sets
for Markov processes with generators studied in Section~\ref{S3.2}.
The results hold for any bounded domain $D$ with $C^{2}$ boundary, but
for simplicity we state them for the unit ball centered at $0$.
As introduced earlier, we use the notation $B\equiv B_{1}$.

For nondegenerate continuous diffusions, it is well known that
if some bounded domain $D$ is positive recurrent with respect to 
some point $x\in \Bar{D}^{c}$, then the process is positive recurrent
and its generator satisfies the Lyapunov stability hypothesis in \eqref{E5.01}
\cite[Lemma~3.3.4]{ari-bor-ghosh}).
In Theorem~\ref{T5.4} we show that the same property holds for the class of
operators $\mathfrak{I}_{\alpha}(\beta,\theta,\lambda)$.

\begin{theorem}\label{T5.4}
Let $\I\in\mathfrak{I}_{\alpha}(\beta,\theta,\lambda)$.
We assume that $\I$ satisfies the growth condition in \eqref{E3.5}.
Moreover, we assume that 
$\Exp_{x}[\tau(B^{c})]<\infty$ for some $x$ in $\Bar{B}^{c}$.
Then $u(x)\df \Exp_{x}[\tau(B^{c})]$ is a viscosity solution to 
\begin{align*}
\I u &\;=\;-1 \quad \text{in}~\Bar{B}^{c}\,,\\
u& \;=\;0 \quad \text{in}~\Bar{B}\,.
\end{align*}
\end{theorem} 

In order to prove Theorem~\ref{T5.4} we need the following two lemmas.

\begin{lemma}\label{L5.2}
Let $\I\in\mathfrak{I}_{\alpha}(\beta,\theta,\lambda)$,
and $G$ a bounded open set containing $\Bar{B}$.
Then there exist positive constants $r_{0}$ and $M_{0}$ depending
only on $G$ such that
\begin{equation*}
\int_{\Bar{B}^{c}(x)}\Exp_{z}[\tau(B^{c})]\,\frac{1}{\abs{z}^{d+\alpha}}\,\D{z}
\;<\;\frac{M_{0}}{r^{\alpha}}\,\Exp_{x}[\tau(B^{c})]
\end{equation*}
for every $r<\dist(x, B)\wedge r_{0}$, and for all $x\in G\setminus\Bar{B}$,
such that $\Exp_{x}[\tau(B^{c})]<\infty$.
\end{lemma}

\begin{proof}
Let $\Breve\tau\df\tau(B^{c})$, and
$\Hat\tau_{r}\df\tau\bigl(B_{r}(x)\bigr)$.
We select $r_{0}$ as in Lemma~\ref{L4.1}, and without loss
of generality we assume $r_{0}\le 1$.
We have
\begin{equation}\label{E5.24}
\Exp_{x}\Bigl[\bm1_{\{\Hat\tau_{r}<\Breve\tau\}}
\Exp_{X_{\Hat\tau_{r}}}[\Breve\tau]\Bigr]
\;\le\; \Exp_{x}[\Breve\tau]\,.
\end{equation}
By Definition~\ref{D-hkernel} we have 
\begin{equation*}
k(y,z)\;\ge\; \lambda^{-1}_{G}\;>\;0
\qquad \forall y\in B_{r_{0}}(x)\,.
\end{equation*}
Let $A\subset \Bar{B}^{c}_{r}(x)\cap \Bar{B}^{c}$ be any Borel set.
Using Proposition~\ref{levy-system}, we have
\begin{align*}
\Prob_{x}(X_{\Hat\tau_{r}\wedge t}\in A)
&\;=\;\Exp_{x}\left[\sum_{s\le \Hat\tau_{r}\wedge t}
\bm1_{\{X_{s-}\in
B_{r}(x),\, X_{s}\in A\}}\right]
\\[5pt]
&\;=\;\Exp_{x}\biggl[\int_{0}^{\Hat\tau_{r}\wedge t}
\bm1_{\{X_{s}\in B_{r}(x) \}}
\int_{A}\uppi(X_{s}, z-X_{s})\,\D{z}\,\D{s}\biggr]
\\[5pt]
&\;\ge\; \lambda^{-1}_{G}\,
\Exp_{x}\biggl[\int_{0}^{\Hat\tau_{r}\wedge t}
\int_{A}\frac{1}{\abs{z}^{d+\alpha}}\,\D{z}\,\D{s}\biggr]
\\[5pt]
&\;\ge\; \lambda^{-1}_{G}\,
\Exp_{x}[\Hat\tau_{r}\wedge t]
\int_{A}\frac{1}{\abs{z}^{d+\alpha}}\,\D{z}\,.
\end{align*}
Letting $t\to\infty$, we obtain
\begin{equation}\label{E5.25}
\Prob_{x}(X_{\Hat\tau_{r}}\in A)\;\ge\;
\lambda^{-1}_{G}\,
\Exp_{x}[\Hat\tau_{r}]
\int_{A}\frac{1}{\abs{z}^{d+\alpha}}\,\D{z}\,.
\end{equation}
By Lemma~\ref{L4.1}
it holds that $\Exp_{x}[\Hat\tau_{r}]>
\kappa_{1}\, r^{\alpha}$
for some positive constant $\kappa_{1}$ which depends on $G$.
Hence combining \eqref{E5.24} and \eqref{E5.25} we obtain
\begin{align*}
\lambda^{-1}_{G}\,\kappa_{1}\, r^{\alpha}\,
\int_{\Bar{B}^{c}(x)}\Exp_{z}[\Breve\tau]\,\frac{1}{\abs{z}^{d+\alpha}}\,\D{z}
&\;\le\;
\Exp_{x}\bigl[\bm1_{\{X_{\Hat\tau_{r}}\in \Bar{B}^{c}\}}
\Exp_{X_{\Hat\tau_{r}}}[\Breve\tau]\bigr]\\
&\;\le\; \Exp_{x}[\Breve\tau]\,,
\end{align*}
where the first inequality follows by the standard approximation
technique using step functions.
This completes the proof.
\end{proof}

Lemma~\ref{L5.2} of course implies that if
$\Exp_{x}[\tau(B^{c})]<\infty$ at some point $x\in\Bar{B}^{c}$
then $\Exp_{x}[\tau(B^{c})]$ is finite a.e.-$x$.
We can express the bound in Lemma~\ref{L5.2}
without reference to Lemma~\ref{L4.1} as
\begin{equation*}
\int_{\Bar{B}^{c}(x)}\Exp_{z}[\tau(B^{c})]\,\frac{1}{\abs{z}^{d+\alpha}}\,\D{z}
\;\le\;
\lambda_{G}\,\frac{\Exp_{x}[\tau(B^{c})]}
{\Exp_{x}[\Bar{B}^{c}]}\,.
\end{equation*}
Now let $x^{\prime}$ be any point such that
$\dist(x^{\prime},x)\wedge\dist(x^{\prime},B)= 2r$.
We obtain
\begin{equation*}
\frac{\omega(r)}{\abs{2r}^{d+\alpha}}\;\inf_{z\in B_{r}(x^{\prime})}\;
\Exp_{z}[\tau(B^{c})]
\;\le\; \frac{M_{0}}{r^{\alpha}}\,
\Exp_{x}[\tau(B^{c})]\,.
\end{equation*}
Therefore for some $y\in B_{r}(x^{\prime})$, we have
$\Exp_{y}[\tau(B^{c})]< C_{1}\, \Exp_{x}[\tau(B^{c})]$.
Applying Lemma~\ref{L5.2} once more we obtain
\begin{equation*}
\int_{\Rds}\Exp_{x+z}[\tau(B^{c})]\,\frac{1}{(1+\abs{z})^{d+\alpha}}\,\D{z}
\;\le\;
C_{0} \Exp_{x}[\tau(B^{c})]\,,
\end{equation*}
with the constant $C_{0}$ depending only on $\dist(x,B)$
and the parameter $\lambda$, i.e., the local bounds on $k$.
We introduce the following notation.

\begin{definition}
We say that $v\in L^{1}(\Rd,s)$ if
\begin{equation*}
\int_{\Rds}\frac{\abs{v(z)}}{(1+\abs{z})^{d+\alpha}}\,\D{z}\;<\;\infty\,.
\end{equation*}
\end{definition}

Thus we have the following.

\begin{corollary}
If $\Exp_{x_{0}}[\tau(B^{c})]<\infty$ for some $x_{0}\in\Bar{B}^{c}$,
then the function $u(x)\df \Exp_{x}[\tau(B^{c})]$ is in
$L^{1}(\Rd,s)$.
\end{corollary}

In what follows, without loss of generality we assume that $\beta<s$.
Then, by Theorem~\ref{T3.1},
$u_{n}(x)\df \Exp_{x}\bigl[\tau(B_{n}\cap \Bar{B}^{c})\bigr]$
is the unique solution in
$C^{\alpha+\beta}(B_{n}\setminus \Bar{B})\cap C(\Bar{B}_{n}\setminus B)$ of
\begin{equation}\label{E5.26}
\begin{split}
\I u_{n} &\;=\;-1 \quad \text{in}\quad B_{n}\cap \Bar{B}^{c}\,,\\[5pt]
u_{n}&\;=\;0 \quad \text{in}\quad B_{n}^{c}\cup B\,.
\end{split}
\end{equation}
The following lemma provides a uniform barrier on the solutions $u_{n}$
near $B$.

\begin{lemma}\label{L5.3}
Let $\I\in\mathfrak{I}_{\alpha}(\beta,\theta,\lambda)$,
and
\begin{equation*}
\widetilde\tau_{n}\df\tau(B_{n}\cap \Bar{B}^{c})\,,\quad n\in\NN\,.
\end{equation*}
Then, provided that
$\sup_{x\in F}\,\Exp_{x}[\tau(B^{c})]<\infty$
for all compact sets $F\subset\Bar{B}^{c}$,
there exists a continuous, nonnegative radial function $\varphi$ that
vanishes on $B$, and satisfies, for some $\eta>0$,
\begin{equation*}
\Exp_{x}[\widetilde\tau_{n}]\le \varphi(x)\qquad
\forall\,x\in B_{1+\eta}\setminus B\,,\quad\forall\,n>1\,.
\end{equation*}
\end{lemma}

\begin{proof}
The proof relies on the construction of barrier.
Let $\Hat{k}(x,z)=k(x,z)-k(x,0)$.
By Lemma~\ref{L6.2}, for $q\in (\nicefrac{\alpha-1}{2}, \nicefrac{\alpha}{2})$,
there exists
a constant $c_{0}>0$ such that for $\varphi_{q}(x)\df [(1-\abs{x})^{+}]^{q}$
we have
\begin{equation*}
(-\Delta)^{\nicefrac{\alpha}{2}}\varphi_{q}(x)
\;>\; c_{0}\,(1-\abs{x})^{q-\alpha}\qquad \forall\, x\in B\,.
\end{equation*}
We recall the Kelvin transform from \cite{RosOton-Serra}.
Define
$\Hat{\varphi}(x)=\abs{x}^{\alpha-d}\varphi_{q}(x^*)$ where
$x^*\df \frac{x}{\abs{x}^{2}}$.
Then by \cite[Proposition A.1]{RosOton-Serra} there exists a positive constant
$c_{1}$ such that
\begin{equation*}
(-\Delta)^{\nicefrac{\alpha}{2}}\Hat{\varphi}(x)
\;>\;c_{1}\,(\abs{x}-1)^{q-\alpha}\qquad \forall\, x\in B_{2}\setminus\Bar{B}\,.
\end{equation*}
We restrict $\Hat{\varphi}$ outside a large compact set so that
it is bounded on $\Rd$. By $\widehat\I$ we denote the operator
\begin{equation*}
\widehat\I f(x)\;=\;b(x)\cdot\grad{f}(x) +
\int_{\Rds}\dd f(x;z)\,
\frac{\Hat{k}(x,z)}{\abs{z}^{d+\alpha}}\,\D{z}\,.
\end{equation*}
It is clear that
$\abs{\grad\Hat\varphi(x)}\le c_{2} (\abs{x}-1)^{q-1}$ for all $\abs{x}\in (1,2)$,
for some constant $c_{2}$.
Also, using the fact that $\Hat{\varphi}$ is H\"{o}lder
continuous of exponent $q$ and \eqref{ee3.10}
we obtain
\begin{equation*}
\babss{\int_{\Rds}\dd\Hat\varphi(x;z)\,
\frac{\Hat{k}(x,z)}{\abs{z}^{d+\alpha}}\,\D{z}}\;\le\; c_{3}
(\abs{x}-1)^{q+\theta-\alpha}\qquad
\forall\,x\in B_{2}\setminus\Bar{B}\,,
\end{equation*}
for some constant $c_{3}$.
Hence
\begin{align*}
\babs{\widehat\I\Hat\varphi (x)} \;\le\; c_{4}\,
\bigl(\abs{x}-1\bigr)^{(q-1)\wedge(q+\theta-\alpha)}\,,
\quad \text{for}~ x\in B_{2}\setminus\Bar{B}\,,
\end{align*}
for some constant $c_{4}$.
Since $\theta>0$, $\alpha>1$, and
$\I = \widehat\I - k(x,0) (-\Delta)^{\nicefrac{\alpha}{2}}$,
it follows that we can find $\eta$ small enough such that
\begin{equation*}
\I\Hat\varphi(x)\;<\; -4\,, \quad \text{for}
~ x\in B_{1+\eta}\setminus\Bar{B}\,.
\end{equation*}
Let $K$ be a compact set containing $B_{1+\eta}$. We define
\begin{equation*}
\Tilde{\varphi}(x)\;=\;\Hat{\varphi}(x)\,\bm1_{K}(x)+\Exp_{x}[\tau(B^{c})]
\,\bm1_{K^{c}}(x)\,.
\end{equation*}
Since the hypotheses of Lemma~\ref{L5.2} are met, we conclude that
$\bm1_{K^{c}}(x)\Exp_{x}[\tau(B^{c})]$ is integrable
with respect to the kernel $\uppi$.
For $x\in B_{1+\eta}\setminus\Bar{B}$, we obtain
\begin{align*}
\I\Tilde{\varphi}(x)&\;<\; -4+ 
\int_{\Rds}\bigl(\Exp_{x+z}[\tau(B^{c})]\,
-\Hat{\varphi}(x+z)\bigr)\,\bm1_{K^{c}}(x+z)\,
\uppi(x,z)\,\D{z}\\[5pt]
&\;=\; -4
+ \int_{K^{c}}\Exp_{z}[\tau(B^{c})]\,
\frac{\uppi(x,z-x)}{\uppi(x,z)}\uppi(x,z)\,\D{z}
- \int_{\Rds}\Hat{\varphi}(x+z)\,\bm1_{K^{c}}(x+z)\,
\uppi(x,z)\,\D{z}\,.
\end{align*}
Since the kernel is comparable to $\abs{z}^{-d-\alpha}$
on any compact set,
we may choose $K$ large enough and use Lemma~\ref{L5.2} to obtain
\begin{equation*}
\I\Tilde{\varphi}(x)\;<\; -2
\qquad \forall\,x\in B_{1+\eta}\setminus\Bar{B}\,.
\end{equation*}
Let
\begin{equation*}
\psi(x)\;\df\;\biggl(1\,\vee\sup_{z\in K\setminus B_{1+\eta}}
\Exp_{z}[\tau(B^{c})]\biggr)\,\biggl(1\,\vee\sup_{z\in K\setminus B_{1+\eta}}
\frac{1}{\Tilde\varphi(z)}\biggr)\,\Tilde\varphi(x)\,.
\end{equation*}
Then, $\I\psi<-2$ on $B_{1+\eta}\setminus\Bar{B}$,
while $\psi\ge u_{n}$ on $B_{1+\eta}^{c}\cup B$.
Therefore, by the comparison principle, $u_{n}\le \psi$ on
$B_{1+\eta}\setminus\Bar{B}$
for all $n\in\NN$ and the proof is complete.
\end{proof}

\begin{proof}[Proof of Theorem~\ref{T5.4}]
Consider the sequence of solutions $\{u_{n}\}$ defined in \eqref{E5.26}.
First we note that $u_{n}(x)\le \Exp_{x}[\tau(B^{c})]$ for all $x$.
Clearly $u_{n+1}-u_{n}$ is bounded,
nonnegative and harmonic in $B_{n}\setminus \Bar{B}$.
By Theorem~\ref{T4.1} the operator $\I$ has the Harnack property.
Therefore
\begin{equation*}
\sup_{x\in F}\;\sum_{n\ge 1}\bigl(u_{n+1}(x)-u_{n}(x)\bigr)\;<\;\infty
\end{equation*}
for any compact subset $F$ in $\Bar{B}^{c}$.
Hence Lemma~\ref{L3.3} 
combined with Fatou's lemma implies that
$\sup_{x\in F}\,\Exp_{x}[\tau(B^{c})]<\infty$
for any compact set $F\subset\Bar{B}^{c}$.

We write
\begin{equation*}
u_{n} \;=\; u_{1}+\sum_{m=1}^{n-1}\bigl(u_{m+1}(x)-u_{m}(x)\bigr)\,,
\end{equation*}
and use the Harnack property once more to conclude that
$u_{n}\nearrow u$ uniformly over compact subsets of $\Bar{B}^{c}$.
Since  $u\le \varphi$ in a neighborhood of $\partial B$ by Lemma~\ref{L5.3}, and 
$\varphi$ vanishes on $\partial B$, it follows that 
$u\in C(\Rd)$.
That $u$ is a viscosity solution follows from the fact that $u_{n}\to u$
uniformly over compacta as $n\to\infty$ and Lemma~\ref{L5.2}.
\end{proof}

\section{The Dirichlet problem for weakly H\"older continuous kernels}
\label{S6}

This section is devoted to the study of the
Dirichlet problem 
\begin{equation}\label{ea6.1}
\begin{split}
\I u(x)
&\;=\;f(x) \quad \text{in}~D\,,\\
u & \;=\; 0 \quad \text{in}~D^{c}\,,
\end{split}
\end{equation}
where $\I\in\mathfrak{I}_{\alpha}(\beta,\theta,\lambda)$,
$f$ is H\"{o}lder continuous with exponent $\beta$, and
$D$ is a bounded open set with a $C^{2}$ boundary.
In this section, it is convenient to use $s\equiv\frac{\alpha}{2}$
as the parameter reflecting the order of the kernel.
Throughout this section, we assume $s>\nicefrac{1}{2}$.

Recall the definition of weighted H\"older norms in Section~\ref{S1.1}.
We start with the following lemma.

\begin{lemma}\label{L6.1}
Let $D$ be a $C^{2}$ bounded domain in $\Rd$, and $r\in(0,s]$.
Suppose $k:\Rd\times\Rd\to \R$ and the constants
$\beta\in(0,1)$, $\theta\in\bigl(0,(2s-1)\wedge\beta\bigr)$, and $\lambda_{D}>0$
satisfy parts \textup{(c)} and \textup{(d)} of
Definition~\ref{D-hkernel}.
We define
\begin{equation*}
\begin{split}
\Tilde{k}(x,z)&\;\df\; c(d,2s)\,
\biggl(\frac{k(x,z)}{k(x,0)}-1\biggr)\,,\\[5pt]
\calH[v](x)&\;\df\;\int_{\Rds}\dd v(x;z)\,
\frac{\Tilde{k}(x,z)}{\abs{z}^{d+2s}}\,\D{z}\,,
\end{split}
\end{equation*}
where $c(d,2s)=c(d,\alpha)$ is the normalization constant of the
fractional Laplacian.

Suppose that either of the following assumptions hold:
\begin{itemize}
\item[(i)] $\beta\le r$.
\item[(ii)]
$\beta\in(r,1)$ and $\frac{\Tilde{k}(x,z)}{\abs{z}^{\theta}}$ is bounded
on $(x,z)\in D\times \Rd$, or, equivalently, it satisfies
\begin{equation}\label{ea6.2}
\abs{k(x,z)-k(x,0)}\;\le\; \Tilde{\lambda}_{D}\, \abs{z}^{\theta}\qquad
\forall x\in D\,, \;\forall z\in \Rd\,,
\end{equation}
for some positive constant $\Tilde{\lambda}_{D}$.
\end{itemize}
Then, if $v\in\mathscr{C}_{2s-\theta}^{(-r)}(D)$, we have
\begin{equation*}
\bdbrac{\calH[v]}^{(2s-r-\theta)}_{0;D}\;\le\; M_{0}\,
\dabs{v}^{(-r)}_{2s-\theta;D}\,,
\end{equation*}
and if $v\in\mathscr{C}_{2s+\beta-\theta}^{(-r)}(D)$,
it holds that $\calH[v]\in\mathscr{C}_{\beta}^{(2s-r-\theta)}(D)$, and
\begin{equation}\label{ea6.3}
\bdabs{\calH[v]}^{(2s-r-\theta)}_{\beta;D}\;\le\;
M_{1}\, \dabs{v}^{(-r)}_{2s+\beta-\theta;D}
\end{equation}
for some constants $M_{0}$ and $M_{1}$ which depend only on
$d$, $s$, $\beta$, $r$, and $D$.

Moreover, over a set of parameters of the form
$\{(r,\beta)\,\colon r\in(\varepsilon,1),\, \beta\in(0,1)\}$, constants
$M_{0}$ and $M_{1}$ can be selected which do not
depend on $\beta$ or $r$, but only on $\varepsilon>0$.
\end{lemma}

\begin{proof}
Let $x\in D$, and define $R=\frac{d_{x}}{4}$.
We suppose that $R<1$.
It is clear that $\Tilde{k}$ satisfies \eqref{ee3.10}, and that it
is H\"older continuous.
Abusing the notation,
we'll use the same symbol $\lambda_{D}$ as a constant in the estimates.
We have,
\begin{equation}\label{E6.4}
\babs{\dd{v}(x;z)} \;\le\;
\abs{z}^{2s-\theta}\, R^{r+\theta-2s}\,
\dbrac{v}^{(-r)}_{2s-\theta;D}\qquad \forall\, z\in B_{R}\,.
\end{equation}
Also, since $\abs{z}\ge R$ on $B^{c}_{R}$, we obtain
\begin{align}\label{ea6.5}
\babs{\dd{v}(x;z)} &\;\le\;
\Bigl(\abs{z}^{r}\, \dbrac{v}^{(-r)}_{r;D}
+\abs{z}\,R^{r-1}\,\dbrac{v}^{(-r)}_{1;D}\Bigr)\,\In
+ 2\, \norm{v}_{C(D)}\,\bm 1_{\{\abs{z}>1\}}\nonumber\\[5pt]
&\;\le\;
\bigl(\abs{z}\wedge1\bigr)^{2s-\theta}\, R^{r+\theta-2s}
\Bigl(\dbrac{v}^{(-r)}_{r;D} +\dbrac{v}^{(-r)}_{1;D}\Bigr)
+ 2\, \norm{v}_{C(D)}\,\bm 1_{\{\abs{z}>1\}}
\end{align}
for all $z\in B^{c}_{R}$.
Integrating, using \eqref{ee3.10}, and \eqref{E6.4}--\eqref{ea6.5},
as well as the H\"older interpolation inequalities, we obtain
\begin{equation*}
\abs{\calH[v](x)}\;\le\;c_{1}\, (4\,d_{x})^{r+\theta-2s}\,
\dabs{v}^{(-r)}_{2s-\theta;D}\qquad\forall x\in D\,,
\end{equation*}
for some constant $c_{1}$.
Therefore, for some constant $M_{0}$, we have
\begin{equation}\label{ea6.6}
\bdbrac{\calH[v]}^{(2s-r-\theta)}_{0;D}\;\le\; M_{0}\,
\dabs{v}^{(-r)}_{2s-\theta;D}\,.
\end{equation}

Next consider two points $x\,,y\in D$.
If $\abs{x-y}\ge4 d_{xy}$, then \eqref{ea6.6} provides a suitable estimate.
Indeed, if $x,\,y\in D$ are such $4 d_{xy} \le \abs{x-y}$,
then, for any $r$ we have
\begin{align*}
d_{xy}^{2s-r-\theta}\,d_{xy}^{\beta}\,
\frac{\abs{\calH[v](x)-\calH[v](y)}}{\abs{x-y}^{\beta}}
&\;\le\; \frac{1}{4^{\beta}}\,d_{xy}^{2s-r-\theta}\,
\abs{\calH[v](x)-\calH[v](y)}\\[5pt]
&\;\le\;\frac{1}{4^{\beta}}\,d_{x}^{2s-r-\theta}\,\abs{\calH[v](x)}
+\frac{1}{4^{\beta}}\,d_{y}^{2s-r-\theta}\,\abs{\calH[v](y)}\\[5pt]
&\;\le\;\frac{2 M_{0}}{4^{\beta}}\,\dabs{v}^{(-r)}_{2s-\theta;D}\,.
\end{align*}

So it suffices to consider the case $\abs{x-y}< 4 d_{xy}$.
Therefore, we may suppose that $x$ is as above and that $y\in B_{R}(x)$.
Then $d_{xy}\le 4R$.
With $\tpi(x,z)\df\frac{\Tilde{k}(x,z)}{\abs{z}^{d+2s}}$, we write
\begin{align*}
F(x,y;z)&\;\df\;\dd{v}(x;z)\,\tpi(x,z)-\dd{v}(y;z)\,\tpi(y,z)\\[5pt]
&\;=\;F_{1}(x,y;z)+F_{2}(x,y;z)\,,
\end{align*}
with
\begin{align*}
F_{1}(x,y;z)&\;\df\;\Bigl(\dd{v}(x;z)+\dd{v}(y;z)\Bigr)\,
\frac{\tpi(x,z)-\tpi(y,z)}{2}\,,\\[5pt]
F_{2}(x,y;z)&\;\df\;\Bigl(\dd{v}(x;z)-\dd{v}(y;z)\Bigr)\,
\frac{\tpi(x,z)+\tpi(y,z)}{2}\,.
\end{align*}
We modify the estimate in \eqref{E6.4}, and write
\begin{align*}
\babs{\dd{v}(x;z)+\dd{v}(y;z)} &\;\le\;
2\,\abs{z}^{\gamma_{0}}\,
R^{r-\gamma_{0}}\,\dbrac{v}^{(-r)}_{\gamma_{0};D}\,,
\qquad\text{if~} z\in B_{R}\,,\\[5pt]
\intertext{with $\gamma_{0}=(2s+\beta-\theta)\wedge (s+1)$, and}
\babs{\dd{v}(x;z)+\dd{v}(y;z)} &\;\le\;
2\,\Bigl(\abs{z}^{r}\, \dbrac{v}^{(-r)}_{r;D}
+\abs{z}\,R^{r-1}\,\dbrac{v}^{(-r)}_{1;D}\Bigr)\,\In
+ 4 \norm{v}_{C(D)}\,\bm 1_{\{\abs{z}>1\}}\,,
\end{align*}
if $z\in B^{c}_{R}$.
We use the H\"older continuity of $x\mapsto\Tilde k(x,\,\cdot\,)$
to obtain
\begin{equation*}
\int_{\Rds} F_{1}(x,y;z)\,\D{z}
\;\le\; c_{2}\,R^{r-2s}\, \abs{x-y}^{\beta}\,
\dabs{v}^{(-r)}_{\gamma_{0};D}
\end{equation*}
for some constant $c_{2}$.
We write this as
\begin{align}\label{ea6.7}
R^{2s-r-\theta}\,R^{\beta}\,
\frac{\int_{\Rds} F_{1}(x,y;z)\,\D{z}}{\abs{x-y}^{\beta}}
&\;\le\; R^{2s-r-\beta}\,R^{\beta}\,
\frac{\int_{\Rds} F_{1}(x,y;z)\,\D{z}}{\abs{x-y}^{\beta}}\nonumber\\[5pt]
&\;\le\; c_{2}\,\dabs{v}^{(-r)}_{\gamma_{0};D}\,.
\end{align}
For $F_{2}$, we use
\begin{equation*}
\dd{v}(x;z) \;=\; z\cdot \int_{0}^{1}
\bigl(\grad v(x+t z) - \grad v(x)\bigr)\,\D{t}\,,
\end{equation*}
combined with the following fact:
If $f\in C^{\gamma}(B)$ for 
$\gamma\in(0,1]$ and $x$, $y$, $x+z$, $y+z$ are points in
$B$ and $\delta\in(0,\gamma)$, then adopting the notation
$\varDelta f_{x}(z) \df f(x+z)-f(x)$,
we obtain by Young's inequality, that
\begin{align*}
\frac{\abs{\varDelta f_{x}(z) - \varDelta f_{y}(z)}}
{\abs{z}^{\gamma-\delta}\abs{x-y}^{\delta}}
&\;\le\;
\frac{\gamma-\delta}{\gamma}\,
\frac{\abs{\varDelta f_{x}(z)} + \abs{\varDelta f_{y}(z)}}
{\abs{z}^{\gamma}}
+\frac{\delta}{\gamma}\,
\frac{\abs{\varDelta f_{x+z}(y-x)} + \abs{\varDelta f_{x}(y-x)}}
{\abs{x-y}^{\gamma}}\\[5pt]
&\;\le\; 2 [f]_{\gamma; B}\,.
\end{align*}
The same inequality also holds for $\gamma\in(1,2)$ and $\delta\in(\gamma-1, 1)$.
For this we use
\begin{align*}
\frac{\abs{\varDelta f_{x}(z) - \varDelta f_{y}(z)}}
{\abs{z}^{\gamma-\delta}\abs{x-y}^{\delta}}
&\;\le\;
\frac{1-\delta}{2-\gamma}\,
\frac{\abs{z}\,\Babs{\int_{0}^{1}
\bigl(\grad f(x+t z) - \grad f(y+t z)\bigr)\,\D{t}}}
{\abs{x-y}^{\gamma-1}\,\abs{z}}\nonumber\\[5pt]
&+\frac{1+\delta-\gamma}{2-\gamma}\,
\frac{\abs{x-y}\,\Babs{\int_{0}^{1}
\bigl(\grad f(y+z+t (x-y)) - \grad f(y+t (x-y))\bigr)\,\D{t}}}
{\abs{z}^{\gamma-1}\,\abs{x-y}}
\end{align*}
Therefore, in either of the cases (i) or (ii) we obtain,
\begin{equation*}
\babs{\grad v(x+t z) - \grad v(x) -
\grad v(y+t z) + \grad v(y)}
\;\le\; 2\abs{t z}^{2s-\theta-1}\abs{x-y}^{\beta}
\, [\grad v]_{2s-\theta -1+ \beta;\,B_{2R}(x)}
\end{equation*}
for $t\in[0,1]$, and
\begin{equation}\label{ea6.8}
\babs{\dd{v}(x;z)-\dd{v}(y;z)} \;\le\;
\frac{2}{2s-\theta}\;
\abs{z}^{2s-\theta}\,\abs{x-y}^{\beta}\,
R^{r+\theta-\beta-2s}\,\dbrac{v}^{(-r)}_{2s+\beta-\theta;D}
\qquad\forall\,z\in B_{R}\,.
\end{equation}
Concerning the integration on $B^{c}_{R}$, we use
\begin{align}\label{ea6.9}
\babs{v(x)-v(y)&- z\cdot\bigl(\grad v(x)-\grad v(y)\bigr)\In}\nonumber\\[5pt]
&\;\le\;
\abs{x-y}^{\beta\vee r}\,d_{xy}^{r-\beta\vee r}\,
\dbrac{v}^{(-r)}_{\beta\vee r;D}
+\bigl(\abs{z}\wedge1\bigr)\,\abs{x-y}^{\beta}\,d_{xy}^{r-\beta-1}\,
\dbrac{v}^{(-r)}_{1+\beta;D}\nonumber\\[5pt]
&\;\le\;
c_{3}\,\bigl(\abs{z}\wedge1\bigr)^{2s-\theta}\,\abs{x-y}^{\beta}\,
R^{r+\theta-\beta-2s}\,
\dabs{v}^{(-r)}_{1+\beta;D}\qquad \forall\, z\in B^{c}_{R}\,,
\end{align}
for some constant $c_{3}$, and
\begin{equation}\label{ea6.10}
\abs{v(x+z)-v(y+z)}\;\le\;
\abs{x-y}^{\beta\vee r}\,
(d_{x+z}\wedge d_{y+z})^{r-\beta\vee r}\,
\dbrac{v}^{(-r)}_{\beta\vee r;D}
\qquad\forall\,z\in B^{c}_{R}\,.
\end{equation}
Integrating the terms on the right hand side of \eqref{ea6.8}--\eqref{ea6.9}
is straightforward.
Doing so, and using the fact that $1+\beta<2s+\beta-\theta$, one obtains
the desired estimate.

Concerning the integral of $\abs{v(x+z)-v(y+z)}$ on $B^{c}_{R}$,
we distinguish between the cases (i) and (ii).
Let $\tpi(z)\df\frac{\abs{\tpi(x,z)+\tpi(y,z)}}{2}$.
In case (i) we have
\begin{align}\label{ea6.11}
\int_{B^{c}_{R}} \abs{v(x+z)-&v(y+z)}\,\tpi(z)\,\D{z}
\nonumber\\[5pt]
&\;\le\;
\abs{x-y}^{r}\, \dbrac{v}^{(-r)}_{r;D}
\int_{B^{c}_{R}} \,\tpi(z)\,\D{z} \nonumber\\[5pt]
&\;\le\;
\abs{x-y}^{\beta}\,R^{r-\beta}\,R^{\theta-2s}\,
\dbrac{v}^{(-r)}_{r;D}
\int_{\Rds}
\bigl(\abs{z}\wedge\diam(D)\bigr)^{2s-\theta}\,\tpi(z)\,\D{z}\,,
\end{align}
where we use the fact that $\abs{z}>R$ on $B^{c}_{R}$.
In case (ii) the integral is estimated over disjoint sets.
We define
\begin{equation*}
\cZ_{xy}(a)\;\df\;\{z\in\Rd\,\colon d_{x+z}\wedge d_{y+z} < a\}
\qquad\text{for~}a\in(0,R)\,.
\end{equation*}
Since $d_{x+z}\wedge d_{y+z} \in[R,\diam(D)]$ for $x\in\cZ^{c}_{xy}(R)$,
integration is straightforward, after replacing
$(d_{x+z}\wedge d_{y+z})^{r-\beta}$ in \eqref{ea6.10} with $R^{r-\beta}$.
Thus, similarly to \eqref{ea6.11}, we obtain
\begin{align}\label{ea6.12}
\int_{B^{c}_{R}\,\cap\,\cZ^{c}_{xy}(R)}& \abs{v(x+z)-v(y+z)}\,\tpi(z)\,\D{z}
\nonumber\\[5pt]
&\;\le\;
\abs{x-y}^{\beta}\, R^{r-\beta}\, \dbrac{v}^{(-r)}_{\beta;D}
\int_{B^{c}_{R}\,\cap\,\cZ^{c}_{xy}(R)} \,\tpi(z)\,\D{z} \nonumber\\[5pt]
&\;\le\;
\abs{x-y}^{\beta}\, R^{r+\theta-\beta-2s}\,
\dbrac{v}^{(-r)}_{\beta;D}
\int_{\Rds}
\bigl(\abs{z}\wedge\diam(D)\bigr)^{2s-\theta}\,\tpi(z)\,\D{z}\,.
\end{align}
Since $\cZ_{xy}(R)\subset B^{c}_{R}$, it remains to compute the integral on
$\cZ_{xy}(R)$.
Recall the definition of $D_{\varepsilon}$ in \eqref{ee3.11}.
We also define for $\varepsilon>0$,
\begin{equation*}
\widetilde{D}(\varepsilon)\;=\;\{z\in D:\,\dist(z,\partial D)\ge \varepsilon\}\,. 
\end{equation*}
In other words $\widetilde{D}(\varepsilon) = (D^{c})^{c}_{\varepsilon}$.
We'll make use of the following simple fact:
There exists a constant $C_{0}$,
such that for all $x\in D$ and positive constants $R$ and $\varepsilon$
which satisfy $0<\varepsilon\le R$ and $d_{x}\ge 3R$, it holds that
\begin{equation}\label{ea6.13}
\int_{x+z\,\in D_{\varepsilon}\setminus \widetilde{D}(\varepsilon)}
\frac{\D{z}}{\abs{z}^{d}}
\;\le\; \frac{C_{0}\, \varepsilon}{R}\,.
\end{equation}
Observe that the support of $\abs{v(x+z)-v(y+z)}$ in $\cZ_{xy}(R)$
is contained in the disjoint union of the sets
\begin{align*}
\widetilde{\cZ}_{xy}(R)&\;\df\;\bigl\{z\in \cZ_{xy}(R)\,\colon
d_{x+z}\wedge d_{y+z} > 0\bigr\}\,,\\
\intertext{and}
\widehat{\cZ}_{xy}&\;\df\;\bigl\{z\in \R^{d}\,\colon x+z\in
D_{\abs{x-y}}\setminus D
~\text{or}~y+z\in D_{\abs{x-y}}\setminus D\bigr\}\,.
\end{align*}
We also have the bound $\abs{v(x+z)-v(y+z)}\le \abs{x-y}^{r}\dbrac{v}^{(-r)}_{r;D}$
for $z\in\widehat{\cZ}_{xy}$.
Therefore, using \eqref{ea6.13}, we obtain
\begin{align}\label{ea6.14}
\int_{\widehat{\cZ}_{xy}} \abs{v(x+z)-v(y+z)}\,\tpi(z)\,\D{z}
&\;\le\;
\abs{x-y}^{r}\dbrac{v}^{(-r)}_{r;D}\,R^{\theta-2s}\,
\int_{\widehat{\cZ}_{xy}}\abs{z}^{2s-\theta}\,\tpi(z)\,\D{z} \nonumber\\[5pt]
&\;\le\;
\abs{x-y}^{r}\dbrac{v}^{(-r)}_{r;D}\,R^{\theta-2s}\,
\int_{\widehat{\cZ}_{xy}}
\frac{\D{z}}{\abs{z}^{d}} \nonumber\\[5pt]
&\;\le\;
2\,\Tilde{\lambda}_{D}\,C_{0}\,\abs{x-y}^{r+1}
\dbrac{v}^{(-r)}_{r;D}\,R^{\theta-2s}\,R^{-1}\nonumber\\[5pt]
&\;\le\;
2\,\Tilde{\lambda}_{D}\,C_{0}\,\abs{x-y}^{\beta}\,R^{r+\theta-\beta-2s}\,
\dbrac{v}^{(-r)}_{r;D}\,.
\end{align}
In order to evaluate the integral over $\widetilde{\cZ}_{xy}(R)$,
we define
\begin{equation*}
G(z) \;\df\; \frac{\abs{v(x+z)-v(y+z)}}{\abs{x-y}^{\beta}\,
\dbrac{v}^{(-r)}_{\beta;D}}\,.
\end{equation*}
By \eqref{ea6.10} we have
\begin{equation*}
\bigl\{z\in \widetilde{\cZ}_{xy}(R)\,\colon G(z)> h\bigr\}
\;\subset\;
\bigl\{z\in \Rd\,\colon x+z\in
\widetilde{D}^{c}\bigl(h^{\frac{-1}{\beta-r}}\bigr)\bigr\}\,\cup\,
\bigl\{z\in \Rd\,\colon y+z\in
\widetilde{D}^{c}\bigl(h^{\frac{-1}{\beta-r}}\bigr)\bigr\}\,.
\end{equation*}
Therefore, by \eqref{ea6.13}, we obtain
\begin{align*}
\tpi\bigl(\bigl\{z\in \widetilde{\cZ}_{xy}(R)\,\colon G(z)> h\bigr\}\bigr)
&\;\le\; 2 R^{\theta-2s}\,
\int_{\widehat{\cZ}_{xy}}\abs{z}^{2s-\theta}\,\tpi(z)\,\D{z} \nonumber\\[5pt]
&\;\le\;
2\,\Tilde{\lambda}_{D}\,C_{0}\,R^{\theta-2s-1}
h^{\frac{-1}{\beta-r}}\,.
\end{align*}
It follows that
\begin{align}\label{ea6.15}
\int_{\widetilde{\cZ}_{xy}(R)} G(z)\,\tpi(z)\,\D{z}
&\;=\; \int_{0}^{\infty}
\tpi\bigl(\bigl\{z\in \widetilde{\cZ}_{xy}(R)\,\colon G(z)> h\bigr\}\bigr)
\,\D{h}\nonumber\\[5pt]
&\;\le\;
2\,\Tilde{\lambda}_{D}\,C_{0}\,R^{\theta-2s-1}\,
\int_{R^{r-\beta}}^{\infty}h^{\frac{-1}{\beta-r}}\,\D{h}
\nonumber\\[5pt]
&\;\le\;
\frac{2\,(\beta-r)}{1+r-\beta}\,\Tilde{\lambda}_{D}\,C_{0}\,R^{\theta-2s-1}\,
R^{1+r-\beta}\,.
\end{align}

Thus, combining \eqref{ea6.8}--\eqref{ea6.9} with
\eqref{ea6.11} in case (i), or with \eqref{ea6.12}, \eqref{ea6.14}
and \eqref{ea6.15} in case (ii), 
and using the H\"older interpolation inequalities, we obtain
\begin{equation}\label{ea6.16}
R^{2s-r-\theta}\,R^{\beta}\,
\frac{\int_{\Rd} F_{2}(x,y;z)\,\D{z}}{\abs{x-y}^{\beta}}
\;\le\; c_{4}\,\dbrac{v}^{(-r)}_{2s+\beta-\theta;D}
\end{equation}
for some constant $c_{4}$.

Therefore, by \eqref{ea6.6}, \eqref{ea6.7} and \eqref{ea6.16} we obtain
\eqref{ea6.3}, and the proof is complete.
\end{proof}

\begin{remark}
It is evident from the proof of Lemma~\ref{L6.1} that
the assumption in \eqref{ea6.2} may be replaced by the following:
There exists a constant ${M}_{D}$,
such that for all $x\in D$ and positive constants $R$ and $\varepsilon$
which satisfy
$0<\varepsilon\le R$ and $d_{x}\ge 3R$,
it holds that
\begin{equation*}
\int_{x+z\,\in D_{\varepsilon}\setminus \widetilde{D}(\varepsilon)}
\frac{\Tilde{k}(x,z)}{\abs{z}^{d-\theta}}\,\D{z}
\;\le\; M_{D}\,\frac{\varepsilon}{R}\,.
\end{equation*}
The same applies to Theorems~\ref{T3.1} and \ref{T6.1}.
\end{remark}

In order to proceed,
we need certain properties of solutions of $(-\Delta)^{s} = f$ in a bounded
domain $D$, and $u=0$ on $D^{c}$, with $f$ not necessarily in $L^{\infty}(D)$.
We start with exhibiting a suitable supersolution.

\begin{lemma}[Supersolution]\label{L6.2}
For any $q\in \bigl(s-\nicefrac{1}{2}, s\bigr)$ there exists a constant
$c_{0}\; > \; 0$ and a radial continuous function $\varphi$ such that
\begin{equation*}
\begin{cases}
(-\Delta)^{s}\varphi(x) \;\ge\; d_{x}^{q-2s}\,,
& \text{in}~B_{4}\setminus \Bar{B}_{1}\,,\\
\varphi \;=\; 0 & \text{in}~B_{1}\,,\\
0\;\le\; \varphi\;\le\; c_{0}(\abs{x}-1)^{q}
& \text{in} \quad B_{4}\setminus B_{1}\,,\\
1\;\le\; \varphi\;\le\; c_{0} & \text{in} \quad \R^{d}\setminus B_{4}\,.
\end{cases}
\end{equation*}
\end{lemma}

\begin{proof}
In view of the Kelvin transform \cite[Proposition~A.1]{RosOton-Serra}
it is enough to prove the following: for
$q\in \bigl(s-\nicefrac{1}{2}, s\bigr)$, and with
$\psi(x) \;\df\; [(1-\abs{x})^{+}]^{q}$, we have
\begin{equation}\label{ea6.17}
(-\Delta)^{s}\psi(x)\;\ge\; c_{1}\,(1-\abs{x})^{q-2s}\,,
\quad \text{for all}~ x\in B_{1}\,,
\end{equation}
for some positive constant $c_{1}$.
To prove \eqref{ea6.17} we let $x_{0}\in B_{1}$.
Due to the rotational symmetry we may assume
$x_{0}=r e_{1}$ for some $r\in (0,1)$. Denote $z=(z_{1},\dotsc, z_{d})$.
Then
\begin{align*}
-(-\Delta)^{s}\psi(x_{0}) \;&=\; c(d,\alpha)
\int_{\Rds}\bigl(\psi(x_{0}+z)-\psi(x_{0})\bigr)\frac{1}{\abs{z}^{d+2s}}\,\D{z}
\\[5pt]
& = \; c(d,\alpha)\int_{\Rds}
\Bigl(\bigl[(1-|re_{1}+z|)^{+}\bigr]^{q}-(1-r)^{q}\Bigr)
\frac{1}{\abs{z}^{d+2s}}\,\D{z}\\[5pt]
& \le \; c(d,\alpha)\int_{\Rds}
\Bigl(\bigl[(1-|re_{1}+z_{1}|)^{+}\bigr]^{q}-(1-r)^{q}\Bigr)
\frac{1}{\abs{z}^{d+2s}}\,\D{z}\\[5pt]
& = \; c_{2}\int_{\R}\Bigl(\bigl[(1-|r+z|)^{+}\bigr]^{q}-(1-r)^{q}\Bigr)
\frac{1}{\abs{z}^{1+2s}}\,\D{z}\\[5pt]
& \le \; c_{2}\int_{\R}\Bigl(\bigl[(1-r-z)^{+}\bigr]^{q}-(1-r)^{q}\Bigr)
\frac{1}{\abs{z}^{1+2s}}\,\D{z}\\[5pt]
& = \; c_{2}(1-r)^{q-2s}\int_{\R}\Bigl(\bigl[(1-z)^{+}\bigr]^{q}-1\Bigr)
\frac{1}{\abs{z}^{1+2s}}\,\D{z}
\end{align*}
for some constant $c_{2}$,
where in the first inequality we use the fact that
$(1-\abs{z})^{+}\le (1-|z_{1}|)^{+}$
and in the second inequality we use $1-\abs{z}\le 1-z$. Define
\begin{align*}
A(q) & \;\df\; \int_{\R}\Bigl([(1-z)^{+}]^{q}-1\Bigr)\frac{1}{\abs{z}^{1+2s}}\,\D{z}
\;=\;\int_{0}^{\infty}\frac{z^{q}-1}{\abs{1-z}^{1+2s}}\,\D{z}
-\int_{-\infty}^{0}\frac{1}{\abs{1-z}^{1+2s}}\,\D{z}\,,\\[5pt]
B(q) &\;\df \; \int_{0}^{\infty}\frac{z^{q}-1}{\abs{1-z}^{1+2s}}\,\D{z}\,.
\end{align*}
We need to show that $A(q)<0$ for $q$ close to $s$.
It is known that $A(s)=0$ \cite[Proposition~3.1]{RosOton-Serra}.
Therefore it is enough to show that
$B(q)$ is strictly increasing for $q\in \bigl(s-\nicefrac{1}{2}, s\bigr)$.
We have
\begin{align*}
B(q) \;&=\;\int_{0}^{1}\frac{z^{q}-1}{\abs{1-z}^{1+2s}}\,\D{z}\, + \,
\int_{1}^{\infty}\frac{z^{q}-1}{\abs{1-z}^{1+2s}}\,\D{z}\\[5pt]
&=\; \int_{0}^{1}\frac{(z^{q}-1)(1-z^{2s-1-q})}{\abs{1-z}^{1+2s}}\,\D{z}\,.
\end{align*}
Therefore, for $q\in \bigl(s-\nicefrac{1}{2}, s\bigr)$, we obtain
\begin{equation*}
\frac{dB(q)}{dq}\; =\;
\int_{0}^{1}\frac{(z^{q}-z^{2s-1-q})\log{z}}{\abs{1-z}^{1+2s}}\,\D{z}\; >\; 0\,,
\end{equation*}
where we use the fact that $\log{z}\le 0$ in $[0, 1]$. This
completes the proof.
\end{proof}

\begin{lemma}\label{L6.3}
Let $f$ be a continuous function in $D$ satisfying
$\sup_{x\in D}\,d_{x}^{\delta} |f(x)|<\infty$
for some $\delta<s$.
Then there exists a viscosity solution $u\in C(\R^{d})$ to 
\begin{equation*}
\begin{split}
(-\Delta)^{s} u(x) & \;=\;-f(x) \quad \text{in}~D\,,\\[5pt]
u& \;=\; 0 \quad \text{in}~D^{c}\,.
\end{split}
\end{equation*}
Also, for every $q< s$ we have
\begin{subequations}
\begin{align}
\abs{u(x)}&\;\le\; C_{1}\, \dbrac{f}^{(\delta)}_{0;D}\,d^{q}_{x}
\qquad\forall\,x\in\Bar{D}\,,\label{ea6.18a}\\[5pt]
\norm{u}_{C^{q}(\Bar{D})}&\;\le\; C_{1}\,\sup_{x\in D}\;
d_{x}^{\delta}\,\abs{f(x)}\,,\label{ea6.18b}
\end{align}
\end{subequations}
for some constant $C_{1}$ that depends only on $s$, $\delta$, $q$
and the domain $D$.
Moreover, since $u=0$ on $D^{c}$, it follows that the H\"{o}lder norm of
$u$ on $\R^{d}$ is bounded by the same constant.
\end{lemma}

\begin{proof}
Existence of a continuous viscosity solution follows from
Lemma~\ref{L6.2} and Perron's method,
since we can always choose $q$ close enough to $s$ in Lemma~\ref{L6.2}
so as to satisfy $2s-q>\delta$, and obtain a bound on the solution $u$.
From the barrier there exists a compact set $K_{1}\subset D$ such that
\begin{equation}\label{ea6.19}
|u(x)|\;\le\; \kappa_{1} \biggl(\sup_{x\in K_{1}}\;|u(x)|
+\dbrac{f}^{(\delta)}_{0;D}\biggr)\,d_{x}^{q}
\qquad \forall \, x\in K_{1}^{c}\,,
\end{equation}
where the constant $\kappa_{1}$ depends only on $K_{1}$ and $D$.
Also, using the same argument as in Lemma~\ref{L3.2}, we can show that
for any compact $K_{2}\subset D$, there exists a constant
$\kappa_{2}$, depending on $D$, and satisfying
\begin{equation}\label{ea6.20}
\sup_{x\in K_{2}}\;|u(x)|\;\le\; \kappa_{2}\biggl(\sup_{x\in K_{2}}\;|f(x)|
+\sup_{x\in D\setminus K_{2}}\;|u(x)|\biggr)\,.
\end{equation}
We choose $K_{2}$ and $K_{1}\subset K_{2}$ such that
$\sup_{x\in K^{c}_{2}\cap D}\, |d_{x}^{q}|<\frac{1}{2\kappa_{1}\kappa_{2}}$.
Then from \eqref{ea6.19}--\eqref{ea6.20} we obtain
\begin{equation}\label{ea6.21}
\sup_{x\in K_{2}}\;|u(x)|\;\le\; \kappa_{3}\;\dbrac{f}^{(\delta)}_{0;D}
\end{equation}
for some constant $\kappa_{3}$.
Hence the bound in \eqref{ea6.18a} follows by combining
\eqref{ea6.19} and \eqref{ea6.21}.

The estimate in \eqref{ea6.18b} is easily obtained by
following the argument in the proof of \cite[Proposition~1.1]{RosOton-Serra}.
\end{proof}

Our main result in this section is the following.

\begin{theorem}\label{T6.1}
Let $\I\in\mathfrak{I}_{2s}(\beta,\theta,\lambda)$,
$f$ be locally H\"older continuous with exponent $\beta$, and
$D$ be a bounded domain with a $C^{2}$ boundary.
We assume that neither $\beta$, nor $2s+\beta$ are integers, and
that either $\beta<s$, or that
$\beta\ge s$ and
\begin{equation*}
\abs{k(x,z)-k(x,0)}\;\le\; \Tilde{\lambda}_{D}\, \abs{z}^{\theta}\qquad
\forall x\in D\,, \;\forall z\in \Rd\,,
\end{equation*}
for some positive constant $\Tilde{\lambda}_{D}$.
Then the Dirichlet problem in \eqref{ea6.1} has a unique solution in
$C^{2s+\beta}_{\mathrm{loc}}(D)\cap C(\Bar{D})$.
Moreover, for any $r<s$, we have the estimate
\begin{equation*}
\dabs{u}^{(-r)}_{2s+\beta;D}
\;\le\; C_{0} \, \norm{f}_{C^{\beta}(\Bar{D})}
\end{equation*}
for some constant $C_{0}$ that depends only on $d$, $\beta$, $r$, $s$ and
the domain $D$.
\end{theorem}

\begin{proof}
Consider the case $\beta\ge s$.
We write \eqref{ea6.1} as
\begin{equation}\label{ea6.22}
\begin{split}
(-\Delta)^{s} u(x) & \;=\;\mathcal{T}[u](x)\;\df\;\frac{c(d,2s)}{k(x,0)}\,
\bigl(-f(x) + b(x)\cdot\grad{u}(x)\bigr)+ \calH[u](x)
\quad \text{in}~D\,,\\[5pt]
u& \;=\; 0 \quad \text{in}~D^{c}\,,
\end{split}
\end{equation}
and we apply the Leray--Schauder fixed point theorem.
Also, without loss of generality, we assume $\theta<2s-1$.
We choose any $r\in(0,s)$ which satisfies
\begin{equation*}
 r\;>\;\Bigl(s -\frac{\theta}{2}\Bigr)
\vee\Bigl(1-s+\frac{\theta}{2}\Bigr)\,,
\end{equation*}
and let $v\in\mathscr{C}_{2s+\beta-\theta}^{(-r)}(D)$.
Then
$\calH[v]\in
\mathscr{C}_{\beta}^{(2s-r-\theta)}(D)$
by Lemma~\ref{L6.1}.
Since $\grad V\in \mathscr{C}_{2s+\beta-\theta-1}^{(1-r)}(D)$
and $(1-r)\wedge (2s-r-\theta) < s$ by hypothesis, then
applying Lemma~\ref{L6.3} we conclude that there exists a solution $u$ to
$(-\Delta)^{s}u = \mathcal{T}[v]$ on $D$, with $u=0$ on $D^{c}$, such that
$u\in\mathscr{C}_{0}^{(-q)}(D)$ for any $q<s$.

Next we obtain some estimates that are needed in order to apply
the Leray--Schauder fixed point theorem.
By Lemma~\ref{L6.1} we obtain
\begin{equation*}
\bdabs{\calH[v]}^{(2s-r-\nicefrac{\theta}{2})}_{0;D}
\;=\; \bdabs{\calH[v]}^{(2s-(r-\nicefrac{\theta}{2})-\theta)}_{0;D}
\;\le\;
\kappa_{1}\, \dabs{v}^{(-r+\nicefrac{\theta}{2})}_{2s-\theta;D}\,,
\end{equation*}
and similarly,
\begin{equation}\label{ea6.23}
\bdabs{\calH[v]}^{(2s-r-\nicefrac{\theta}{2})}_{\beta;D}\;\le\;
\kappa_{1}\, \dabs{v}^{(-r+\nicefrac{\theta}{2})}_{2s+\beta-\theta;D}\,,
\end{equation}
for some constant $\kappa_{1}$ which does not depend on $\theta$ or $r$.
Thus, since by hypothesis $2s-r-\nicefrac{\theta}{2}<s$
and $1-r+\nicefrac{\theta}{2}<s$,
we obtain by Lemma~\ref{L6.3} that
\begin{equation}\label{ea6.24}
\norm{u}_{C^{r}(\Rd)}\;\le\; \kappa'_{1}\,
\Bigl(\norm{f}_{C(\Bar{D})}+\dabs{\grad v}^{(1-r+\nicefrac{\theta}{2})}_{0;D}
+\dabs{v}^{(-r+\nicefrac{\theta}{2})}_{2s-\theta;D}\Bigr)\,
\end{equation}
for some constant $\kappa'_{1}$.
Also, by Lemma~2.10 in \cite{RosOton-Serra}, there exists a constant
$\kappa_{2}$, depending only on $\beta$, $s$, $r$ and $d$, such that
\begin{equation}\label{ea6.25}
\dabs{u}^{(-r)}_{2s+\beta;D}\;\le\; \kappa_{2}\,\Bigl(\norm{u}_{C^{r}(\Rd)}
+\bdabs{\mathcal{T}[v]}^{(2s-r)}_{\beta;D}\Bigr)\,.
\end{equation}
It follows by \eqref{ea6.24}--\eqref{ea6.25} that $v\mapsto u$
is a continuous map from $\mathscr{C}^{(-r)}_{2s+\beta-\theta}$ to itself.
Moreover, since $\mathscr{C}^{(-r)}_{2s+\beta}(D)$ is precompact in
$\mathscr{C}^{(-r)}_{2s+\beta-\theta}(D)$,
it follows that $v\mapsto u$ is compact.

Next we obtain a bound for $\dabs{u}^{(-r)}_{2s+\beta;D}$.
By \eqref{ea6.23} we have
\begin{align*}
\bdabs{\calH[v]}^{(2s-r)}_{\beta;D}
&\;\le\;\bigl(\diam(D)\bigr)^{\nicefrac{\theta}{2}}
 \bdabs{\calH[v]}^{(2s-r-\nicefrac{\theta}{2})}_{\beta;D}\nonumber\\[5pt]
&\;\le\; \kappa_{1}\,\bigl(\diam(D)\bigr)^{\nicefrac{\theta}{2}}\,
 \dabs{v}^{(-r+\nicefrac{\theta}{2})}_{2s+\beta-\theta;D}\Bigr)\,.
\end{align*}
Therefore, since also $2s-r>1-r+\nicefrac{\theta}{2}$, we obtain
\begin{equation}\label{ea6.26}
\bdabs{\mathcal{T}[v]}^{(2s-r)}_{\beta;D}
\;\le\; \kappa_{3}\,\Bigl(\norm{f}_{C^{\beta}(\Bar{D})}
+ \dbrac{v}^{(-r+\nicefrac{\theta}{2})}_{1;D}
+\dabs{v}^{(-r+\nicefrac{\theta}{2})}_{2s+\beta-\theta;D}\Bigr)
\end{equation}
for some constant $\kappa_{3}$.
By the H\"older interpolation inequalities, for any
$\varepsilon>0$, there exists $\widetilde{C}(\varepsilon)>0$ such that
\begin{equation}\label{ea6.27}
\dbrac{v}^{(-r+\nicefrac{\theta}{2})}_{1;D} +
\dabs{v}^{(-r+\nicefrac{\theta}{2})}_{2s+\beta-\theta;D}\;\le\;
\widetilde{C}(\varepsilon)\,\dbrac{v}^{(-r+\nicefrac{\theta}{2})}_{0;D}
+\varepsilon\,\dabs{v}^{(-r+\nicefrac{\theta}{2})}_{2s+\beta;D}\,.
\end{equation}
Combining \eqref{ea6.24}, \eqref{ea6.25}, and \eqref{ea6.26}, and then
using \eqref{ea6.27} and the inequality
\begin{equation*}
\dbrac{v}^{(-r+\nicefrac{\theta}{2})}_{2s+\beta;D}\;\le\;
\bigl(\diam(D)\bigr)^{\nicefrac{\theta}{2}}
\dabs{v}^{(-r)}_{2s+\beta;D}
\end{equation*}
we obtain
\begin{equation}\label{ea6.28}
\dabs{u}^{(-r)}_{2s+\beta;D}\;\le\; \kappa_{4}(\varepsilon)\,
\Bigl(\norm{f}_{C^{\beta}(\Bar{D})}
+\dabs{v}^{(-r+\nicefrac{\theta}{2})}_{0;D}\Bigr)
+ \varepsilon\,\dabs{v}^{(-r)}_{2s+\beta;D}\,.
\end{equation}
In order to apply the  Leray--Schauder fixed point theorem, it suffices to show
that the set of solutions $u\in \mathscr{C}^{(-r)}_{2s+\beta}(D)$ of
$(-\Delta)^{s} u(x)  \;=\;\xi\,\mathcal{T}[u](x)$, for $\xi\in[0,1]$,
with $u=0$ on $D^{c}$, is bounded in
$\mathscr{C}^{(-r)}_{2s+\beta}(D)$.
However, from the above calculations, any such solution $u$ satisfies
\eqref{ea6.28} with $v\equiv u$.
Moreover by Lemma~\ref{L3.2},
\begin{equation}\label{ea6.29}
\sup_{x\in D}\; \abs{u(x)} \;\le\; \kappa_{5}\; \sup_{x\in D}\; \abs{f(x)}
\end{equation}
for some constant $\kappa_{5}$.
We also have that
\begin{align}\label{ea6.30}
\dabs{u}^{(-r+\nicefrac{\theta}{2})}_{0;D}&
\;\le\; \varepsilon^{-r+\nicefrac{\theta}{2}}\,
\sup_{x\in D,\,d_{x}\ge\varepsilon}\; \abs{u(x)}
+ \varepsilon^{\nicefrac{\theta}{2}}\,
\sup_{x\in D,\,d_{x}<\varepsilon}\; d_{x}^{-r}\,\abs{u(x)}\nonumber\\[5pt]
&\;\le\;\varepsilon^{-r+\nicefrac{\theta}{2}}\,\sup_{x\in D}\; \abs{u(x)}
+ \varepsilon^{\nicefrac{\theta}{2}}\,
\dabs{u}^{(-r)}_{0;D}\,.
\end{align}
Choosing $\varepsilon>0$ small enough, and using \eqref{ea6.29}--\eqref{ea6.30}
on the right hand side of \eqref{ea6.28} with $v\equiv u$, we obtain
\begin{equation}\label{ea6.31}
\dabs{u}^{(-r)}_{2s+\beta;D}\;\le\; \kappa_{6}\,
\norm{f}_{C^{\beta}(\Bar{D})}
\end{equation}
for some constant $\kappa_{6}$.
Hence by the Leray--Schauder fixed point theorem the map $v\mapsto u$
given by \eqref{ea6.22} has a fixed point
$u\in\mathscr{C}^{(-r)}_{2s+\beta}(D)$, i.e.,
\begin{equation*}
(-\Delta)^{s} u(x)=\mathcal{T}[u](x)\,.
\end{equation*}
Hence, this is a solution to \eqref{ea6.1}.
Uniqueness is obvious as $u$ is a classical solution.
The bound in \eqref{ea6.31} then applies and the proof is complete.
The proof in the case $\beta<s$ is completely analogous.
\end{proof}

Optimal regularity up to the boundary can be obtained under additional hypotheses.
The following result is a modest extension
of the results in \cite[Proposition~1.1]{RosOton-Serra}.

\begin{corollary}
Let $\I\in\mathfrak{I}_{2s}(\beta,\theta,\lambda)$ with $\theta>s$,
$f$ be locally H\"older continuous with exponent $\beta$, and
$D$ be a bounded domain with a $C^{2}$ boundary.
Suppose in addition that $b=0$ and that $k$ is symmetric, i.e.,
$k(x,z)=k(x,-z)$.
Then the solution of the Dirichlet problem in
\eqref{ea6.1} is in $C^{s}(\Rd)$.
Moreover, for any $\beta<s$ we have $u\in \mathscr{C}_{2s+\beta}^{(-s)}(D)$.
\end{corollary}

\begin{proof}
By Theorem~\ref{T6.1}, the Dirichlet problem in \eqref{ea6.1} has a
unique solution in $C^{2s+\rho}_{\mathrm{loc}}(D)\cap C(\Bar{D})$,
for any $\rho<\beta\wedge s$.
Moreover, for any $r<s$, we have the estimate
\begin{equation*}
\dabs{u}^{(-r)}_{2s+\rho;D}
\;\le\; C_{0} \, \norm{f}_{C^{\beta}(\Bar{D})}\,.
\end{equation*}
Fix  $r=2s-\theta$.
Then
\begin{equation*}
\int_{R<\abs{z}<1}
\abs{z}^{r}\,\frac{\Tilde{k}(x,z)}{\abs{z}^{d+2s}}\,
\D{z} \;=\; \int_{R<\abs{z}<1}
\abs{z}^{2s-\theta}\,\frac{\Tilde{k}(x,z)}{\abs{z}^{d+2s}}\,\D{z}
\;\le\;\lambda_{D}\,.
\end{equation*}
By \eqref{ea6.5} and the symmetry of the kernel, it follows that
\begin{equation*}
\babss{\int_{R<\abs{z}} \dd{u}(x;z)\,
\frac{\Tilde{k}(x,z)}{\abs{z}^{d+2s}}}\,
\D{z} \;\le\; \kappa_{1}
\Bigl(\dbrac{u}^{(-r)}_{r;D}
+ \norm{u}_{C(\Bar{D}}\Bigr)\qquad \forall x\in D\,,
\end{equation*}
for some constant $\kappa_{1}$.
Combining this with the estimate in Lemma~\ref{L6.1} we obtain
\begin{align*}
\bdbrac{\calH[u]}^{(0)}_{0;D}
\;\le\;M_{0}\,\dabs{u}^{(-r)}_{r;D}
\;<\;\infty\,,
\end{align*}
implying that  $\calH[u]\in L^{\infty}(D)$.
It then follows by \cite[Proposition~1.1]{RosOton-Serra}
that $u\in C^{s}(\Rd)$,
and that for some constant $C$ depending only on $s$, we have
\begin{align*}
\norm{u}_{C^{s}(\Rd)}
&\;\le\;
C\,\bnorm{\mathcal{T}[u]}_{L^{\infty}(D)}\\[5pt]
&\;\le\; C\,\lambda_{D}^{-1}\,c(d,2s)
\Bigl(\norm{f}_{L^{\infty}(D)}+\norm{\calH[u]}_{L^{\infty}(D)}\Bigr)\\[5pt]
&\;\le\;
C\,\lambda_{D}^{-1}\,c(d,2s)
\Bigl(\norm{f}_{L^{\infty}(D)}+M_{0}\,\dabs{u}^{(-r)}_{r;\,D}
\Bigr)\,.
\end{align*}
Using the H\"older interpolation inequalities
we obtain from the preceding estimate that
\begin{equation*}
\norm{u}_{C^{s}(\Rd)}
\;\le\; \Tilde{C}\,\norm{f}_{L^{\infty}(D)}
\end{equation*}
for some constant $\Tilde{C}$ depending only on $s$, $\theta$,
and $\lambda_{D}$.

Applying Lemma~\ref{L6.1} once more, we conclude that
$\calH[u]\in\mathscr{C}_{\beta'}^{(s)}(D)$ for any $\beta'\leq r$, and that
\begin{equation*}
\bdabs{\calH[u]}^{(s)}_{\beta';\,D}\;\le\;
M_{1}\, \dabs{u}^{(-r)}_{2s+\beta'-\theta ;\,D}\,.
\end{equation*}
Hence, applying \cite[Proposition~1.4]{RosOton-Serra}, we obtain
\begin{align*}
\dabs{u}^{(-s)}_{2s+\beta';\,D}
&\;\le\; C_{1}\Bigl(\norm{u}_{C^{s}(\Rd)}
+\bdabs{\mathcal{T}[u]}^{(s)}_{\beta';\,D}\Bigr)
\end{align*}
for some constant $C_{1}$,
and we can repeat this procedure to reach $u\in\mathscr{C}_{2s+\beta}^{(-s)}(D)$.
\end{proof}

\medskip

\noindent\textbf{Acknowledgments.}
We thank Dennis Kriventsov for helping us in clarifying some points of his
paper \cite{Kriventsov}
and for suggesting us the paper \cite{RosOton-Serra}.
We also thank H\'{e}ctor Chang Lara and Gonzalo D{\'a}vila for their help.
The work of Anup Biswas was supported in part by an award from the Simons
Foundation (No.~197982 to The University of Texas at Austin) and in part
by the Office of Naval Research
through the Electric Ship Research and Development Consortium.
The work of Ari Arapostathis was supported in part
by the Office of Naval Research
through the Electric Ship Research and Development Consortium.
This research of Luis Caffarelli is supported by an award from NSF.

\def\cprime{$'$}

\end{document}